\documentclass[reqno,11pt]{amsart}

\usepackage{color}
\usepackage{amssymb}
\usepackage{graphicx}
\usepackage[hidelinks]{hyperref}
\usepackage{enumerate}
\usepackage{latexsym,array}
\usepackage{amsfonts}
\usepackage{shadow}
\usepackage{amsbsy}
\usepackage{dsfont}
\usepackage{mathtools}
\usepackage{chngcntr}

\usepackage{doi}

\newcommand{\ran}{\operatorname{ ran}\,}

\newcommand{\id}{\mathcal{I}}

\newcommand{\sderiv}{\partial_S}

\newcommand{\Q}[3][\empty]{\mathcal{Q}_{#2}(#3)^{#1}}
\newcommand{\Qinv}[3][1]{\mathcal{Q}_{#2}(#3)^{-#1}}
\newcommand{\Qdiff}{\mathfrak{D}}

\newcommand{\lhol}{\mathcal{SH}_L}
\newcommand{\rhol}{\mathcal{SH}_R}
\newcommand{\intrin}{\mathcal{N}}

\newcommand{\dom}{\mathcal{D}}

\counterwithin{equation}{section}

\newtheorem{Pa}{Paper}[section]
\newtheorem{theorem}[Pa]{{\bf Theorem}}
\newtheorem{lemma}[Pa]{{\bf Lemma}}
\newtheorem{definition}[Pa]{{\bf Definition}}
\newtheorem{corollary}[Pa]{{\bf Corollary}}

\newtheorem{Pn}[Pa]{{\bf Proposition}}

\theoremstyle{definition}
\newtheorem{remark}[Pa]{{\bf Remark}}

\def\C{\mathbb C}
\def\hh{\mathbb{H}}

\def\H{\mathbb H}

\def\R{\mathbb R}
\def\N{\mathbb N}
\def\C{\mathbb C}

\def\Z{\mathbb Z}

\def\(s){\mathscr S(\R\times\R)}
\def\F{\mathcal F}

\newcommand{\rr}{\mathbb{R}}
\renewcommand{\S}{\mathbb{S}}

\newcommand{\boundOP}{\mathcal{B}}
\newcommand{\closOP}{\mathcal{K}}

\newcommand{\U}{\mathcal{U}}

\newcommand{\bcdot}{\,\cdot\,}

\renewcommand{\Re}{\mathrm{Re}}

\newcommand{\dist}{\mathrm{dist}}

\newcommand{\csector}[2]{\overline{\Sigma(#1,#2)}}
\newcommand{\osector}[2]{\Sigma(#1,#2)}

\title[
Fractional  powers of quaternionic operators ]
{Fractional  powers of quaternionic operators and Kato's formula using slice hyperholomorphicity } \oddsidemargin
0.2in \evensidemargin 0.2in \topmargin -0.5in \textwidth
15.5truecm \textheight 23truecm

\author[F. Colombo]{Fabrizio Colombo}
\address{(FC)
Politecnico di Milano\\Dipartimento di Matematica\\Via E. Bonardi, 9\\20133
Milano, Italy}
\email{fabrizio.colombo@polimi.it}

\author[J. Gantner]{Jonathan Gantner}
\address{(JG)
Politecnico di Milano\\Dipartimento di Matematica\\Via E. Bonardi, 9\\20133
Milano, Italy
} \email{jonathan.gantner@polimi.it}

\begin{document}
\maketitle

\begin{abstract}
In this paper we introduce fractional powers of quaternionic operators.
Their definition is based on the theory of slice-hyperholomorphic functions and on the $S$-resolvent operators of the quaternionic functional calculus.
The integral representation formulas of the fractional powers
and the quaternionic version of Kato's formula are  based on the notion of $S$-spectrum of a quaternionic operator.

The proofs of several properties of the fractional powers of quaternionic operators rely on the $S$-resolvent equation. This equation, which is very important and of independent interest, has already been introduced in the case of bounded quaternionic operators, but for the case of unbounded operators some additional considerations have to be taken into account. Moreover, we introduce a new series expansion for the pseudo-resolvent, which is  of independent interest and allows to investigate the behavior of the $S$-resolvents close to the $S$-spectrum.

The paper is addressed to researchers working  in operator theory and in complex analysis.
\end{abstract}
\vskip 1cm
\section{Introduction}

The theory of holomorphic functions has several applications in operator theory. For example it allows to define
groups and semigroups of linear operators that have applications in PDE and in other fields of mathematics and physics, see
\cite{ds,EngelNagel,Hille,Kantorovitz,Lunardi,Pazy}.

Quaternionic linear operators play a crucial role in quaternionic quantum mechanics because Schr\"odinger equation can be formulated just using
complex numbers or quaternions, see the fundamental paper \cite{BvN} on the logic of quantum mechanics by  Birkhoff and  von Neumann and  the subsequent papers  \cite{12,14,21}.
For the quaternionic formulation of quantum mechanics, we refer the reader to the book of Adler \cite{adler}.

In the definition of functions of quaternionic linear operators, the classical theory of holomorphic functions has to be replaced by the
recently developed theory of slice hyperholomorphic functions,
see the books \cite{ACSBOOK,MR2752913,GSSb}.
The notion of $S$-spectrum is the most fundamental concept in quaternionic operator theory: the slice hyperholomorphic functional calculus (called $S$-functional or quaternionic functional calculus), which is the quaternionic analogue of the Riesz-Dunford functional calculus, is based on this notion of spectrum, see \cite{ds,rudin} for the classical case and  the original papers \cite{acgs,JGA,CLOSED} and the book \cite{MR2752913} for the quaternionic case.

Thanks to the quaternionic functional calculus, it has been possible to develop the theory of quaternionic evolution operators, see
\cite{FUCGEN, perturbation,MR2803786,GR}, and the spectral theorem for quaternionic operators using the notion of $S$-spectrum \cite{ack,acks2,spectcomp}.
We point out that there are several attempts to prove the spectral theorem
over the quaternions in the literature, for instance \cite{14,sc,Viswanath}, but the notion of spectrum used is not clearly specified except for the paper \cite{fp}. In this paper the authors use the right spectrum $\sigma_R(M)$ of a normal quaternionic matrix $M$, which however turns out to be equal to the $S$-spectrum $\sigma_S(M)$.
Using the notion of $S$-spectrum, it is possible to define also the continuous functional calculus see \cite{GMP}, where the authors use the notion of slice hyperholomorphicity in the approach of \cite{GP}.

The theory of slice hyperholomorphic functions and the quaternionic functional calculus have allowed to generalize Schur analysis to this setting.
The literature on  classical Schur analysis is very large, we quote the books \cite{MR2002b:47144,adrs}, for an overview of the classical case,
and the first papers \cite{acs3,acs2,acs1} for the slice hyperholomorphic setting together with the book \cite{ACSBOOK} and the references therein for an overview of the existing literature.

Recently, we have proved the Taylor expansion in the operator for the $S$-functional calculus, see \cite{CGTAYLOR}.
In this paper we address the problem of defining  fractional powers of quaternionic operators using the Cauchy formula of slice hyperholomorphic functions.
To explain our results and the main differences with respect to the classical case, we recall some facts on the classical theory, see for example \cite{EngelNagel}.
\\
\\
The theory of fractional powers of linear operators has been developed by several authors. Without claiming completeness, we mention among them, for the early works, the papers \cite{Balakrishnan,Guzman3,Guzman1,Guzman2,Kato,Komatsu1,Komatsu2,Komatsu3,Watanabe,Yosida}.
The literature is now very wide and it has developed in several directions.
\\
\\
Let $A$ be a closed linear operator on a complex Banach space such that $(0,\infty)\subset \rho(A)$, where $\rho(A)$ is the resolvent set of $A$.
We denote by
$$
R(\lambda, A):=(\lambda I-A)^{-1}, \ \ \ \lambda\in \rho(A)
$$
the resolvent operator
and we assume that
$\|R(\lambda, A)\|\leq M/(1+\lambda)$ for some constant $M>0$ and all $\lambda\in (0,\infty)$. If this holds true then
there exists an open sector $\Sigma$ in $\mathbb{C}$ such that $\mathbb{R}^+\subset \Sigma \subset  \rho(A)$ and
$\|R(\lambda, A)\|\leq 2 M/(1+|\lambda|)$ all $\lambda\in \Sigma$.
\\
We consider a branch of the fractional power $\lambda \to \lambda^{-\alpha}$ for $\alpha >0$. This is a holomorphic function and  the resolvent operator  $\lambda\to R(\lambda, A)$ is also a holomorphic operator-valued function. We can therefore define
$$
A^{-\alpha}:=\frac{1}{2\pi i}\int_{\gamma} \lambda^{-\alpha} R(\lambda, A) d\lambda
$$
where $\gamma$ is a piecewise smooth path in $\Sigma\setminus \mathbb{R}^+$ that surrounds the spectrum of $A$, which is in general unbounded.
 Due to Cauchy's integral theorem, the definition of the fractional power $A^{\alpha}$ does not depend on the choice of $\gamma$ if the path does not intersect the spectrum of $A$.
Moreover some properties of the fractional powers depend on the resolvent equation
$$
R(\lambda, A)-R(\mu, A)=-(\lambda-\mu)R(\lambda, A)R(\mu, A),\ \ \ \lambda, \mu\in \rho(A).
$$
The idea of using the Cauchy formula for holomorphic functions to define functions of operators is the basis of the Riesz-Dunford functional calculus.
\\
\\
In the case of quaternionic operators (and also in the case of $n$-tuples of operators)
 slice hyperholomorphicity is among the most useful notions of holomorphicity to define functions of operators.
The skew-field of quaternions $\mathbb{H}$ can be written as $\mathbb{H}=\bigcup_{I\in \mathbb{S}}\mathbb{C}_I$, where
$\mathbb{S}$ consists of all quaternions $I$ such that $I^2=-1$ and $\mathbb{C}_I$ is the complex plane with imaginary unit $I$.
We say that a function $f:U\subset \mathbb{H} \to \mathbb{H}$ is (left) slice hyperholomorphic in $U$ if its restriction $f_I$ to the complex plane $\mathbb{C}_I$ is in the kernel of the Cauchy-Riemann operator
$$
\frac12\left( \frac{\partial}{\partial x_0} f_I(x) + I \frac{\partial}{\partial x_1} f_I(x)\right) = 0 \quad \text{for all }x = x_0 + I x_1\in U\cap\C_I \ \text{and all } I\in \mathbb{S}.
$$
In a similar way, one can define right slice hyperholomorphic functions.

The Cauchy formula for slice hyperholomorphic functions is the tool to define the quaternionic functional calculus or $S$-functional calculus.
There are however differences with respect to the classical case: the notion of spectrum of an operator is  for example not what one would expect by readapting the classical case. Instead, the $S$-spectrum of a bounded right linear operator on $T$ on a two sided quaternionic Banach space $V$ is defined as
$$
\sigma_S(T)=\{ s\in \mathbb{H}\ \ :\ \ T^2-2 \Re(s)T+|s|^2\mathcal{I}\ \ \
{\rm is\ not\  invertible}\},
$$
where $s=s_0+s_1i+s_2j+s_3k$ is a quaternion, $\Re(s)=s_0$ and $|s|^2=s_0^2+s_1^2+s_2^2+s_3^2$.
There are two resolvent operators associated with the quaternionic functional calculus, because the theory of slice hyperholomorphic functions contains different Cauchy kernels for left or right slice hyperholomorphic functions.
The left  and the right $S$-resolvent operators are defined as
\begin{equation}
S_L^{-1}(s,T):=-(T^2-2\Re(s) T+|s|^2\mathcal{I})^{-1}(T-\overline{s}\mathcal{I}),\ \ \ s \in  \mathbb{H}\setminus\sigma_S(T)
\end{equation}
and
\begin{equation}
S_R^{-1}(s,T):=-(T-\overline{s}\mathcal{I})(T^2-2\Re(s) T+|s|^2\mathcal{I})^{-1},\ \ \ s \in  \mathbb{H}\setminus\sigma_S(T),
\end{equation}
respectively.
Let $U\subset \hh$ be a suitable domain that contains the S-spectrum of $T$.
We define   the quaternionic functional calculus for  left slice hyperholomorphic functions $f:U \to \hh$ by
\begin{equation}\label{quatinteg311def}
f(T)={{1}\over{2\pi }} \int_{\partial (U\cap \mathbb{C}_I)} S_L^{-1} (s,T)\  ds_I \ f(s),
\end{equation}
where $ds_I=- ds I$,
and for right slice hyperholomorphic functions $f:U \to \hh$, by
\begin{equation}\label{quatinteg311rightdef}
f(T)={{1}\over{2\pi }} \int_{\partial (U\cap \mathbb{C}_I)} \  f(s)\ ds_I \ S_R^{-1} (s,T).
\end{equation}
As one may observe, the integrals in (\ref{quatinteg311def}) and (\ref{quatinteg311rightdef}) are computed
on the boundary of $U$ in the complex plane $\mathbb{C}_I$ and in principle the integral would depend on the imaginary unit $I$ chosen in $\mathbb{S}$.
Fortunately, this is not the case so that the quaternionic functional calculus turns out to be well defined.
This calculus can be extended to the case of unbounded operators under the condition that the function is slice hyperholomorphic  at infinity.

In the case of the quaternions, the function $s\to  s^{-\alpha}$ with $\alpha >0$ is both left and right slice hyperholomorphic on
$\mathbb{H} \setminus (-\infty,0]$, but not at infinity. So we cannot use the $S$-functional calculus, but  have to proceed directly with the definition
using the Cauchy formula in order to define $T^{-\alpha}$. For a right slice hyperholomorphic function, we have
\begin{equation}\label{talright}
T^{-\alpha}:= \frac{1}{2\pi}\int_{\Gamma} s^{-\alpha}\, ds_I\, S_R^{-1}(s,T),
\end{equation}
where $I\in\S$, and $\Gamma$ is a path in the complex plane $\mathbb{C}_I$
 that surrounds the intersection of the $S$-spectrum of the operator $T$ with the complex plane $\mathbb{C}_I$.
The same operator can be defined using the left S-resolvent operator
\begin{equation}\label{talleft}
T^{-\alpha}:= \frac{1}{2\pi}\int_{\Gamma} S_L^{-1}(s,T)\, ds_I\, s^{-\alpha}.
\end{equation}
Using the above definition we can give some canonical integral representations of the fractional powers depending on the location of the $S$-spectrum.
For example if $\sigma_S(T)\subset\{s \in\hh: \Re(s)>0\}$, under further assumptions, we have
\[
T^{-\alpha} = \frac{1}{\pi}\int_{0}^\infty \tau^{-\alpha} \left(\cos\left(\frac{\alpha\pi}{2}\right)T +\sin\left(\frac{\alpha\pi}{2}\right)\tau\id\right)(T^2 + \tau^2)^{-1}\, d\tau.
\]
This representation can be deduced from both formula (\ref{talright}) and (\ref{talleft}). We point out that  the functions, to which we can apply both versions of the S-functional calculus, are called intrinsic functions and they play an important role in the theory of slice hyperholomorphic functions and quaternionic operators.

The proofs of several properties of the fractional powers, for example the semigroup property, are based on the resolvent equation.
In the case of bounded operators we have shown in  \cite{acgs} that
\[
\begin{split}
S_R^{-1}(s,T)S_L^{-1}(p,T)&=[[S_R^{-1}(s,T)-S_L^{-1}(p,T)]p
\\
&
-\overline{s}[S_R^{-1}(s,T)-S_L^{-1}(p,T)]](p^2-2s_0p+|s|^2)^{-1},
\end{split}
\]
 for  $s$, $p \in  \mathbb{H}\setminus\sigma_S(T)$.

This equation holds also for unbounded operators, but one has to show that it is meaningful on the entire space $V$.
This has been verified in Section 2.
\\
\\
Moreover, in this paper we extend Kato's formula to the quaternionic setting.
 Kato considers in his paper \cite{Kato} linear operators in the Banach space $X$
that are not necessarily infinitesimal generators of semigroups.
 He considers the class of operators of type $(\omega, M)$
that are defined as follows:

\begin{enumerate}[(I)]
\item $A$ is densely defined and closed,
\item the resolvent set of $-A$ contains the open sector $|\arg \lambda|<\pi-\omega$, $\omega\in (0,\pi)$ and $\lambda (\lambda I+A)^{-1}$ is uniformly bounded in each small sector $|\arg \lambda|<\pi-\omega-\varepsilon$, for $\varepsilon>0$ and
$$
\lambda\|(\lambda I+A)^{-1}\|\leq M,\ \ \ \ \lambda>0.
$$
\end{enumerate}
Kato quotes in his references that similar operators are considered also
by M. A. Krasnosel'skii and P. E. Sobolevskii.

If $A$ is an operator of type $(\omega, M)$, then the fractional powers $A^\alpha$ for $\alpha\in (0,1)$ can be defined indirectly via
$$
 (\lambda I+A^\alpha)^{-1}=\frac{\sin(\pi\alpha)}{\pi}\int_0^\infty \frac{\mu^\alpha}{\lambda^2+2\lambda\mu^\alpha\cos(\pi\alpha)+\mu^{2\alpha}}(\mu I+A)^{-1}d\mu.
$$
The proof of Kato's formula is done in several steps.
We point out that the formula is obtained using a Cauchy integral representation: the second hand side of the formula is equal to
$$
J(\lambda):=\frac{1}{2\pi i}\int_{\gamma}(\lambda+z^\alpha)^{-1}(A-z)\,dz,
$$
where $\lambda>0$ and the path $\gamma$ lies in the resolvent set of $A$ and goes from $\infty e^{-i\theta}$ to $\infty e^{i\theta}$ with
$\omega<\theta<\pi$ avoiding the negative real axis and zero.
Operator-valued function $J(\lambda)$ satisfies the resolvent equation. The question is whether it can be expressed in the form
$$
J(\lambda)=(\lambda+A^\alpha)^{-1}.
$$
In other words one has to show that a single-valued function $J$ defined on a subset $E$
of the complex plane with values in the Banach algebra of all bounded linear operators $\mathcal{B}(X)$ that satisfies the equation
\begin{equation}\label{HHJUI}
J(\lambda)-J(\lambda')=-(\lambda-\lambda')J(\lambda)J(\lambda')=-(\lambda-\lambda')J(\lambda')J(\lambda),\ \ \ \ \forall \ \ \lambda,\ \lambda'\in E
\end{equation}
is the resolvent operator of a closed linear operator $A^\alpha$.

In the quaternionic case the $S$-resolvent equation involves both  resolvent operators,
 but despite this fact we are able to prove
Kato's formula in this setting. If $T$ is of type $(M,\omega)$, then we can define for $p$ with $\arg(p)>\phi_0>\max\{\omega, \alpha\pi\}$ the operator
$$
 F_{\alpha}(p,T) := \frac{\sin(\alpha\pi)}{\pi}\int_0^{+\infty}t^{\alpha}(p^2-2pt^{\alpha}\cos(\alpha\pi)+t^{2\alpha})^{-1} S_R^{-1}(-t,T)\,dt.
$$
The operator-valued left slice hyperholomorphic  function $p\to F_{\alpha}(p,T)$ does not satisfy the equation \eqref{HHJUI} in general, but only if $\lambda$ and $\mu$ are real. This is however sufficient in order to show the existence of an operator $B_\alpha$, with $S_R^{-1}(p,B_\alpha)=F_{\alpha}(p,T)$, which we define to be $T^\alpha$.

We point out that in order to show that $F_{\alpha}(p,T)$ satisfies \eqref{HHJUI} on the negative real line, we prove that it equals a Cauchy-type integral if $0\in\rho_S(T)$. In this case, we have
\[ F_{\alpha}(p,T) = \frac{1}{2\pi}\int_{\Gamma} S_R^{-1}(p,s^{\alpha})\,ds_I\,S_R^{-1}(s,T),\]
with
\[ S_R^{-1}(p,s^{\alpha}) = -(s^{\alpha}-\overline{p})(s^{2\alpha} - 2\Re(p)s^{\alpha}  + |p|^2)^{-1} \]
and  $\Gamma$ being any path that goes from $\infty e^{I\theta}$ to $\infty e^{-I\theta}$ with $\theta\in (\phi_0, \pi)$ in an arbitrary plane $\C_I$ avoiding the negative real axis and $0$. If $p$ is real, then $S_R^{-1}(p,s^{\alpha})$ is intrinsic and we can represent $F_{\alpha}(p,T)$ also using the left $S$-resolvent. Once more it is then the $S$-resolvent equation that allows us to show that $F_{\alpha}(p,T)$ satisfies \eqref{HHJUI} in this case.
\\
\\
{\it Outline of the paper}.
In Section \ref{PreRes} we introduce the necessary definitions and results of the theory of slice hyperholomorphic functions and the $S$-functional calculus.
The $S$-resolvent equation for unbounded operators and some results on operator-valued slice hyperholomorphic functions are proved here. Section~\ref{HolSect} contains preliminary results on the $S$-resolvents: we introduce a new series expansion for the pseudo-resolvent and use it to show that the $S$-resolvent is actually slice hyperholomorphic, which has always been used but never shown for unbounded operators. Then we show that the norms of the $S$-resolvents tends, in a certain sense, to infinity as one approaches the $S$-spectrum under suitable assumptions and that therefore there cannot exist any slice hyperholomorphic continuation of the $S$-resolvents, which is a fundamental fact for several proofs in this paper.
Section \ref{NagelSect} contains the definition of the fractional powers of a quaternionic operator and the proof that they are well defined. Moreover, we prove  the semigroup property of the fractional powers and some integral representations.
Section \ref{KatoSect} is dedicated to the extension of  Kato's formula to the quaternionic setting. Finally, the appendices contain several quite technical estimates needed in the proofs of Section~\ref{NagelSect}.

\section{Preliminary results}\label{PreRes}
The skew-field of quaternions consists of the real vector space
$$
\H:=\{\xi_0 + \sum_{i=1}^3\xi_ie_i: \xi_i\in\R\},
$$
 which is endowed with an associative product satisfying
\[e_1^2 = e_2^2 = e_3^2 = e_1e_2e_3 = - 1.\]
The real part of a quaternion $x = \xi_0 + \sum_{i=1}^3\xi_ie_i$ is defined as $\Re(x) := \xi_0$, its imaginary part as $\underline{x} := \sum_{i=1}^3\xi_ie_i$ and its conjugate as $\overline{x} := \Re(x) - \underline{x}$.

Each element of the set
\[\S := \{ x\in\H: \Re(x) = 0, |x| = 1 \}\]
 is a square-root of $-1$ and is therefore called an imaginary unit. For any $I\in\S$, the subspace $\C_I := \{x_0 + I x_1: x_1,x_2\in\R\}$ is an isomorphic copy of the field of complex numbers. If $I,J\in\S$ with $I\perp J$, set $K=IJ = -JI$. Then $1$, $I$, $J$ and $K$ form an
 orthonormal basis of $\H$ as a real vector space and $1$ and $J$ form a basis of $\H$ as a left or right vector space over the complex plane $\C_I$, that is
 \[ \H = \C_I + \C_I J�\quad\text{and}\quad \H = \C_I + J\C_I.\]
  Any quaternion $x$ belongs to such a complex plane: if we set
 \[I_x := \begin{cases}\underline{x}/|\underline{x}|,& \text{if  }\underline{x} \neq 0 \\ \text{any }I\in\S, \quad&\text{if }\underline{x}  = 0,\end{cases}\]
 then $x = x_0 + I_x x_1$ with $x_0 =\Re(x)$ and $x_1 = |\underline{x}|$. The set
 \[
 [x] := \{x_0 + Ix_1: I\in\S\},
 \]
is a 2-sphere, that reduces to a single point if $x$ is real.

\subsection{Slice hyperholomorphic functions} As pointed out in the introduction, quaternionic operator theory is based on the theory of slice hyperholomorphic functions, the most important results of which we introduce now. The proofs of the results stated in this subsection can be found in the book \cite{MR2752913}.
\begin{definition}\label{AAAKKKE}
Let $U\subset\H$ be open and let $f: U \to \H$ be real differentiable. For any $I\in\S$, let $f_I := f|_{U\cap\C_I}$ denote the restriction of $f$ to the plane $\C_I$. The function $f$ is called left slice hyperholomorphic if, for all $I\in\S$,
\begin{equation}\label{LHolEQ}
\frac12\left( \frac{\partial}{\partial x_0} f_I(x) + I \frac{\partial}{\partial x_1} f_I(x)\right) = 0 \quad \text{for all }x = x_0 + I x_1\in U\cap\C_I\end{equation}
and right slice hyperholomorphic if, for all $I\in\S$,
\begin{equation}\label{RHolEQ}\frac12\left( \frac{\partial}{\partial x_0} f_I(x) + \frac{\partial}{\partial x_1} f_I(x)I\right) = 0 \quad \text{for all }x = x_0 + I x_1\in U\cap\C_I.\end{equation}
A left slice hyperholomorphic function that satisfies $f(U\cap\C_I)\subset\C_I$ for all $I\in\S$ is called intrinsic.

We denote the set of all left slice hyperholomorphic functions on $U$ by $\lhol(U)$, the set of all right slice hyperholomorphic functions on $U$ by $\rhol(U)$ and the set of all intrinsic functions by $\intrin(U)$.
\end{definition}
Note that an intrinsic function is both left and right slice hyperholomorphic because $f_I(x)\in\C_I$ commutes with the imaginary unit $I$ in the respective Cauchy-Riemann-operator. The converse is not true: the constant function $x\mapsto b\in\H\setminus\R$ is left and right slice hyperholomorphic, but it is not intrinsic.

The importance of the class of intrinsic functions is due to the fact that the multiplication and composition with intrinsic functions preserve slice hyperholomorphy. This is not true for arbitrary slice hyperholomorphic functions.
\begin{corollary}\label{lkjsWW}
 If $f\in\intrin(U)$ and $g\in\lhol(U)$, then $fg\in\lhol(U)$. If $f\in\rhol(U)$ and $g\in\intrin(U)$, then $fg\in\rhol(U)$.

 If $g\in\intrin(U)$ and $f\in\lhol(g(U))$, then $f\circ g\in \lhol(U)$. If $g\in\intrin(U)$ and $f\in\rhol(g(U))$, then $f\circ g\in \rhol(U)$.
\end{corollary}

Important examples of slice hyperholomorphic functions are power series with quaternionic coefficients: series of the form $\sum_{n=0}^{+\infty}x^na_n$ are left slice hyperholomorphic and series of the form $\sum_{n=0}^{\infty} a_nx^n$ are right slice hyperholomorphic on their domains of convergence. A power series is intrinsic if and only if its coefficients are real.

Any slice hyperholomorphic function on the other hand can be expanded into a power series at any real point.
\begin{definition}
The slice-derivative of a function $f\in\lhol(U)$ is defined as
\[\sderiv  f(x) = \lim_{\C_{I_x}\ni s\to x} (s-x)^{-1}(f(s)-f(x))\qquad\text{for }x = x_0 + I_xx_1\in U, \]
where $\lim_{\C_{I_x}\ni s\to x} g(s)$ denotes the limit as $s$ tends to $x$ in $\C_{I_x}$.
The slice-derivative of a function $f\in\rhol(U)$ is defined as
\[\sderiv  f(x) = \lim_{\C_{I_x}\ni s\to x} (f(s)-f(x))(s-x)^{-1}\qquad\text{for }x = x_0 + I_xx_1\in U. \]
\end{definition}
\begin{corollary}\label{SDProp}
The slice derivative of a left (or right) slice hyperholomorphic function is again left (or right) slice hyperholomorphic. Moreover, it coincides with the derivative with respect to the real part, that is
\[\sderiv  f(x) = \frac{\partial}{\partial x_0} f(x)\quad\text{for } x = x_0 + I x_1. \]
\end{corollary}
\begin{theorem}\label{Taylor}
If $f$ is left slice hyperholomorphic on the ball $B(r,\alpha)$ with radius $r$ centered at $\alpha\in\R$, then
\[f(x) = \sum_{n=0}^{+\infty} (x-\alpha)^n \frac{1}{n!}\sderiv ^n f(\alpha)\quad\text{for }x\in B(r,\alpha).\]
If $f$ is right slice hyperholomorphic on $B(r,\alpha)$, then
\[f(x) = \sum_{n=0}^{+\infty}\frac{1}{n!}\sderiv ^n f(\alpha) (x-\alpha)^n \quad\text{for }x\in B(r,\alpha).\]
\end{theorem}

Slice hyperholomorphic functions possess good properties when they are defined on suitable domains.
\pagebreak[2]
\begin{definition}
A set $U\subset\H$ is called
\begin{enumerate}[(i)]
\item axially symmetric if $[x]\subset U$ for any $x\in U$ and
\item a slice domain if $U$ is open, $U\cap\R\neq 0$ and $U\cap\C_I$ is a domain for any $I\in\S$.
\end{enumerate}
\end{definition}

\begin{theorem}[Identity Principle]\label{IDP}
Let $U$ be a  slice domain, let $f$ be left or right slice hyperholomorphic on $U$ and let $\mathcal{Z}$ be  the set of zeros of $f$. If there exists an imaginary unit $I\in\S$ such that $\mathcal{Z}\cap\C_I$ has an accumulation point in $U\cap\C_I$, then $f\equiv 0$.

\end{theorem}

As a consequence of the Identity Principle,  the values of a slice hyperholomorphic function on an axially symmetric slice domain are uniquely determined by its values on an arbitrary complex plane $\C_I$. Therefore, any function that is holomorphic on a suitable subset of a complex plane possesses an unique slice hyperholomorphic extension.
\begin{theorem}[Representation Formula]\label{RepFo}
Let $U$ be an axially symmetric slice domain and let $I\in\S$. For any $x = x_0 + I_x x_1\in U$ set $x_I := x_0 + Ix_1$. If $f\in\lhol(U)$, then
\[f(x) = \frac12(1-I_xI)f(x_I) + \frac12(1+I_xI)f(\overline{x_I}) \quad\text{for all }x\in U.\]
If $f\in\rhol(U)$, then
\[f(x) = f(x_I)(1-II_x)\frac12 + f(\overline{x_I})(1+II_x)\frac12 \quad\text{for all }x\in U.\]
\end{theorem}

\begin{corollary}\label{extLem}
Let $I\in\S$ and let $f:O\to\H$ be real differentiable, where $O$ is a domain in $\C_I$ that is symmetric with respect to the real axis.
\begin{enumerate}[(i)]
\item The axially symmetric hull $[O]: = \bigcup_{z\in O}[z]$ of $O$ is an axially symmetric slice domain.
\item If $f$ satisfies \eqref{LHolEQ}, then there exists a unique left slice hyperholomorphic extension of $f$ to $[O]$.
\item If $f$ satisfies \eqref{RHolEQ}, then there exists a unique right slice hyperholomorphic extension of $f$ to $[O]$.
\end{enumerate}
\end{corollary}
\begin{remark}
If $f$ has a left and a right slice hyperholomorphic extension, they do not necessarily coincide. Consider for instance the function $z\mapsto bz$ on $\C_I$ with a constant $b\in\C_I\setminus\R$. Its left slice hyperholomorphic extension to $\H$ is $x\mapsto xb$, but its right
slice hyperholomorphic extension is $x\mapsto bx$.
\end{remark}
Finally, slice hy\-per\-ho\-lo\-mor\-phic functions satisfy an adapted version of Cauchy's integral theorem and a Cauchy-type integral formula with a modified kernel, which is the starting point for the definition of the $S$-functional calculus.
\begin{definition}
We define the  left slice hyperholomorphic Cauchy kernel as
\[S_L^{-1}(s,x) = -(x^2-2\Re(s)x + |s|^2)^{-1}(x-\overline{s})\quad\text{for }x\notin[s]\]
and the right slice hyperholomorphic Cauchy kernel as
\[S_R^{-1}(s,x) = -(x-\overline{s})(x^2-2\Re(s)x + |s|^2)^{-1}\quad\text{for }x\notin[s].\]
\end{definition}
\begin{corollary}
The left slice hyperholomorphic Cauchy-kernel $S_L^{-1}(s,x)$ is left slice hyperholomorphic in the variable $x$ and right slice hy\-per\-ho\-lo\-mor\-phic in the variable $s$ on its domain of definition. Moreover, we have $S_R^{-1}(s,x) = - S_L^{-1}(x,s)$.
\end{corollary}
\begin{remark}
If $x$ and $s$ belong to the same complex plane, they commute and the slice hyperholomorphic Cauchy-kernels reduce to the classical one:
\[ \frac{1}{s-x} = S_L^{-1}(s,x) = S_R^{-1}(s,x).\]
\end{remark}
\begin{theorem}[Cauchy's integral theorem]\label{CInt}
Let $O\subset\H$ be open, let $I\in\S$ and let $D_I$ be a bounded open subset of $O\cap\C_I$ with $\overline{D_I}\subset O\cap\C_I$ such that its boundary consists of a finite number of continuously differentiable Jordan curves. For any $f\in\rhol(U)$ and $g\in\lhol(U)$, it is
\[\int_{\partial D_I}f(s)\,ds_I\,g(s) = 0,\]
where $ds_I = -I\ ds$.
\end{theorem}
\begin{theorem}[Cauchy's integral formula]\label{Cauchy}
Let $U\subset\H$ be a bounded axially symmetric slice domain such that its boundary $\partial (U\cap\C_I)$ in $\C_I$ consists of a finite number of continuously differentiable Jordan curves. Let $I\in\S$ and set $ds_I = -I\, ds$. If $f$ is left slice hyperholomorphic on an open set that contains $\overline{U}$, then
\[f(x) = \frac{1}{2\pi}\int_{\partial(U\cap\C_I)} S_L^{-1}(s,x)\,ds_I\, f(s)\quad\text{for all }x\in U.\]
If $f$ is right slice hyperholomorphic on an open set that contains $\overline{U}$, then
\[f(x) = \frac{1}{2\pi}\int_{\partial(U\cap\C_I)}f(s)\, ds_I\, S_R^{-1}(s,x)\quad\text{for all }x\in U.\]
\end{theorem}

The pointwise product of two slice hyperholomorphic functions is in general not slice hyper\-holo\-morphic. However, it is possible to define regularized products that preserve left and right slice hyperholomorphicity. The left slice hyperholomorphic Cauchy-kernel $S_L^{-1}(s,x)$ is the inverse of the function $x\mapsto s-x$ with respect to this left slice hyperholomorphic product. Similarly, $S_R^{-1}(s,x)$ the inverse of the function $x\mapsto s-x$ with respect to the right slice hyperholomorphic product. We therefore define $S_L^{-n}(s,x)$ and $S_R^{-n}(s,x)$ as the $n$-th inverse power of the function $x\mapsto s-x$ with respect to the left resp. right slice hyperholomorphic product.

\begin{definition}\label{CauchySlicePow}
Let $s,x\in\H$ with $s\notin [x]$. For $n\in\N_0$, we define
\[S_L^{-n}(s,x) := (x^2-2\Re(s)x + |s|^2)^{-n}\sum_{k=0}^n\binom{n}{k}(-x)^k\overline{s}^{n-k}\]
and
\[S_R^{-n}(s,x) := \sum_{k=0}^{n}\binom{n}{k}\overline{s}^{n-k}(-x)^k(x^2-2\Re(s)x + |s|^2)^{-n}.\]
\end{definition}

A theory of slice hyperholomorphicity can also be developed for functions with values a two-sided Banach space over the quaternions. First results were stated in \cite{OPValPaper}.
\begin{definition}\label{OpValDef}
Let $V$ be a two-sided quaternionic Banach space, let $U\subset\H$ be open and let $f: U \to V$ be real differentiable. For any $I\in\S$, let $f_I := f|_{U\cap\C_I}$ denote the restriction of $f$ to the plane $\C_I$. The function $f$ is called left slice hyperholomorphic if, for any $I\in\S$,
\begin{equation}\label{sslekeio}
\frac12\left( \frac{\partial}{\partial x_0} f_I(x) + I \frac{\partial}{\partial x_1} f_I(x)\right) = 0 \quad \text{for all }x = x_0 + I x_1\in U\cap\C_I\end{equation}
and right slice hyperholomorphic if, for any $I\in\S$,
\begin{equation}\label{sslekeio2}
\frac12\left( \frac{\partial}{\partial x_0} f_I(x) + \frac{\partial}{\partial x_1} f_I(x)I\right) = 0 \quad \text{for all }x = x_0 + I x_1\in U\cap\C_I.\end{equation}
We denote the set of all left slice hyperholomorphic functions on $U$ by $\lhol(U,V)$, the set of all right slice hyperholomorphic functions on $U$ by $\rhol(U,V)$.
\end{definition}
\begin{remark}
The paper \cite{OPValPaper} actually starts from a different definition. Therein, an operator-valued function $f$ is called strongly left slice hyperholomorphic if it admits a left slice derivative, that is
\begin{equation}\label{SSDDJJU}
 \lim_{\C_I\ni s\to x} (s-x)^{-1}(f_I(x)-f_I(s))
 \end{equation}
exists in the topology of $V$ for any $s \in U\cap\C_I$ and any $I\in\S$ and the limits coincide for all $I\in\S$ if $s$ is real .

The function $f$ is called weakly left slice hyperholomorphic  if, for any continuous left linear functional $\varLambda$ in the (left) dual space $V^*$ of $V$, the function $\varLambda f$ admits a left slice derivative that is
\[ \lim_{\C_I\ni s\to x} (s-x)^{-1}(\varLambda f_I(x)-\varLambda f_I(s))\]
exists in $\H$ for any $s \in U\cap\C_I$ and any $I\in\S$ and the limits coincide of all $I\in\S$  if $s$ is real. This is equivalent to $\varLambda f$ being left slice hyperholomorphic in the sense of Definition~\ref{AAAKKKE} by \cite[Proposition~3.2]{OPValPaper} if $\varLambda f$ is real differentiable. By Theorem~3.6 in \cite{OPValPaper} the notions of strong and weak left slice hyperholomorphicity are equivalent and by Proposition~3.9 they are both equivalent to the notion of left slice hyperholomorphicity defined in Definition~\ref{OpValDef} as long as $f$ is assumed to be real differentiable.

Analogue considerations can be done for right slice hyperholomorphic functions. In this case, the right slice derivative defined by
\begin{equation}\label{SSDDJJUr}
 \lim_{\C_I\ni s\to x} ( f_I(x)- f_I(s))(s-x)^{-1}
\end{equation}
replaces the notion of left slice derivative and the notion of weak right slice hyperholomorphicity must  be defined using right linear functionals in the right dual space of $V$ instead of left linear functionals.
\end{remark}
Many results for scalar-valued slice hyperholomorphic functions also hold true in the operator valued case. The proofs of the following results can be found in \cite{OPValPaper}.

\begin{theorem}[Identity Principle, {\cite[Proposition~3.11]{OPValPaper}}]\label{IDPOP}
Let $U$ be a slice domain, let $f$ be left or right slice hyperholomorphic on $U$ with values in a two-sided quaternionic Banach space and let $\mathcal{Z}$ be  the set of zeros of $f$. If there exists an imaginary unit $I\in\S$ such that $\mathcal{Z}\cap\C_I$ has an accumulation point in $U\cap\C_I$, then $f\equiv 0$.
\end{theorem}
\begin{theorem}[Representation Formula, {\cite[Theorem~3.15]{OPValPaper}}]\label{RepFoOP}
Let $V$ be a two-sided quaternionic Banach space, let $U$ be an axially symmetric slice domain and let $I\in\S$. For any $x = x_0 + I_x x_1\in U$ set $x_I := x_0 + Ix_1$. If $f\in\lhol(U,V)$, then
\[f(x) = \frac12(1-I_xI)f(x_I) + \frac12(1+I_xI)f(\overline{x_I}) \quad\text{for all }x\in U.\]
If $f\in\rhol(U,V)$, then
\[f(x) = f(x_I)(1-II_x)\frac12 + f(\overline{x_I})(1+II_x)\frac12 \quad\text{for all }x\in U.\]
\end{theorem}
An immediate consequence of this Theorem is the following extension result, which can be shown as the one in the scalar case.
\begin{corollary}\label{extLemOP}
Let $V$ be a two-sided quaternionic Banach space, let $I\in\S$ and let $f:O\to V$ be real differentiable, where $O$ is a domain in $\C_I$ that is symmetric with respect to the real axis.
\begin{enumerate}[(i)]
\item If $f$ satisfies \eqref{sslekeio}, then there exists a unique left slice hyperholomorphic extension of $f$ to $[O]$.
\item If $f$ satisfies \eqref{sslekeio2}, then there exists a unique right slice hyperholomorphic extension of $f$ to $[O]$.
\end{enumerate}
\end{corollary}

\begin{theorem}[Cauchy's integral formula, {\cite[Theorem~3.13]{OPValPaper}}]\label{Cauchy}
Let $U\subset\H$ be a bounded axially symmetric slice domain such that its boundary $\partial (U\cap\C_I)$ in $\C_I$ consists of a finite number of continuously differentiable Jordan curves. Let $I\in\S$ and set $ds_I = -I\, ds$. If $f$ is left slice hyperholomorphic on an open set that contains $\overline{U}$, then
\[f(x) = \frac{1}{2\pi}\int_{\partial(U\cap\C_I)} S_L^{-1}(s,x)\,ds_I\, f(s)\quad\text{for all }x\in U.\]
If $f$ is right slice hyperholomorphic on an open set that contains $\overline{U}$, then
\[f(x) = \frac{1}{2\pi}\int_{\partial(U\cap\C_I)}f(s)\, ds_I\, S_R^{-1}(s,x)\quad\text{for all }x\in U.\]
\end{theorem}

To the best of the authors' knowledge, the proves of several fundamental results have not yet been given explicitly for the case of vector-valued slice hyperholomorphic functions. Thus, we shall give them for the sake of completeness.
\begin{corollary}\label{SDPropOP} Let $V$ be a two-sided quaternionic Banach space.
The slice derivative of a left (or right) slice hyperholomorphic function with values in $V$ defined by \eqref{SSDDJJU} resp. \eqref{SSDDJJUr} is again left (or right) slice hyperholomorphic. Moreover, it coincides with the derivative with respect to the real part, that is
\[\sderiv  f(x) = \frac{\partial}{\partial x_0} f(x)\quad\text{for } x = x_0 + I x_1. \]
\end{corollary}
\begin{proof}Assume that $f\in\lhol(U,V)$, choose $I\in\S$ and consider $f_I = f_{U\cap\C_I}$. The quaternionic Banach space $V$ also carries the structure of a complex Banach space over the complex field $\C_I$, which we obtain by restricting the multiplication with quaternionic scalars on the left to $\C_I$. The function $f_I$ is then a function with values in this complex Banach space that is holomorphic in the classical sense. Its derivative coincides with the slice derivative of $f$ on $\C_I$, i.e.
\[ (\sderiv f )|_{U\cap\C_I} = f_I' = \frac{\partial}{\partial x_0} f_I,\]
where $f_I'$ denotes the usual derivative of $f_I$, when it is considered a holomorphic function with values in a complex Banach space over $\C_I$. In particular this function is again holomorphic in the classical sense and thus satisfies \eqref{sslekeio}. The statement for right slice hyperholomorphic functions follows with analogous arguments.

\end{proof}
\begin{theorem}\label{TaylorOP}Let $V$ be a two-sided quaternionic Banach space and assume that the function $f$ takes values in $V$.
If $f$ is left slice hyperholomorphic on the ball $B(r,\alpha)$ with radius $r$ centered at $\alpha\in\R$, then
\[f(x) = \sum_{n=0}^{+\infty} (x-\alpha)^n \frac{1}{n!}\sderiv ^n f(\alpha)\quad\text{for }x\in B(r,\alpha).\]
If $f$ is right slice hyperholomorphic on $B(r,\alpha)$, then
\[f(x) = \sum_{n=0}^{+\infty}\frac{1}{n!}\sderiv ^n f(\alpha) (x-\alpha)^n \quad\text{for }x\in B(r,\alpha).\]
\end{theorem}
\begin{proof}
Assume that $f$ is left slice hyperholomorphic on $B(r,\alpha)$ and consider an imaginary unit $I\in\S$. As in the proof of Corollary~\ref{SDPropOP} we may consider $V$ as a Banach space over $\C_{I}$ by restricting the scalar multiplication with quaternions on the left to $\C_{I}$. The restriction $f_{I}$ of $f$ to the complex plane $\C_{I}$ is then a holomorphic function with values in a complex Banach space and thus admits a power series expansion at $\alpha$ that converges on $B(r,\alpha)\cap\C_I$. For $x\in B(r,\alpha)\cap\C_I$, we obtain
\[
f(x) = f_{I}(x) = \sum_{n=0}^{+\infty}(x-\alpha)^n \frac{1}{n!} f_{I}^{(n)}(\alpha) =  \sum_{n=0}^{+\infty}(x-\alpha)^n \frac{1}{n!} \frac{\partial^n}{\partial x_0^n}f_{I}(\alpha).
\]
By Corollary~\ref{SDPropOP}, we have $\frac{\partial^n}{\partial x_0^n}f_{I_x}(\alpha)= \frac{\partial^n}{\partial x_0^n}f(\alpha) = \sderiv^n f(\alpha)$. Thus the coefficients are independent of the plane $\C_{I}$ and the statement holds true. The case of right slice hyperholomorphic functions can be shown by analogous arguments.

\end{proof}

Finally, we also give the proof of Cauchy's integral theorem for vector-valued slice hyperholomorphic functions. It is based on the quaternionic version of the Hahn-Banach-theorem, which was originally proved in \cite{Souc} and a  proof of which can also be found in \cite{MR2752913}. We need the following corollary.

\begin{corollary}[{\cite[Corollary~4.10.2]{MR2752913}}]\label{HBCor}
Let $V$ be a two-sided quaternionic Banach space. If $\varLambda(v) = 0$ for all continuous right linear functionals $\varLambda: V\to\H$, then $v = 0$.

Similarly,  if $\varLambda(v) = 0$ for all continuous left linear functionals $\varLambda: V\to\H$, then also $v = 0$.
\end{corollary}

\begin{theorem}[Cauchy's integral theorem]\label{COP}
Let $V$ be a two-sided quaternionic Banach space, let $O\subset\H$ be open and let $I\in\S$. Furthermore assume that $D_I$ is a bounded open subset of $O\cap\C_I$ with $\overline{D_I}\subset O\cap\C_I$, whose boundary consists of a finite number of continuously differentiable Jordan curves. If $f\in\rhol(U,V)$ and $g\in\lhol(U,\H)$ or if $f\in\rhol(U,\H)$ and $g\in\lhol(U,V)$, then
\[\int_{\partial D_I}f(s)\,ds_I\,g(s) = 0,\]
where $ds_I = -I\, ds$.
\end{theorem}
\begin{proof}
Assume that $f\in\rhol(U,V)$ and $g\in\lhol(U,\H)$ and consider a continuous right linear functional $\varLambda: V\to\H$. The function $s\mapsto\varLambda f(s) = \varLambda(f(s))$ is a quaternion-valued right slice hyperholomorphic function. Thus we deduce from Theorem~\ref{CInt} that
\[ \varLambda\left(\int_{\partial D_I}f(s)\,ds_I\,g(s)\right) = \int_{\partial D_I}\varLambda f(s)\,ds_I\,g(s) = 0\]
and from Corollary~\ref{HBCor} in turn that the statement holds true. The other case follows with analogous arguments.

\end{proof}

\subsection{The S-functional calculus}
The natural extension of the Riesz-Dunford-functional calculus for complex linear operators  to quaternionic linear operators is the so-called $S$-functional calculus. It is based on the theory of slice hyperholomorphic functions and follows the principal idea of the classical case: to formally replace the scalar variable $x$ in the Cauchy formula by an operator. The proofs of the results stated in this subsection can be found in \cite{acgs,MR2752913}.

Let $V$ be a two-sided quaternionic Banach space. We denote the set of all bounded quaternionic right-linear operators on $V$ by $\boundOP(V)$ and the set of all closed and densely defined quaternionic right-linear operators on $V$ by $\closOP(V)$.
\begin{definition}
We define the $S$-resolvent set of an operator $T\in\closOP(V)$ as
\[\rho_S(T):= \{ s\in\H: (T^2 - 2\Re(s)T + |s|^2\id)^{-1}\in\boundOP(V)\}\]
and the $S$-spectrum of $T$ as
\[\sigma_S(T):=\H\setminus\rho_S(T).\]
\end{definition}
For $s\in\H$ and $T\in\closOP(V)$, we set
\[ \Q{s}{T} := T^2 - 2\Re(s)T + |s|^2\id.\]
If $s\in\rho_S(T)$, then the operator
\[
\Qinv{s}{T}=(T^2-2\Re(s)T +|s|^2\id )^{-1}
\]
is called the pseudo-resolvent of $T$ at $s$. We point out that, in contrast to the notation we use in this paper, in the literature it is often the pseudo-resolvent that is denoted by the symbol $\Q{s}{T}$.
\begin{definition}
Let $T\in\closOP(V)$. The left $S$-resolvent operator is defined as
\begin{equation}\label{SresolvoperatorL}
S_L^{-1}(s,T):= \Qinv{s}{T}\overline{s} -T\Qinv{s}{T}
\end{equation}
and the right $S$-resolvent operator is defined as
\begin{equation}\label{SresolvoperatorR}
S_R^{-1}(s,T):=-(T-\id \overline{s})\Qinv{s}{T}.
\end{equation}
\end{definition}
\begin{remark}\label{RkResExtension} Observe that  one obtains the right $S$-resolvent operator by formally replacing the variable $x$ in the right slice hyperholomorphic Cauchy kernel by the operator $T$. The same procedure yields
\begin{equation}\label{LResShort}
S_L^{-1}(s,T)v = -\Qinv{s}{T}(T-\overline{s}\id)v,\quad\text{for }v\in\dom(T)
\end{equation}
for the left $S$-resolvent operator. This operator is not defined on the entire space $V$, but only on  the domain $\dom(T)$ of $T$. One can exploit the fact that $ \Qinv{s}{T}$ and $T$ commute on $\dom(T)$ in order to overcome this problem: commuting $T$ and $\Qinv{s}{T}$ in \eqref{LResShort} yields \eqref{SresolvoperatorL}. For arbitrary $s\in\H$, the operator $T^2 - 2\Re(s)T + |s|^2\id$ maps $\dom(T^2)$ to $V$. Hence, the pseudo-resolvent $\Qinv{s}{T}$ maps $V$ to $\dom(T^2)\subset \dom(T)$ if $s\in\rho_S(T)$. Since $T$ is closed and $\Qinv{s}{T}$ is bounded, equation \eqref{SresolvoperatorL} then defines a continuous and therefore bounded right linear operator on the entire space $V$. Hence, the left resolvent $S_L^{-1}(s,T)$ is the natural extension of the operator \eqref{LResShort} to~$V$. In particular, if $T$ is bounded, then $S_L^{-1}(s,T)$ can be defined directly by~\eqref{LResShort}.

If one considers left linear operators, then one must modify the definition
of the right $S$-resolvent operator for the same reasons.
\end{remark}
\begin{remark}
The $S$-resolvent operators reduce to the classical resolvent if $T$ and $s$ commute, that is
\[S_L^{-1}(s,T) = S_R^{-1}(s,T) = (s\id - T)^{-1}.\]
This is in particular the case if $s$ is real.
\end{remark}
As pointed out in the introduction, the $S$-spectrum is the proper generalization of the notion of right-eigenvalues \cite[Theorem2.5]{SigSSigR}.
\begin{theorem}\label{SvsR}
Let $T\in\closOP(V)$. Then $s\in\H$ is a right eigenvalue if and only if it is an $S$-eigenvalue.
\end{theorem}
The following important result has implicitly been assumed to hold true in the literature. For the case of bounded operators a proof can be found in \cite{MasterThesis}, but by the best of the authors' knowledge, it has never been shown for unbounded operators. (The paper \cite{GR} contains a proof in a more general setting: it considers real alternative $*$-algebras instead of quaternions. However, the proof in this paper requires that there exits a real point in the $S$-resolvent set of the operator.) For the sake of completeness, we therefore give its proof in this paper. Since the arguments for unbounded operators are quite technical, we postpone them to  Section~\ref{HolSect}.

\begin{lemma}\label{ResHol2342}
Let $T\in\closOP(V)$. The map $s\mapsto S_L^{-1}(s,T)$ is a right slice hyperholomorphic function on $\rho_S(T)$ with values in the two-sided  quaternionic Banach space $\boundOP(V)$. The map $s\mapsto S_R^{-1}(s,T)$ is a left slice hyperholomorphic function on $\rho_S(T)$ with values in the two-sided quaternionic Banach space $\boundOP(V)$.
\end{lemma}

The $S$-resolvent equation has been proved in \cite{acgs} for the case that $T$ is a bounded operator. For the sake of completeness we show the $S$-resolvent equation for the case of unbounded operators.

\begin{theorem}[$S$-resolvent equation]Let $T\in\closOP(V)$. If  $s,p \in  \rho_S(T)$ with $s\notin[p]$, then
\begin{equation}\label{resEQ}
\begin{split}
S_R^{-1}(s,T)S_L^{-1}(p,T)v=&\big[[S_R^{-1}(s,T)-S_L^{-1}(p,T)]p
\\
&
-\overline{s}[S_R^{-1}(s,T)-S_L^{-1}(p,T)]\big](p^2-2s_0p+|s|^2)^{-1}v, \ \ \ v\in V.
\end{split}
\end{equation}
\end{theorem}

\begin{proof}
We recall from  \cite{MR2752913} that the left $S$-resolvent operator satisfies the equation
\begin{equation}\label{LresEQ}
TS_L^{-1}(p,T)v=S_L^{-1}(p,T)pv- v,\ \ \ \ v\in V
\end{equation}
and that the right $S$-resolvent operator satisfies the equation
\begin{equation} \label{RresEQ}
S_R^{-1}(s,T)Tv = sS_R^{-1}(s,T)v - v,\ \ \ \ v\in\dom(T).
\end{equation}
As in the case of bounded operators, the $S$-resolvent equation is deduced from these two relations. However, we have to pay attention to being consistent with the domains of definition of every operator that appears in the proof.

We show that, for every $v\in V$, one has
\begin{multline}
\label{RESeqStart}
S_R^{-1}(s,T)S_L^{-1}(p,T)(p^2-2s_0p+|s|^2)v=\\
[S_R^{-1}(s,T)-S_L^{-1}(p,T)]pv-\overline{s}[S_R^{-1}(s,T)-S_L^{-1}(p,T)]v.
\end{multline}
We then obtain the original equation \eqref{resEQ} by replacing $v$ by $(p^2-2s_0p+|s|^2)^{-1}v$.

For $w\in V$, the left $S$-resolvent equation \eqref{LresEQ} implies
\[ S_R^{-1}(s,T)S_L^{-1}(p,T)pw = S_R^{-1}(s,T)TS_L^{-1}(p,T)w + S_R^{-1}(s,T)w. \]
Since $ \Qinv{s}{T} = (T^2 - 2s_0 T + |s|^2\id)^{-1}$ maps $V$ onto $\dom(T^2)$, the left $S$-resolvent $S_L^{-1}(s,T) =  \Qinv{s}{T}\overline{s} - T \Qinv{s}{T}$ maps $V$ to $\dom(T)$. Consequently, $S_L^{-1}(p,T)w\in\dom(T)$ and the right $S$-resolvent equation \eqref{RresEQ} yields
\begin{equation}
\label{SRSLsplit}
S_R^{-1}(s,T)S_L^{-1}(p,T)pw = sS_R^{-1}(s,T)S_L^{-1}(p,T)w - S_L^{-1}(p,T)w + S_R^{-1}(s,T)w.
\end{equation}
If we apply this identity with $w = pv$ we get
\begin{align*}
S_R^{-1}(s,T)&S_L^{-1}(p,T)(p^2-2s_0p+|s|^2)v\\
=&S_R^{-1}(s,T)S_L^{-1}(p,T)p^2v -2s_0 S_R^{-1}(s,T)S_L^{-1}(p,T)pv + |s|^2 S_R^{-1}(s,T)S_L^{-1}(p,T) v\\
=&sS_R^{-1}(s,T)S_L^{-1}(p,T)pv - S_L^{-1}(p,T)pv + S_R^{-1}(s,T)pv\\
 &-2s_0 S_R^{-1}(s,T)S_L^{-1}(p,T)pv + |s|^2 S_R^{-1}(s,T)S_L^{-1}(p,T) v.
\end{align*}
Applying identity \eqref{SRSLsplit} again with $w = v$ gives
\begin{align*}
S_R^{-1}(s,T)&S_L^{-1}(p,T)(p^2-2s_0p+|s|^2)v\\
=&s^2S_R^{-1}(s,T)S_L^{-1}(p,T)v - sS_L^{-1}(p,T)v + sS_R^{-1}(s,T)v - S_L^{-1}(p,T)pv + S_R^{-1}(s,T)pv\\
 &-2s_0 sS_R^{-1}(s,T)S_L^{-1}(p,T)v +2s_0 S_L^{-1}(p,T)v -2s_0S_R^{-1}(s,T)v + |s|^2 S_R^{-1}(s,T)S_L^{-1}(p,T) v\\
=&(s^2-2s_0s + |s|^2)S_R^{-1}(s,T)S_L^{-1}(p,T)v - (2s_0 -s)[S_R^{-1}(s,T)v-S_L^{-1}(p,T)v] \\
&+[S_R^{-1}(s,T)- S_L^{-1}(p,T)]pv.
\end{align*}
The identity $2s_0 = s + \overline{s}$ implies $s^2 - 2s_0s+ |s|^2 = 0$ and $2s_0 - s = \overline{s}$, and hence we obtain the desired equation \eqref{RESeqStart}.

\end{proof}
\begin{definition}
Let $T\in\closOP(V)$.
\begin{enumerate}[(i)]
\item An axially symmetric slice domain $U$ is called $T$-admissible if $\overline{\sigma_S(T)}\subset U$ and $\partial(U\cap\C_I)$ is the union of a finite number of Jordan curves for any $I\in\S$.
\item A function $f$ is said to be left (or right) slice hyperholomorphic on $\sigma_S(T)$ if it is left (or right) slice hyperholomorphic on an open set $O$ such that $\overline{U}\subset O$ for some $T$-admissible slice domain $U$. We will denote the class of  such functions by $\lhol(\sigma_S(T))$ (or $\rhol(\sigma_S(T))$).
\end{enumerate}
\end{definition}

Formally replacing the slice hyperholomorphic Cauchy-kernels  in the Cauchy-formula by the $S$-resolvent operators leads to the natural generalization of the Riesz-Dunford-functional calculus to quaternionic linear operators.
\begin{definition}[$S$-functional calculus for bounded operators]\label{SCalcBd}
Let $T\in\boundOP(V)$, choose $I\in\S$ and set $ds_I = -I\, ds$. For $f\in\lhol(\sigma_S(T))$, we define
\[f(T) := \frac{1}{2\pi}\int_{\partial(U\cap\C_I)} S_L^{-1}(s,T)\, ds_I\, f(s).\]
For $f\in\rhol(\sigma_S(T))$, we define
\[f(T) := \frac{1}{2\pi}\int_{\partial(U\cap\C_I)} f(s)\, ds_I\, S_R^{-1}(s,T).\]
These integrals are independent of the choice of the bounded slice domain $U$ and the imaginary unit $I\in\S$.
\end{definition}

A function $f$ is said to be left (or right) slice hyperholomorphic at $\infty$, if $f$ is left (or right) slice hyperholomorphic on $\H\setminus \overline{B(r,0)}$ for some ball $B(r,0)$ and the limit $f(\infty):=\lim_{x\to\infty}f(x)$ exists. By $\lhol(\sigma_S(T)\cup\{\infty\})$ we denote the set of functions $f\in\lhol(\sigma_S(T))$ that are left slice hyperholomorphic at $\infty$, and similarly we denote the corresponding sets of right slice hyperholomorphic and intrinsic functions.

\begin{lemma}\label{1to1}
Let $T\in\closOP(T)$ with $\rho_S(T)\cap\R \neq\emptyset$, let $\alpha\in\rho_S(T)\cap\R$ and set $A := (T-\alpha\id)^{-1}\linebreak[2] = -S_L^{-1}(\alpha,T)\in\boundOP(V)$. We define the function $\Phi_\alpha: \H\cup\{\infty\}\to\H\cup\{\infty\}$ by $\Phi_\alpha(s) = ( s-\alpha)^{-1}$ for $s\in\H\setminus\{\alpha\}$ and $\Phi_{\alpha}(\alpha) = \infty$ and $\Phi_{\alpha}(\infty) = 0$. Then $f\in \lhol(\sigma_S(T)\cup\{\infty\})$ if and only if $f\circ \Phi_{\alpha}^{-1}\in\lhol(\sigma_S(A)$) and $f\in \rhol(\sigma_S(T)\cup\{\infty\})$ if and only if $f\circ \Phi_{\alpha}^{-1}\in\rhol(\sigma_S(A)$).
\end{lemma}

\begin{definition}\label{SCalcUB}
Let $T\in\closOP(T)$, let $\alpha\in\rho_S(T)\cap\R$ and let $A$ and $\Phi_{\alpha}$ be as in Lemma~\ref{1to1}. For $f\in\lhol(\sigma_S(T)\cup\{\infty\})$ or $f\in\rhol(\sigma_S(T)\cup\{\infty\})$, we define
\[ f(T) = f\circ\Phi_{\alpha}^{-1}(A)\]
in the sense of Definition~\ref{SCalcBd}.
\end{definition}

\begin{theorem}\label{SCalcInt}
Let $T\in\closOP(V)$. If $f\in\lhol(\sigma_S(T)\cup\{\infty\})$, then
\[ f(T) = f(\infty)\id +  \frac{1}{2\pi}\int_{\partial(U\cap\C_I)} S_L^{-1}(s,T)\,ds_I\, f(s)\]
and if $f\in\rhol(\sigma_S(T)\cup\{\infty\})$, then
\[ f(T) = f(\infty)\id + \frac{1}{2\pi}\int_{\partial(U\cap\C_I)}f(s)\,ds_I\, S_R^{-1}(s,T)\]
for any imaginary unit $I\in\S$ and  any $T$-admissible slice domain $U$ such that $f$ is left (resp. right) slice hyperholomorphic on $U$. In particular, Definition~\ref{SCalcUB} is independent of the choice of $\alpha$.
\end{theorem}
\begin{corollary}\label{AlgUb}
Let $T\in\closOP(V)$ and let $\overline{\sigma_S}(T)$ denote the extended $S$-spectrum of $T$, that is $\sigma_S(T) = \sigma_S(T)$ if $T$ is bounded and $\sigma_S(T) = \sigma_S(T)\cup\{\infty\}$ if $T$ is unbounded. The $S$-functional calculus has the following properties:
\begin{enumerate}[(i)]
\item If $f,g\in\lhol(\overline{\sigma_S}(T))$ and $a\in\H$, then $(fa+g)(T) = f(T)a+g(T)$.  If $f,g\in\rhol(\overline{\sigma_S}(T))$ and $a\in\H$, then $(af+g)(T) = af(T)+g(T)$.
\item If $f\in\intrin(\overline{\sigma_S}(T))$ and $g\in\lhol(\overline{\sigma_S}(T))$ or if $f\in\rhol(\overline{\sigma_S}(T))$ and $g\in\intrin(\overline{\sigma_S}(T))$, then $(fg)(T) = f(T)g(T)$.
\item If $g\in\intrin(\overline{\sigma_S}(T))$, then $\sigma_S(g(T)) = g(\overline{\sigma_S}(T))$ and $f(g(T)) = f\circ g(T)$ if $f\in \lhol(g(\overline{\sigma_S}(T)))$ or $f\in\rhol(g(\overline{\sigma_S}(T)))$.
\end{enumerate}
\end{corollary}

Although polynomials do not belong to the class of admissible functions  if the operator is unbounded, they are still compatible with the $S$-functional calculus as the following lemma shows \cite[Lemma~4.4]{DA}

\begin{lemma}\label{ProdPoly}
Let $T\in\closOP(V)$ with $\rho_S(T)\neq\emptyset$ and assume that  $f\in\intrin(\sigma_{S}(T)\cup\{\infty\})$ has a zero of order $n\in\N_0\cup\{+\infty\}$ at infinity.
\begin{enumerate}[(i)]
\item\label{LKA1} For any intrinsic polynomial $P$ of degree lower than or equal to $n$, we have $P(T)f(T) = (Pf)(T)$.
\item\label{LKA2}  If $v\in\dom(T^m)$ for some $m\in\N_0\cup\{\infty\}$, then $f(T)v\in\dom(T^{m+n})$.
\end{enumerate}
\end{lemma}

Finally, we determine the slice derivatives of the $S$-resolvent operators. For bounded $T$, this has been done in \cite{CGTAYLOR}, but since the calculations are slightly more delicate for unbounded operators, we give the proof again for the sake of completeness. Definition~\ref{CauchySlicePow} and considerations as in Remark~\ref{RkResExtension} motivate the following definition.
\begin{definition}
Let $T\in\closOP(V)$. For $s\in\rho_S(T)$, we set
\[S_L^{-n}(s,T) := \sum_{k=0}^n \binom{n}{k} (-T)^k(T^2-2\Re(s)T + |s|^2\id)^{-n}\overline{s}^{n-k}\]
and
\[S_R^{-n}(s,T) :=\sum_{k=0}^n\binom{n}{k}\overline{s}^{n-k}(-T)^{k}(T^2 - 2\Re(s)T + |s|^2)^{-n}.\]
\end{definition}
\begin{remark}
The operator $\Qinv[n]{s}{T} = (T^2-2\Re(s)T+|s|^2\id)^{-n}$ is bounded and maps $V$ to $\dom(T^{2n})$, cf. the arguments in Remark~\ref{RkResExtension}. Furthermore, by Corollary~4.5 in \cite{DA}, the operators $(-T)^k$.  Hence, the operators $S_L^{-n}(s,T)$ and $S_R^{-n}(s,T)$ are bounded as $\dom((-T)^k)=\dom(T^k)\supset \dom(T^{2n})$ for any $k\in\{0,\ldots,n\}$.\end{remark}
\begin{lemma}\label{ResSDeriv}
Let $T\in\closOP(V)$ with $\rho_S(T)\cap\R\neq\emptyset$. The $n$-th slice derivatives of the left and right $S$-resolvent of $T$ for $n\in\N_0$ are
\[ \sderiv^n S_L^{-1}(s,T) = (-1)^nn!\,S_L^{-(n+1)}(s,T) \quad\text{and}\quad \sderiv^n S_R^{-1}(s,T) = (-1)^nn!\,S_R^{-(n+1)}(s,T).\]
\end{lemma}
\begin{proof}
We consider the case of the left $S$-resolvent operator. For $p\in\rho_S(T)$, let $\varepsilon$ be such that $B(\varepsilon,p)\subset\rho_S(T)$ and set $U_{p,\varepsilon}:= \H\setminus[B(\varepsilon,p)]$, where $[B(\varepsilon,p)]$ denotes the axially symmetric hull of $B(\varepsilon,p)$. For any $s\in [B(\varepsilon/2, p)]$, the map $x\mapsto S_L^{-1}(s,x)$ is then left slice hyperholomorphic on $U_{p,\varepsilon}$ and also at infinity with $S_L^{-1}(s,\infty) = 0$. We thus have by Theorem~\ref{SCalcInt}, Corollary~\ref{AlgUb} and Lemma~\ref{ProdPoly} that
\begin{align*}
&S_L^{-1}(s, T) =  \Qinv{s}{T}\overline{s} - T\Qinv{s}{T} \\
=&\frac{1}{2\pi} \int_{\partial (U_{p,\varepsilon}\cap\C_I)} S_L^{-1}(s,x)\,dx_I\, \Qinv{s}{x}\overline{s} - \frac{1}{2\pi} \int_{\partial (U_{p,\varepsilon}\cap\C_I)} S_L^{-1}(s,x)\,dx_I\, x \Qinv{s}{x} \\
=& \frac{1}{2\pi}\int_{\partial(U_{p,\varepsilon}\cap\C_I)}S_L^{-1}(x,T)\,dx_I\,S_L^{-1}(s,x),
\end{align*}
where $\Qinv{s}{x} = (x^2-2\Re(s)x-+|s|^2)^{-1}$.
By induction, one can easily see that
\[\sderiv^n S_L^{-1}(s,x) = (-1)^nn!\,S_L^{-(n+1)}(s,x) = (-1)^nn!\sum_{k=0}^{n+1} \binom{n+1}{k}(-x)^k\Qinv[(n+1)]{s}{x}\overline{s}^{n+1-k},\]
where the second equality holds because $\Qinv{s}{x}$ and $x$ commute. Since $(-x)^k$ is an intrinsic polynomial, since $\Qinv[(n+1)]{s}{x}$ is also intrinsic and since $ s\mapsto S_L^{-1}(s,T)$ can be represented by the above integral on a neighborhood of $p$, we deduce again from Theorem~\ref{SCalcInt}, Corollary~\ref{AlgUb} and Lemma~\ref{ProdPoly} that
\begin{align*}
\sderiv^{n}S_L^{-1}(s,T)& = \frac{1}{2\pi}\int_{\partial(U_{p,\varepsilon}\cap\C_I)}S_L^{-1}(x,T)\,dx_I\,\sderiv^{n}S_L^{-1}(s,x)\\
& = \frac{1}{2\pi}\int_{\partial(U_{p,\varepsilon}\cap\C_I)}S_L^{-1}(x,T)\,dx_I\,(-1)^nn!\sum_{k=0}^{n+1} \binom{n+1}{k}(-x)^k\Qinv[(n+1)]{s}{x}\overline{s}^{n+1-k}\\
&= (-1)^nn!\sum_{k=0}^{n+1} \binom{n+1}{k}(-T)^k\Qinv[(n+1)]{s}{T}\overline{s}^{n+1-k} = (-1)^nn!S_L^{-(n+1)}(s,T).
\end{align*}

\end{proof}

\subsection{Logarithm and fractional powers in the quaternions}

In order to state the main results we finally recall  the logarithmic function in the slice hyperholomorphic setting. The logarithmic function on $\hh$ is defined as
\begin{equation}\label{LOGDEF}
\log s:=\ln |s|+I_s \arccos(s_0/|s|)\qquad\text{for } s\in \mathbb{H}\setminus \{(-\infty, 0]\}.
\end{equation}
Note that for $s=s_0\in[0,\infty)$ we have  $\arccos(s_0/|s|) = 0$ and so $\log s = \ln s$. Therefore, $\log s$ is well defined also on the positive real axis and does not depend on the choice of the imaginary unit $I_s$.
It is
\[ e^{\log s} = s  \quad\text{for } s\in\hh\]
and
\[ \log e^s = s \quad\text{for }s\in\hh\ \ \text{ with }|\underline{s}|<\pi.\]
The logarithmic function is real differentiable on $\hh\setminus (-\infty,0]$. Moreover, for any $I\in\S$, the restriction of $\log s$ to the complex plane $\C_I$ coincides with a branch of the complex logarithm on $\C_I$ and is therefore holomorphic on $\C_I\setminus (-\infty, 0]$. Thus, $\log s$ is left and right slice hyperholomorphic on $\hh\setminus(-\infty,0]$.
\begin{remark}\label{LogRem}
Observe that there exist other definitions of the quaternionic logarithm in the literature. In \cite{LogBook}, the logarithm of a quaternion is for instance defined as
\[\log_{k,i} s := \begin{cases}\ln|s| + I_x\left(\arccos\frac{s_0}{|s|}+2k\pi\right), & |\underline{s}|\neq 0 \text{ or }|\underline{s}| = 0,s_0>0\\
\ln|s| + e_i\pi, &|\underline{s}|=0, s_0<0
\end{cases}
\]
where $k\in\Z$ and $e_i$ is one of the generating units of $\H$. This logarithm is however not continuous at the real line (and therefore in particular not slice hyperholomorphic at the real line) unless $k=0$. But in this case this definition of the logarithm coincides with the one given in \eqref{LOGDEF}. Indeed, the identity principle implies that \eqref{LOGDEF} defines the maximal slice hyperholomorphic extension of the natural logarithm on $(0,+\infty)$ to a subset of the quaternions.
\end{remark}

We define fractional powers of a quaternion for $\alpha\in \mathbb{R}$ as
\begin{equation}\label{fracDef}
s^{\alpha}:= e^{\alpha\log s}=e^{\alpha (\ln |s|+I_s \arccos(s_0/|s|))}, \ \ \  s\in \mathbb{H}\setminus (-\infty, 0].
\end{equation}
This function is obviously also left and right slice hyperholomorphic on $\hh\setminus(-\infty, 0]$.

\begin{definition}[Argument function]
Let $p\in\hh\setminus\{0\}$. We define $\arg(p)$ as the unique number $\theta \in [0,\pi]$ such that $p = |p| e^{\theta I_p}$.
\end{definition}
Again $\theta = \arg(s)$ does not depend on the choice of $I_s$ if $s\in\rr\setminus\{0\}$ since $p = |p|e^{0 I}$ for any $I\in\S$ if $p>0$ and $p = |p|e^{\pi I}$ for any $I\in\S$ if $p<0$.

\section{Remarks on the Slice-Hyperholomorphicity of the $S$-resolvents}\label{HolSect}
In this section we first give a precise proof for the slice hyperholomorphicity of the $S$-resolvents and then we consider the question whether the $S$-resolvents could have slice-hyperholomorphic continuations to sets larger than the $S$-resolvent set of the respective operator. We start with a new series expansion for the pseudo-resolvent $\Qinv{s}{T}$. An heuristic approach to find this expansion is to consider the immediate equation
\begin{equation}\label{QResEQ}
\Qinv{s}{T} - \Qinv{p}{T} = \Qinv{s}{T}(\Q{p}{T} -\Q{s}{T}) \Qinv{p}{T}
\end{equation}
and transform it to
\begin{equation*}
\Qinv{s}{T} = \Qinv{p}{T} + \Qinv{s}{T}(\Q{p}{T} - \Q{s}{T})\Qinv{p}{T}.
\end{equation*}
Recursive application of this equation then yields the series expansion proved in the following lemma.
\begin{lemma}\label{QExp}
Let $T\in\closOP(V)$ and $p\in\rho_S(T)$ and let $s\in\H$. If the series
\begin{equation}\label{EEQD}
\mathcal{J}(s) =  \sum_{n=0}^{+\infty} \left(\Q{p}{T} - \Q{s}{T}\right)^n\Qinv[(n+1)]{p}{T}
\end{equation}
converges absolutely in $\boundOP(V)$, then $s\in \rho_S(T)$ and its sum is the inverse of $\Q{s}{T}$.

The series converges in particular uniformly on any of the closed axially symmetric neighborhoods
\[C_{\varepsilon}(p) = \left\{ s\in\H: d_S(s,p)\leq\varepsilon\right\}\]
of $p$ with
\[d_S(s,p) =  \max\left\{2|s_0 - p_0|, \left||p|^2 - |s|^2\right|\right\}\]
and
\[\varepsilon <\frac{1}{\left\| T\Qinv{p}{T}\right\|+ \left\| \Qinv{p}{T}\right\|}.\]
\end{lemma}
\begin{proof}
Let us first consider the question of convergence of the series. The sets $C_{\varepsilon}(p)$ are obviously axially symmetric: if $s_I$ belongs to the sphere $[s]$ associated to $s$, then $s_0 = \Re(s) = \Re(s_I)$ and $|s|^2 = |s_I|^2$. Thus $d_S(s_I, p) = d_S(s,p)$ and in turn $s\in C_{\varepsilon}(p)$ if and only if $s_I\in C_{\varepsilon}(p)$. Moreover, since the map $s\mapsto d_S(s,p)$ is continuous, the sets $U_{\varepsilon}(p) := \{s\in\H: d_S(s,p) < \varepsilon\}$ are open in $\H$. Since $U_{\varepsilon}(p)\subset C_{\varepsilon}(p)$, the sets $C_{\varepsilon}$ are actually neighborhoods of~$p$.

 In order to simplify the notation, we set
\[\Qdiff(p,s) := \Q{p}{T} - \Q{s}{T} = 2(s_0 - p_0)T + (|p|^2 - |s|^2)\id.\]
 Since $\Qinv{p}{T}$ maps $V$ to $\dom(T^2)$ and  $\Qdiff(p,s)$ commutes with $\Qinv{p}{T}$  on $\dom(T^2)$, we have for any $s\in C_{\varepsilon}(p)$
\begin{align*}
 & \sum_{n=0}^{+\infty}\left\| \Qdiff(p,s)^n\Qinv[(n+1)]{p}{T} \right\| \\
 =& \sum_{n=0}^{+\infty}\left\|\left( \Qdiff(p,s)\Qinv{p}{T}\right)^n \Qinv{p}{T} \right\| \\
  \leq & \sum_{n=0}^{+\infty}\left\|\Qdiff(p,s)\Qinv{p}{T}\right\|^n \left\| \Qinv{p}{T} \right\|.
\end{align*}
We further have
\begin{align*}
\left\|\Qdiff(p,s)\Qinv{p}{T}\right\| \leq& 2|s_0-p_0|\left\|T\Qinv{p}{T}\right\| +  \left||p|^2 - |s|^2\right| \left\|\Qinv{p}{T}\right\|\\
 \leq& d_S(s,p) \left(\left\|T\Qinv{p}{T}\right\| +  \left\|\Qinv{p}{T}\right\|\right)\\
  \leq& \varepsilon \left(\left\|T\Qinv{p}{T}\right\| +  \left\|\Qinv{p}{T}\right\|\right) =: q.
\end{align*}
If now $\varepsilon <1/\left(\left\| T\Qinv{p}{T}\right\|+ \left\| \Qinv{p}{T}\right\| \right)$, then $0<q<1$ and thus
\begin{align*}
  \sum_{n=0}^{+\infty}\left\| \Qdiff(p,s)^n\Qinv[(n+1)]{p}{T} \right\|   \leq   \left\| \Qinv{p}{T} \right\| \sum_{n=0}^{+\infty}q^n < + \infty
\end{align*}
and the series converges uniformly in $\boundOP(V)$ on $C_\varepsilon(p)$.

Now assume that the series \eqref{EEQD} converges and observe that $\Q{s}{T}$, $\Q{p}{T}$ and $\Qinv{p}{T}$ commute on $\dom(T^2)$ and hence we have  for $v\in\dom(T^2)$ that
\begin{align*}
\mathcal{J}(s) \Q{s}{T}v &=  \sum_{n=0}^{+\infty} \Qdiff(p,s)^n\Qinv[(n+1)]{p}{T} \Q{s}{T} v \\
=&  \sum_{n=0}^{+\infty} \Qdiff(p,s)^n\Qinv[(n+1)]{p}{T} \left[-\Qdiff(p,s) + \Q{p}{T}\right]v  \\
=& - \sum_{n=0}^{+\infty} \Qdiff(p,s)^{n+1}\Qinv[(n+1)]{p}{T} v  \\
&+  \sum_{n=0}^{+\infty} \Qdiff(p,s)^n\Qinv[n]{p}{T}v = v.  \\
\end{align*}
On the other hand $v_N := \sum_{n =0}^N  \Qdiff(p,s)^n\Qinv[(n+1)]{p}{T} v $ belongs to $\dom(T^2)$ for any $v\in V$ and we have
\begin{align*}
\Q{s}{T} v_N =& (-\Qdiff(p,s) + \Q{p}{T}) \sum_{n=0}^N \Qdiff(p,s)^n\Qinv[(n+1)]{p}{T}v \\
=&  -\sum_{n=0}^N \Qdiff(p,s)^{n+1}\Qinv[(n+1)]{p}{T}v + \sum_{n=0}^N \Qdiff(p,s)^n\Qinv[n]{p}{T}v\\
=&  - \Qdiff(p,s)^{N+1}\Qinv[(N+1)]{p}{T}v + v.
\end{align*}
Now observe that $\Qdiff(p,s) = 2(s_0 - p_0)T + (|p|^2 - |s|^2)\id$ is defined on $\dom(T)$ and maps $\dom(T^2)$ to $\dom(T)$. Hence $\Qdiff(p,s)^2 \Qinv{p}{T}$ belongs to $\boundOP(V)$ and for $N\geq 1$
\begin{align*}
& \left\| - \Qdiff(p,s)^{N+1}\Qinv[(N+1)]{p}{T}v \right\| \\
=  &\left\| - \Qdiff(p,s)^{N-1}\Qinv[N]{p}{T} \Qdiff(p,s)^2 \Qinv{p}{T}v \right\|\\
 \leq & \left\| - \Qdiff(p,s)^{N-1}\Qinv[N]{p}{T} \right\| \left\|\Qdiff(p,s)^2 \Qinv{p}{T}v\right\|\overset {N\to\infty}{\longrightarrow} 0
 \end{align*}
 because the series \eqref{EEQD} converges in $\boundOP(V)$ by assumption. Thus $ \Q{s}{T} v_N \to v$ and $v_N \to v_{\infty} := \mathcal{J}(s)v$ as $N\to\infty$. Since $\Q{s}{T}$ is closed, we obtain that $\mathcal{J}(s) v \in \dom(\Q{s}{T}) = \dom(T^2)$ and $\Q{s}{T}\mathcal{J}(s) v = v$. Hence, $\mathcal{J}(s) = \Qinv{s}{T}$ and in turn $s\in\rho_S(T)$.

 \end{proof}

\begin{lemma}\label{Q-TQ-cont}
Let $T\in\closOP(V)$. The functions $s\to \Qinv{s}{T}$ and $s\to T \Qinv{s}{T}$, which are defined on $\rho_S(T)$ and take values in $\boundOP(V)$, are continuous.
\end{lemma}
\begin{proof}
Let $p\in\rho_S(T)$. Then $\Qinv{s}{T}$ can be represented by the series \eqref{EEQD}, which converges uniformly on a neighborhood of $p$. Hence, we have
\[ \lim_{s\to p} \Qinv{s}{T} = \sum_{n=0}^{+\infty} \lim_{s\to p}\left(2(s_0-p_0)T + \left(|p|^2-|s|^2\right)\id\right)^n\Qinv[(n+1)]{p}{T} = \Qinv{p}{T},\]
because each term in the sum is a polynomial in $s_0$ and $s_1$ with coefficients in $\boundOP(V)$ and thus continuous. Indeed
\begin{align*}
 &\left((s_0-p_0)T + \left(|p|^2-|s|^2\right)\id\right)^n\Qinv[(n+1)]{p}{T}  \\
 = &\sum_{k = 0}^n \binom{n}{k} (s_0-p_0)^k\left(|p|^2-|s|^2\right)^{n-k}T^k\Qinv[(n+1)]{p}{T}
 \end{align*}
and $T^k \Qinv[(n+1)]{p}{T}\in\boundOP(V)$ because $\Qinv[(n+1)]{p}{T}$ maps $V$ to $\dom(T^{2(n+1)})$ and $k <  2(n+1)$.
The function $s\mapsto T\Qinv{s}{T}$ is continuous because \eqref{QResEQ} implies
\begin{equation*}
\lim_{h\to 0}\left\| T\Qinv{s+h}{T} - T\Qinv{s}{T} \right\|= \lim_{h\to 0}\left\| T\Qinv{s+h}{T}(\Q{s}{T}  - \Q{s+h}{T})\Qinv{s}{T} \right\|.
\end{equation*}
 Now observe that $\Qinv{s}{T}$ maps $V$ to $\dom(T^2)$ such that
\[(\Q{s}{T}  - \Q{s+h}{T})\Qinv{s}{T} = (2h_0T + (|s|^2 - |s+h|^2)\id)\Qinv{s}{T}\]
  in turn  maps $V$to $\dom(T)$. Since $T$ and $\Qinv{s+h}{T}$ commute on $\dom(T)$ we thus have
\begin{align*}
&\lim_{h\to 0}\left\| T\Qinv{s+h}{T} - T\Qinv{s}{T} \right\|\\
=& \lim_{h\to 0}\left\| \Qinv{s+h}{T}\left(2h_0T^2 + \left(|s|^2 - |s+h|^2\right)T\right)\Qinv{s}{T} \right\|\\
\leq &  \lim_{h\to 0}\left\| \Qinv{s+h}{T}\right\| \lim_{h\to 0}\left(2h_0\left\|T^2\Qinv{s}{T}\right\| + \left(|s|^2 - |s+h|^2\right)\left\|T\Qinv{s}{T}\right\| \right)\\
 =& 0.
\end{align*}

\end{proof}

\begin{lemma}\label{Qdiff}
Let  $T\in\closOP(V)$ and $s\in\rho_S(T)$. The pseudo resolvent $\Qinv{s}{T}$ is continuously real differentiable with
\[\frac{\partial}{\partial s_0} \Qinv{s}{T} = (2T - 2s_0\id) \Qinv[2]{s}{T}\quad\text{and}\quad\frac{\partial}{\partial s_1} \Qinv{s}{T} =  -2s_1 \Qinv[2]{s}{T}.\]
\end{lemma}
\begin{proof}

Let us first compute the partial derivative of $\Qinv{s}{T}$ with respect to the real part $s_0$. Applying equation \eqref{QResEQ}, we have
\begin{align*}
\frac{\partial}{\partial s_0} \Qinv{s}{T} =& \lim_{\R\ni h\to 0} \frac{1}{h}\left(\Qinv{s+h}{T} - \Qinv{s}{T}\right)\\
= &  \lim_{\R\ni h\to 0} \frac{1}{h}\Qinv{s+h}{T}\left(\Q{s}{T}  - \Q{s+h}{T}\right)\Qinv{s}{T}\\
= & \lim_{\R\ni h\to 0} \Qinv{s+h}{T}\left(2T -2s_0\id - h\id\right)\Qinv{s}{T},
\end{align*}
 where $ \lim_{\R\ni h \to 0}f(h)$ denotes the limit of a function $f$ as $h$ tends to $0$ in $\R$.
Since the composition and the multiplication with scalars are continuous operations on $\boundOP(V)$, we further have
\begin{align*}
\frac{\partial}{\partial s_0} \Qinv{s}{T} = & \lim_{\R\ni h\to 0} \Qinv{s+h}{T} \lim_{\R\ni h\to 0}\left((2T -2s_0\id)\Qinv{s}{T} - h \Qinv{s}{T}\right)\\
=& \Qinv{s}{T}(2T-2s_0\id)\Qinv{s}{T} = (2T - 2s_0\id)\Qinv[2]{s}{T},
\end{align*}
where the last equation holds true because $\Qinv{s}{T}$ maps $V$ to $\dom(T^2)\subset\dom(T)$ and $T$ and $\Qinv{s}{T}$ commute on $\dom(T)$. Observe that the partial derivative $\frac{\partial}{\partial s_0}\Qinv{s}{T}$ is even continuous because it is the sum and product of continuous functions by Lemma~\ref{Q-TQ-cont}.

If we write $s = s_0 + I_s s_1$, then we can argue in a similar way to show that the derivative of $\Qinv{s}{T}$ with respect to $s_1$ is
\begin{align*}
\frac{\partial}{\partial s_1} \Qinv{s}{T} =& \lim_{\R\ni h\to 0} \frac{1}{h}\left(\Qinv{s+hI_s}{T} - \Qinv{s}{T}\right)\\
= &  \lim_{\R\ni h\to 0} \frac{1}{h}\Qinv{s+h_Is}{T}\left(\Q{s}{T}  - \Q{s+hI_s}{T}\right)\Qinv{s}{T}\\
= & \lim_{\R\ni h\to 0} \Qinv{s+hI_s}{T}\left(-2s_1 - h\right)\Qinv{s}{T}\\
= & \lim_{\R\ni h\to 0} \Qinv{s+hI_s}{T} \lim_{\R\ni h\to 0}\left(-2s_1\Qinv{s}{T} - h\Qinv{s}{T}\right)\\
&= -2s_1 \Qinv[2]{s}{T}.
\end{align*}
Again this derivative is continuous as the product of two continuous functions by Lemma~\ref{Q-TQ-cont}.

Finally, we easily obtain that $\Qinv{s}{T}$ is continuously real differentiable from the fact that $\Qinv{s}{T}$ is continuously differentiable in the variables $s_0$ and $s_1$: if $s = \xi_0 + \sum_{i=1}^3 \xi_i e_i$, then the partial derivative with respect to $\xi_0$ corresponds to the partial derivative with respect to $s_0$ and thus exists and is continuous. The partial derivative with respect to $\xi_i$ for $1\leq i \leq 3$ on the other hand exists and is continuous for $s_1\neq 0$ because $\Qinv{s}{T}$ can be considered as the composition of the continuously differentiable functions $s \mapsto s_1$ and $s_1 \to \Qinv{s_0+ I_s s_1}{T}$. For $s_1 = 0$ (that is for $s\in \R$), we can simply choose $I_s  =e_i$ and then the partial derivative with respect to $\xi_i$ corresponds to the partial derivative with respect to~$s_1$.

\end{proof}

\begin{lemma}\label{TQDiff}
Let  $T\in\closOP(V)$ and $s\in\rho_S(T)$. The function $s\mapsto T\Qinv{s}{T}$ is continuously real differentiable with
\[
\frac{\partial}{\partial s_0} T\Qinv{s}{T} = (2T^2-2s_0T)\Qinv[2]{s}{T}\quad\text{and}\quad \frac{\partial}{\partial s_1} T\Qinv{s}{T} = -2s_1 T\Qinv[2]{s}{T}.
\]
\end{lemma}
\begin{proof}
If  $ \lim_{\R\ni h \to 0}f(h)$ denotes again the limit of a function $f$ as $h$ tends to $0$ in $\R$, then we obtain from \eqref{QResEQ} that
\begin{align*}
\frac{\partial}{\partial s_0} T\Qinv{s}{T} = & \lim_{\R\ni h \to 0} \frac{1}{h} \left(T \Qinv{s+h}{T} - T \Qinv{s}{T}\right) \\
= & \lim_{\R\ni h \to 0} \frac{1}{h} T\Qinv{s+h}{T}\left( \Q{s}{T} -  \Q{s+h}{T}\right)\Qinv{s}{T}\\
= & \lim_{\R\ni h \to 0} \frac{1}{h} T\Qinv{s+h}{T}\left(2hT - 2hs_0\id -h^2\id\right)\Qinv{s}{T}\\
= & \lim_{\R\ni h \to 0} \Qinv{s+h}{T}\left(2T^2 - 2s_0T -hT\right)\Qinv{s}{T},
\end{align*}
because $\left(2hT - 2hs_0\id -h^2\id\right)\Qinv{s}{T}$ maps $V$ to $\dom(T)$ and $T$ and $\Qinv{s+h}{T}$ commute on $\dom(T)$. Since the composition and the multiplication with scalars are continuous operations on the space $\boundOP(V)$ and since the pseudo-resolvent is continuous by Lemma~\ref{Q-TQ-cont}, we  get
 \begin{gather*}
\frac{\partial}{\partial s_0} T\Qinv{s}{T}  =   \lim_{\R\ni h \to 0} \Qinv{s+h}{T} \lim_{\R\ni h \to 0}\left(\left(2T^2 - 2s_0T\right)\Qinv{s}{T} -hT\Qinv{s}{T}\right)\\
=\Qinv{s}{T} (2T^2 - 2s_0T) \Qinv{s}{T}  = (2s_0 T -  2T^2)\Qinv[2]{s}{T}.
\end{gather*}
This function is continuous because we can write it as the product of functions that are continuous by Lemma~\ref{Q-TQ-cont}.

The derivative with respect to $s_1$ can be computed using similar arguments via
\begin{align*}
\frac{\partial}{\partial s_1} T\Qinv{s}{T} &=  \lim_{\R\ni h \to 0} \frac{1}{h} \left(T \Qinv{s+hI_s}{T} - T \Qinv{s}{T}\right)
\\
= & \lim_{\R\ni h \to 0} \frac{1}{h} T\Qinv{s+hI_s}{T}\left( \Q{s}{T} -  \Q{s+hI_s}{T}\right)\Qinv{s}{T}\\
= & \lim_{\R\ni h \to 0} \frac{1}{h} T\Qinv{s+hI_s}{T}\left( - 2hs_1 -h^2\right)\Qinv{s}{T}\\
= & \lim_{\R\ni h \to 0} \Qinv{s+hI_s}{T} \lim_{\R\ni h \to 0}\left(- 2s_1T\Qinv{s}{T} -hT\Qinv{s}{T}\right)\\
= &  -2s_1T\Qinv[2]{s}{T}.
\end{align*}
Also this derivative is continuous because it can be written in the form $\frac{\partial}{\partial s_1} T\Qinv{s}{T} = - 2s_1\left(T\Qinv{s}{T}\right)Q_{s}(T)^{-1}$ as the product of functions that are continuous by Lemma~\ref{Q-TQ-cont}.

Finally, we see as in the proof of Lemma~\ref{Qdiff} that $T\Qinv{s}{T}$ is continuously differentiable in the four real coordinates by considering it as the composition of the two continuously real differentiable functions $s \mapsto (s_0, s_1)$ and $(s_0,s_1) \mapsto T\Qinv{s_0+ Is_1}{T}$ resp. by choosing $I_s$ appropriately if $s\in\R$.

\end{proof}
\begin{corollary}
Let  $T\in\closOP(V)$ and $s\in\rho_S(T)$. The left and the right S-resolvent are real continuously differentiable.
\end{corollary}
\begin{proof}
The $S$-resolvents are sums of functions that are continuously real differentiable by Lemma~\ref{Qdiff} and Lemma~\ref{TQDiff} and hence continuously real differentiable themselves.

\end{proof}

Let us now give the proof of Lemma~\ref{ResHol2342}: the left $S$-resolvent is right slice hyperholomorphic and the right $S$-resolvent is left slice hyperholomorphic in the scalar variable on $\rho_S(T)$.
\begin{proof}[Proof of Lemma~\ref{ResHol2342}]
We consider only the case of the left $S$-resolvent, the other one works with analogous arguments. Applying Lemma~\ref{Qdiff} and Lemma~\ref{TQDiff}, we have
\begin{align*}
 \frac{\partial}{\partial s_0} S_L^{-1}(s,T) = & \frac{\partial}{\partial s_0} \Qinv{s}{T}\overline{s}  - \frac{\partial}{\partial s_0} T\Qinv{s}{T}\\
= & (2T-2s_0\id)\Qinv[2]{s}{T}\overline{s} + \Qinv{s}{T}- \left(2T^2-2s_0T\right)\Qinv[2]{s}{T}\\
= & (2T-2s_0\id)\Qinv[2]{s}{T}\overline{s} +  \left(-T^2  + |s|^2\id\right)\Qinv[2]{s}{T}.
\end{align*}
Since $s_0$ and $|s|^2$ are real, we can commute them with $\Qinv[2]{s}{T}$ and by applying the identities $2s_0 = s + \overline{s}$ and $ |s|^2 = s\overline{s}$ we obtain
\begin{align*}
\frac{\partial}{\partial s_0} S_L^{-1}(s,T) = -T^2 \Qinv[2]{s}{T} + 2T\Qinv[2]{s}{T} \overline{s} - \Qinv[2]{s}{T}\overline{s}^2.
\end{align*}
For the partial derivative with respect to $s_1$, we obtain
\begin{align*}
&\frac{\partial}{\partial s_1} S_L^{-1}(s,T) \\
= & \frac{\partial}{\partial s_1} \Qinv{s}{T}\overline{s}  - \frac{\partial}{\partial s_1} T\Qinv{s}{T}\\
= & -2 s_1 \Qinv[2]{s}{T}\overline{s} - \Qinv{s}{T}I_s + 2s_1T\Qinv[2]{s}{T}\\
= & -2 s_1 \Qinv[2]{s}{T}\overline{s} -  (T^2 - 2s_0T + |s|^2\id)\Qinv[2]{s}{T}I_s+2s_1T\Qinv[2]{s}{T}.
\end{align*}
We can now commute $2s_0$, $2s_1$ and $|s|^2$ with $\Qinv{s}{T}$ because they are real.  By exploiting the identities $2s_0 = s + \overline{s}$, $-2s_1 = (s - \overline{s})I_s$ and $|s|^2 = s \overline{s}$, we obtain
\begin{align*}
&\frac{\partial}{\partial s_1} S_L^{-1}(s,T) = \left(-T^2 \Qinv[2]{s}{T} + 2T\Qinv[2]{s}{T}\overline{s} - \Qinv[2]{s}{T}(T)\overline{s}^2\right)I_s
\end{align*}
Hence, $s\mapsto S_L^{-1}(s,T)$ is right slice hyperholomorphic as
\[ \frac12\left(\frac{\partial}{\partial s_0} S_L^{-1}(s,T) + \frac{\partial}{\partial s_1} S_L^{-1}(s,T)I_s\right) = 0.  \]

\end{proof}

 In the sequel we will need the fact that the $S$-resolvent set is the maximal domain of slice hyperholomorphicity of the $S$-resolvents such that they do not have a slice hyperholomorphic continuation. In the complex case this is guaranteed by the well-known estimate
\begin{equation}\label{CResEst}
 \| R(z,A) \| \geq \frac{1}{\dist(z,\sigma(A))},
 \end{equation}
 where $R(z,A)$ denotes the resolvent and $\sigma(A)$ the spectrum of the complex linear operator~$A$. This estimate assures that $ \| R(z,A) \|\to+\infty$ as $z$ approaches $\sigma(A)$ and in turn that the resolvent does not have any holomorphic continuation to a larger domain.

   In the quaternionic setting, an estimate similar to \eqref{CResEst} cannot hold true: consider $\lambda = \lambda_0 + I_{\lambda} \lambda_1$ with $\lambda_1 > 0$ and $I\in\S$ and the operator $T = \lambda \id$ acting on some two-sided Banach space $V$. Its $S$-spectrum $\sigma_S(T)$ coincides with the sphere $[\lambda]$ associated to $\lambda$ and its left $S$-resolvent is
 \[ S_L^{-1}(s,T) =  (\lambda^2 - 2 s_0 \lambda + |s|^2)^{-1} (\overline{s}-\lambda)\id.\]
If  $s\in\C_{I_{\lambda}}$, then $\lambda$ and $s$ commute and the left $S$-resolvent reduces to  $S_L^{-1}(s,T) = (s-\lambda)^{-1}\id$ with $\|S_L^{-1}(s,T)\| = 1/|s-\lambda |$. Thus, if $s$ tends to  $\overline{\lambda}$ in $\C_{I{_\lambda}}$, then $\dist(s,\sigma_S(T))\to 0$ because $\overline{\lambda}\in\sigma_S(T)$ but at the same time $\|S_L^{-1}(s,T)\| \to 1/|\lambda - \overline{\lambda}| = 1/(2\lambda_1)< +\infty$.

Nevertheless, although \eqref{CResEst} does not have a counterpart in the quaternionic setting, we can show that, under suitable assumptions, the norms of the $S$-resolvents explode near the $S$-spectrum. As it happens often in quaternionic operator theory, this requires that we work with spectral spheres of associated quaternions instead of single spectral values.

\begin{lemma}
Let $T\in\closOP(V)$ and $s\in\rho_S(T)$. Then
\begin{equation}\label{QsEst}
 \|\Qinv{s}{T}\| + \|T\Qinv{s}{T}\| \geq \frac{1}{d_S(s,\sigma_S(T))},
 \end{equation}
where $ d_S(s,\sigma_S(T)) = \inf_{p\in\sigma_S(T)} d_S(s,p)$ and $d_S(s,p)$ is defined as in the Lemma~\ref{QExp}.
\end{lemma}
\begin{proof}
Set $C_s:= \|\Qinv{s}{T}\| + \|T\Qinv{s}{T}\|$. If $d_S(s,p) < 1/C_s$, then $p\in\rho_S(T)$ by Lemma~\ref{QExp}. Thus, $ d_S(s,p) \geq 1/C_s$ for any $p\in\sigma_S(T)$. Taking the infimum over all $p\in\sigma_S(T)$, this inequality still holds true such that we obtain $d_S(s,\sigma_S(T))  \geq 1/C_s$, which is equivalent to~\eqref{QsEst}.

\end{proof}

\begin{lemma}\label{QSest}
Let  $T\in\closOP(V)$ and $s\in\rho_S(T)$. Then
\[ \sqrt{2\left\| \Qinv{s}{T} \right\|} \leq   \left\| S_L^{-1}(s,T) \right\| + \left\| S_L^{-1}(\overline{s},T)\right\|  \]
and in turn
\[ \sqrt{\left\| \Qinv{s}{T} \right\|} \leq   \sqrt{2} \sup_{s_I\in[s]}\left\| S_L^{-1}(s_I,T) \right\|.  \]
Analogous estimates hold for the right $S$-resolvent operator.

\end{lemma}
\begin{proof}
Observe that $\Qinv{s}{T} = \Qinv{\overline{s}}{T}$ for $s\in\rho_S(T)$. Hence, because $2s_0 = s + \overline{s}$, we have
\begin{align*}
&S_L^{-1}(s,T)S_L^{-1}(s,T) + S_L^{-1}(s,T)S_L^{-1}(\overline{s},T) \\
=&\left(\Qinv{s}{T}\overline{s}-T\Qinv{s}{T}\right)\left(\Qinv{s}{T}\overline{s}-T\Qinv{s}{T}\right)\\
& + \left(\Qinv{s}{T}\overline{s}-T\Qinv{s}{T}\right)\left(\Qinv{s}{T}s-T\Qinv{s}{T}\right)\\
= &\left(\Qinv{s}{T}\overline{s}-T\Qinv{s}{T}\right)2\left(s_0\id-T\right)\Qinv{s}{T}
\end{align*}
and similarly
\begin{align*}
&S_L^{-1}(\overline{s},T)S_L^{-1}(s,T) + S_L^{-1}(\overline{s},T)S_L^{-1}(\overline{s},T) \\
= &\left(\Qinv{s}{T}s-T\Qinv{s}{T}\right)2\left(s_0\id-T\right)\Qinv{s}{T}.
\end{align*}
Therefore
\begin{align*}
&S_L^{-1}(s,T)S_L^{-1}(s,T) + S_L^{-1}(s,T)S_L^{-1}(\overline{s},T)\\
&+ S_L^{-1}(\overline{s},T)S_L^{-1}(s,T) + S_L^{-1}(\overline{s},T)S_L^{-1}(\overline{s},T) \displaybreak[3]\\
= &\left(\Qinv{s}{T}\overline{s}-T\Qinv{s}{T}\right)2\left(s_0\id-T\right)\Qinv{s}{T} \\
&+ \left(\Qinv{s}{T}s-T\Qinv{s}{T}\right)2\left(s_0\id-T\right)\Qinv{s}{T} \displaybreak[3]\\
= & 2\left(s_0\id-T\right)\Qinv{s}{T}2\left(s_0\id-T\right)\Qinv{s}{T}\\
= & 4(T^2-2s_0T + s_0^2\id)\Qinv[2]{s}{T} = 4\Qinv[1]{s}{T} - 4 s_1^2\Qinv[2]{s}{T},
\end{align*}
which can be rewritten as
\begin{align*}
4\Qinv{s}{T} = &S_L^{-1}(s,T)S_L^{-1}(s,T) + S_L^{-1}(s,T)S_L^{-1}(\overline{s},T)\\
&+ S_L^{-1}(\overline{s},T)S_L^{-1}(s,T) + S_L^{-1}(\overline{s},T)S_L^{-1}(\overline{s},T) \displaybreak[3] + 4 s_1^2\Qinv[2]{s}{T}.
 \end{align*}
Thus, we can estimate
\begin{align}
\notag4\left\| \Qinv{s}{T} \right\| = &
\left\|S_L^{-1}(s,T) \right\|\left\|S_L^{-1}(s,T) \right\| + \left\|S_L^{-1}(s,T) \right\|\left\|S_L^{-1}(\overline{s},T)\right\|\\
\notag&+ \left\|S_L^{-1}(\overline{s},T) \right\|\left\|S_L^{-1}(s,T) \right\|+ \left\|S_L^{-1}(\overline{s},T)\right\|\left\|S_L^{-1}(\overline{s},T) \right\|+
4 \left\|s_1^2\Qinv[2]{s}{T} \right\|\\
= & \left(\left\| S_L^{-1}(s,T) \right\| + \left\| S_L^{-1}(\overline{s},T)\right\|\right)^2 +  \left\|2s_1\Qinv{s}{T} \right\| \left\|2s_1\Qinv{s}{T} \right\| \label{BBFB}.
 \end{align}

Finally observe that
\begin{align*}
 2\Qinv{s}{T}s_1I_s =& T\Qinv{s}{T} - \Qinv{s}{T}(s_0 - I_ss_1) \\
 &- \left( T\Qinv{s}{T} - \Qinv{s}{T}(s_0 +I_ss_1)\right) = S_L^{-1}(s,T) - S_L^{-1}(\overline{s},T)
 \end{align*}
such that
\begin{align*}
\left\|2s_1\Qinv{s}{T} \right\| = \left\|2\Qinv{s}{T} s_1I_s \right\| \leq \left\| S_L^{-1}(s,T) \right\| + \left\| S_L^{-1}(\overline{s},T)\right\|.
\end{align*}
Combining this estimate with \eqref{BBFB}, we finally obtain
\[ 2\left\| \Qinv{s}{T} \right\| \leq   \left(\left\| S_L^{-1}(s,T) \right\| + \left\| S_L^{-1}(\overline{s},T)\right\|\right)^2  \]
and hence the statement for the left $S$-resolvent operator. The estimates for the right $S$-resolvent operator can be shown with similar computations.

\end{proof}

\begin{lemma}\label{NormExplode}
Let $T\in\closOP(V)$. If $(s_n)_{n\in\N}$ is a bounded sequence in $ \rho_S(T)$  with
\[\lim_{n\to\infty} \dist(s_n,\sigma_S(T) )= 0,\]
 then
\[\lim_{n\to\infty} \sup_{s\in[s_n]} \left\|S_L^{-1}(s,T)\right\| =+ \infty\quad \text{and} \quad\lim_{n\to\infty} \sup_{s\in[s_n]} \left\|S_R^{-1}(s,T)\right\| =+ \infty.\]
\end{lemma}
\begin{proof}
First of all observe that $\dist(s_n,\sigma_S(T)) \to 0$ if and only if $d_S(s_n,\sigma_S(T)) \to 0$ because $\sigma_S(T)$ is axially symmetric. Indeed, for any $n\in\N$ there exits $p_n\in\sigma_S(T)$ such that
\[|s_n - p_n| < \dist(s_n, \sigma_S(T)) + 1/n.\]
If $\dist(s_n,\sigma_S(T)) \to 0$, then $|s_n - p_n|\to 0$ and hence $|s_{n,0} - p_{n,0}| \to 0$. Since the sequence $s_n$ is bounded, the sequence $p_n$ is bounded too and we also have
\[
 \left||s_n|^2 - |p_n|^2\right| \leq |s_n| |\overline{s_n} - \overline{p_n}| + |s_n-p_n||\overline{p_n}| \to 0
 \]
and in turn
\begin{equation*}
0 < d_S(s_n,\sigma_S(T)) \leq d_S(s_n,p_n) = \max\left\{ |s_{n,0} - p_{n,0}|, \left| |s_n|^2 - |p_n|^2\right|\right\} \longrightarrow 0.
\end{equation*}
If on the other hand $d_S(s_n,\sigma_S(T))$ tends to zero, then there exists a sequence $p_n \in \sigma_S(T)$ such that
\[d_S(s_n,p_n) < d_S(s_n, \sigma_S(T)) + 1/n\]
 and in turn $d_S(s_n,p_n) \to 0$. Since $\sigma_S(T)$ is axially symmetric and $d(s_n,p_{n,I}) = d(s_n,p_n)$ for any $p_{n,I}\in [p_n]$, we can moreover assume that $I_{p_n} = I_{s_n}$. Then
 \[0 \leq |s_{n,0} - p_{n,0}| \leq d_S(s_n,p_n)\to 0.\]
  Since $s_n$ and in turn also $p_n$ are bounded, this implies $|s_{n,0}^2 - p_{n,0}^2| \to 0$, from which we deduce that also $|s_{n,1}^2 - p_{n,1}^2|\to 0$ because
\[ 0\leq \left| s_{n,0}^2 - p_{n,0}^2 + s_{n,1}^2 - p_{n,1}^2\right| = \left||s_n|^2 -|p_n|^2\right| \leq d_S(s_n,p_n) \to 0. \]
Since $s_{n,1}\geq 0$ and $p_{n,1}\geq 0$, we conclude $s_{n,1} - p_{n,1} \to 0$ and, since $I_s = I_p$, also
\[ 0 < \dist(s_n,\sigma_S(T)) \leq | s_n - p_n| = \sqrt{(s_{n,0} - p_{n,0})^2 + (s_{n,1} - p_{n,1})^2} \to 0.\]

 Now assume that $s_n\in\rho_S(T)$ with $\dist(s_n,\sigma_S(T))\to 0$. By the above considerations and \eqref{QsEst}, we have
 \begin{equation}\label{UUZZTT}
 \|\Qinv{s_n}{T}\| + \|T\Qinv{s_n}{T} \| \to +\infty.
 \end{equation}
 We show now that every subsequence $(s_{n_k})_{k\in\N}$ has a subsequence $(s_{n_{k_j}})_{j\in\N}$ such that
 \begin{equation}\label{YXCV}
\lim_{j\to+\infty} \sup_{s\in[s_{n_{k_j}}]}\|S_L^{-1}(s,T)\|= +\infty,
 \end{equation}
  which implies $\lim_{n\to+\infty}\sup_{s\in[s_n]}\|S_L^{-1}(s,T)\| = +\infty$. We thus consider an arbitrary subsequence $(s_{n_k})_{k\in\N}$ of $(s_n)_{n\in\N}$. If it has a subsequence $(s_{n_{k_j}})_{j\in\N}$ such that $\|Q_{s_{n_{k_j}}}(T)\| \to +\infty$, then Lemma~\ref{QSest} implies \eqref{YXCV}. Otherwise $\|\Qinv{s_{n_j}}{T}\|\leq C$ for some constant $C>0$ and we deduce from \eqref{UUZZTT} that $\|T\Qinv{s_{n_j}}{T} \| \to +\infty$. Observe that
  \begin{align*}
T\Qinv{s_{n_k}}{T} =  - \frac12 S_L^{-1}(s_{n_k},T) - \frac12 S_L^{-1}(\overline{s_{n_k}},T) + s_{n_k,0}\Qinv{s_{n_k}}{T},
\end{align*}
from which we obtain the estimate
\begin{align*}
\left\| T\Qinv{s_{n_k}}{T}\right\|&\leq \sup_{s\in [s_{n_k}]}\left\| S_L^{-1}(s_{n_k},T)\right\| + |s_{n_k,0}|\left\|\Qinv{s_{n_k}}{T}\right\| \\
&\leq \sup_{s\in [s_{n_k}]}\left\| S_L^{-1}(s_{n_k},T)\right\| + CM
\end{align*}
with $M = \sup_{n\in\N}|s_n| <+\infty$. Since the left-hand side tends to infinity as $k\to +\infty$, we obtain that also  $\sup_{s\in [s_{n_k}]}\left\| S_L^{-1}(s_{n_k},T)\right\|\to+\infty$ and thus the statement holds true. The case of the right $S$-resolvent can be shown with analogous arguments.

\end{proof}
\begin{definition}
Let $f$ be a left (or right) slice hyperholomorphic function defined on an axially symmetric open set $U$. A left (or right) slice hyperholomorphic function $g$ defined on an axially symmetric open set $U'$ with $U\subsetneq U'$ is called a slice hyperholomorphic continuation of $f$ if $f(s) = g(s)$ for all $s\in U$. It is called nontrivial if $V = U'\setminus U$ cannot be separated from $U$, i.e. if $U' \neq U \cup V$ for some open set $V$ with $\overline{V}\cap \overline{U} = \emptyset$.
\end{definition}
\begin{theorem}\label{SExt}
Let $T\in\closOP(V)$. There does not exist any nontrivial slice hyperholomorphic continuation of the left or of the right $S$-resolvent operator.
\end{theorem}
\begin{proof}
Assume that there exists a nontrivial extension $f$ of $S_L^{-1}(s,T)$ to an axially symmetric open set $U$ with $\rho_S(T)\subsetneq U$. Then there exists a point $s\in U \cap \partial\rho_S(T)$ and a sequence $s_n\in\rho_S(T)$ with $\lim_{n \to +\infty}s_n = s$ such that
\[\lim_{n\to+\infty}\left\|S_L^{-1}(s_n,T)\right\| = \lim_{n\to+\infty} \|f(s_n)\| = \|f(s)\| < +\infty. \]
Moreover, also $\overline{s_n}\to\overline{s}$ as $n\to+\infty$ and in turn
\[\lim_{n\to+\infty}\left\|S_L^{-1}(\overline{s_n},T)\right\| = \lim_{n\to+\infty} \|f(\overline{s_n})\| = \|f(\overline{s})\| < +\infty. \]
From the representation formula, Theorem~\ref{RepFoOP}, we then deduce
\[ \lim_{n\to+\infty}\sup_{s\in[s_n]}\left\|S_L^{-1}(s,T)\right\| \leq \lim_{n\to+\infty}\left\|S_L^{-1}(s_n,T)\right\|+\left\|S_L^{-1}(\overline{s_n},T)\right\|<+\infty.\]

On the other hand the sequence $s_n$ is bounded and $\dist(s_n,\sigma_S(T)) \leq |s_n - s| \to 0$. Lemma~\ref{NormExplode} therefore implies
$\lim_{n\to+\infty}\sup_{s\in [s_n]}\left\|S_L^{-1}(s,T)\right\| = +\infty$, which is a contradiction. Thus, the analytic continuation $(f,U)$ cannot exist.

For the right $S$-resolvent, we argue analogously.

\end{proof}

\begin{remark}
We suspected that it might be possible to improve the above results by finding an estimate of the form \eqref{CResEst} for the pseudo-resolvent $\Qinv{s}{T}$ instead of the $S$-resolvents. In this case Lemma~\ref{QSest} would yield an estimate of the form \eqref{CResEst} for the norm of the $S$-resolvents on an entire sphere instead of a single point. This is however not possible as the following example shows: consider the space $\ell^p(\N)$ of $\H$-valued $p$-summable sequences with $p\in[1,+\infty)$. Any sequence $(\lambda_n)_{n\in\N}$ with $\lambda_n\in\H$ does  obviously define a right linear, densely defined and closed operator on $\ell^p(\N)$ via $T(a) = (\lambda_na_n)_{n\in\N}$ for $a = (a_n)_{n\in\N}$. If $(\lambda_n)_{n\in\N}$ is unbounded, then $T$ is unbounded. Otherwise $\| T\| = \sup_{n\in\N}|\lambda_n| = \|(\lambda_n)_{n\in\N}\|_{\infty}$. Indeed,
\[\|T(a)\|_p = \sqrt[p]{\sum_{n\in\N} |\lambda_na_n|^p} \leq \|(\lambda_n)_{n\in\N}\|_{\infty} \sqrt[p]{\sum_{n\in\N} |a_n|^p} =  \|(\lambda_n)_{n\in\N}\|_{\infty} \|a\|_p \]
such that $\| T \| \leq \|(\lambda_n)_{n\in\N}\|_{\infty} $ and, with $e_m = (\delta_{m,n})_{n\in\N}$ where $\delta_{m,n}$ is the Kronecker delta, on the other hand
\[ \|\lambda_m\| =  \sqrt[p]{\sum_{n\in\N} |\lambda_n\delta_{m,n}\|^p} =  \|T(e_{m})\| \leq \|T\|\]
for any $m\in\N$ such that also $ \|(\lambda_n)_{n\in\N}\|_{\infty}\leq \|T\|$. The $S$-spectrum of $T$ is
\begin{equation}\label{Tspec} \sigma_S(T) = \overline{\bigcup_{n\in\N}[\lambda_n]}\end{equation}
as one can see easily: any $\lambda_n$ is a right eigenvalue since for instance $T(e_n) = e_n\lambda_n$ and hence the relation $\supset$ in \eqref{Tspec} holds true by Theorem~\ref{SvsR}, the axial symmetry and the closedness of the $S$-spectrum. If on the other hand  $s$ does not belong to the right hand side of \eqref{Tspec}, then  $\delta_s = \inf_{n\in\N}\dist(s,[\lambda_n]) =  \inf_{n\in\N} |s_{I_{\lambda_n}}-\lambda_n| >0$, where $s_{I_{\lambda_n}} = s_0 + I_{\lambda_n}s_1$. As
\[\Q{s}{T}(a) = \left((\lambda_n - s_{I_{\lambda_n}})(\lambda_n - \overline{s_{I_{\lambda_n}}})a_n\right)_{n\in\N}\]
and in turn
\[\Qinv{s}{T}(a) = \left((\lambda_n - s_{I_{\lambda_n}})^{-1}(\lambda_n - \overline{s_{I_{\lambda_n}}})^{-1}a_n\right)_{n\in\N},\]
we have $\|\Qinv{s}{T}\| \leq 1/\delta_s^2 < +\infty$ such that $s\in\rho_S(T)$. Thus, the relation $\subset$ in \eqref{Tspec} also holds true.

Now choose a sequence $(\lambda_n)_{n\in\N}$ such that $\lambda_{n,1}\to+\infty$ as $n\to+\infty$ and consider the respective operator $T$ on $\ell^p(\N)$. For simplicity, consider for instance $\lambda_n = In$ with $I\in\S$. By the above considerations, the sequence $s_N = I(N+1/N)$ with $N = 2,3,\ldots$ does then satisfy $\dist(s_N,\sigma_S(T))\to 0$ as $N\to +\infty$ and
\begin{equation}\label{CCTFA}
\| \Qinv{s_N}{T} \| = \sup_{n\in\N} \frac{1}{|\lambda_n-s_N||\lambda_n - \overline{s_N}|} = \frac{1}{|\lambda_N-s_N||\lambda_N - \overline{s_N}|} = \frac{1}{2 + \frac{1}{N^2}}.
\end{equation}
Indeed, if $n<N$, then some simple computations show that the inequality
\[ \frac{1}{|\lambda_n-s_N||\lambda_n - \overline{s_N}|} = \frac{1}{N + \frac{1}{N}-n}\frac{1}{n + N + \frac{1}{N}} < \frac{1}{2 + \frac{1}{N^2}} = \frac{1}{|\lambda_N-s_N||\lambda_N - \overline{s_N}|} \]
is equivalent to $0 < N^2 - n^2$, which is obviously true. Similarly, in the case $n>N$, the inequality
\[ \frac{1}{|\lambda_n-s_N||\lambda_n - \overline{s_N}|} = \frac{1}{n - N - \frac{1}{N}}\frac{1}{n + N + \frac{1}{N}} < \frac{1}{2 + \frac{1}{N^2}} = \frac{1}{|\lambda_N-s_N||\lambda_N - \overline{s_N}|} \]
is equivalent to $ 4 + 1/N^2 < n^2 - N^2$, which holds true since $2\leq N < n$.

From \eqref{CCTFA}, we see that $\| \Qinv{s_N}{T} \|\leq 2$ although $\dist(S_N,\sigma_S(T))\to 0$. Consequently, the pseudo-resolvent cannot satisfy an estimate that is analogue to \eqref{CResEst}.

Also controlling the norm of $T\Qinv{s}{T}$ by the norm of $\Qinv{s}{T}$ in order to improve \eqref{QsEst} is not possible: if we consider the operator $T\Qinv{s_N}{T}$ in the above example, then
\[T\Qinv{s_N}{T}(a) = \left(\frac{n}{n - N - \frac{1}{N}}\frac{1}{ I\left(n + N + \frac{1}{N}\right)}a_n\right)_{n\in \N}\]
and
\[ \|T\Qinv{s_N}{T}\|  \leq \|T\Qinv{s_N}{T}(e_N)\| = \frac{N^2}{2N+ \frac{1}{N}} \to +\infty\]
shows that $\|T\Qinv{s_N}{T}\|$ tends to infinity although $\|\Qinv{s_N}{T}\|$ stays bounded.
\end{remark}

\section{Fractional Powers of an Operator}\label{NagelSect}

In the following we assume that $T$ is a closed quaternionic right linear operator such that $(-\infty,0]\subset \rho_S(T)$ and such that there exists a positive constant $M>0$ such that
\begin{equation}\label{ResEst}
\|S_R^{-1}(s,T)\|\leq \frac{M}{1+|s|} \qquad\text{for }s\in(-\infty,0].
\end{equation}

\begin{definition}
For $a\in\R$ and $\theta\in(0,\pi)$ we denote by $\osector{\theta}{a}$ the open sector
\[\osector{\theta}{a} := \{s\in\hh: \arg(s-a)> \theta\}\]
and by $\csector{\theta}{a}$ the closed sector
\[\csector{\theta}{ a} := \{s\in\hh: \arg(s-a)\geq \theta\}.\]
\end{definition}
\begin{lemma}\label{StripCy}
There exist constants $a_0 > 0 $ and $\theta_0\in(0,\pi)$ and $M_n>0, n\in\N,$ such that the closed sector $\csector{\theta_0}{a_0}$
is contained in $\rho_S(T)$ and such that, for $n\in\N$,
\begin{equation}\label{StripEstimate}
\|S_R^{-n}(s,T)\|\leq \frac{M_n}{(1+|s|)^n}\quad\text{for }s\in \csector{\theta_0}{a_0}
\end{equation}
and
\begin{equation}\label{StripEstL}
\|S_L^{-n}(s,T)\|\leq \frac{M_n}{(1+|s|)^n}\quad\text{for }s\in \csector{\theta_0}{a_0}.
\end{equation}
\end{lemma}
\begin{proof}
By Lemma~\ref{ResSDeriv}, we have
\begin{equation}\label{wqerw}
 \sderiv ^k S_R^{-1}(s,T) = (-1)^kk!\,S_R^{-(k+1)}(s,T)\ \ \ {\rm for}\ \  k\in\N.
 \end{equation}
Hence, by Lemma~\ref{ResHol2342} and Corollary~\ref{SDPropOP}, the map $s\mapsto S_R^{-n}(s,T)$ is a left slice hyperholomorphic function on $\rho_S(T)$ with values in $\boundOP(V)$ for any $n\in\N$. From the identity \eqref{wqerw} we deduce
\[ \sderiv^k S_R^{-n}(s,T) = \sderiv^{k+n-1} \frac{(-1)^{n-1}}{(n-1)!} S_R^{-1}(s,T) = \frac{(-1)^{k}(k+n-1)!}{(n-1)!}\,S_R^{-(k+n)}(s,T).\]
When we apply Theorem~\ref{TaylorOP} in order to expand $S_R^{-n}(s,T)$ into a Taylor series at a real point $\alpha\in\rho_S(T)$, we therefore get
\begin{equation}\label{BinSum}
\begin{split}
 S_R^{-n}(s,T) &= \sum_{k=0}^{+\infty}\frac{1}{k!} (s-\alpha)^k \sderiv ^{k}S_R^{-n}(\alpha,T)
\\
&= \sum_{k=0}^{+\infty}\frac{1}{k!} (s-\alpha)^k  \frac{(-1)^{k}(k+n-1)!}{(n-1)!}\,S_R^{-(k+n)}(\alpha,T)
\\
&
=\sum_{k=0}^{+\infty}(-1)^k\binom{n+k-1}{k}(s-\alpha)^kS_R^{-(n+k)}(\alpha,T)
\end{split}
\end{equation}
on any ball $B(r,\alpha)$ contained in $\rho_S(T)$. Since $\alpha$ is real,  $S_R^{-n}(\alpha,T) = \left(S_R^{-1}(\alpha,T)\right)^n$ and thus $\|S_R^{-n}(\alpha,T) \|\leq\|S_R^{-1}(\alpha,T)\|^n$. The ratio test and the estimate \eqref{ResEst} therefore imply that this series converges on the ball with radius $(1+|\alpha|)/M$ centered at $\alpha$ for $\alpha\in(-\infty,0]$. In particular, considering the case $n=1$, we deduce from Theorem~\ref{SExt}  that any such ball is contained in $\rho_S(T)$. Otherwise the above series would give a nontrivial slice hyperholomorphic continuation of $S_R^{-1}(s,T)$.

 Set $a_0 = \min\left\{\frac{1}{4M}, 1\right\}$.  Then the closed ball $\overline{B(a_0,0)}$ is contained in $\rho_S(T)$ and for any $s\in \overline{B(a_0,0)}$, we have the estimate
\[
\begin{split}
\left\| S_R^{-n}(s,T)\right\| &\leq \sum_{k=0}^{\infty}\binom{n+k-1}{k}|s|^k\|S_R^{-(n+k)}(0,T)\|
\\
&
\leq \sum_{k=0}^{\infty}\binom{n+k-1}{k}\frac{1}{(4M)^k}M^{n+k}\frac{(1+|s|)^{n+k}}{(1+|s|)^{n+k}}
\\
&
= \frac{2^nM^n}{(1+|s|)^n}\sum_{k=0}^{\infty}\binom{n+k-1}{k}\frac{1}{2^k}
\\
&
= \frac{4^nM^n}{(1+|s|)^n},
\end{split}
\]
where the last equation follows from the Taylor series expansion $(1-z)^{-n} = \sum_{k=0}^{\infty}\binom{n+k-1}{k}z^k$ for $|z|<1$.

Now set $\varphi = \pi- \arctan(\frac{1}{2M})$ and consider the sector $\csector{\varphi}{ 0} = \{s\in\hh:\arg(s) \geq \varphi\}$. For any $s=s_0+I_ss_1\in\csector{\varphi}{0}$, we have $0\leq s_1\leq|s_0|/(2M)$ and from the power series expansion \eqref{BinSum} of $S_R^{-n}(s,T)$ at $s_0$, we obtain
\begin{align*}
\left\|S_R^{-n}(s,T)\right\| &\leq \sum_{k=0}^{\infty} \binom{n+k-1}{k}|s_1|^k\left\|S_R^{-1}(s_0,T)\right\|^{n+k} \\
& \leq \sum_{k=0}^{\infty} \binom{n+k-1}{k}\left(\frac{|s_0|}{2M}\right)^k\left(\frac{M}{1+|s_0|}\right)^{n+k}\\
& \leq  \left(\frac{M}{1+|s_0|}\right)^{n} \sum_{k=0}^{\infty} \binom{n+k-1}{k}\frac{1}{2^k} = \frac{2^nM^n}{(1+|s_0|)^n}.
\end{align*}
Since $|s| \leq |s_0|+|s_1| \leq (1+ \frac{1}{2M})|s_0|$, we get
\[\|S_R^{-n}(s,T)\| \leq  \frac{2^nM^n}{\left(1 + \left(1+ \frac{1}{2M}\right)^{-1}|s|\right)^n} \leq \frac{\left(1+ \frac{1}{2M}\right)^n2^nM^n}{\left(1 + |s|\right)^n}.\]
Hence, the estimate
$$
\left\|S_R^{-n}(s,T)\right\|\leq \frac{M_n}{(1+|s|)^n}
$$
 with
 $$
 M_n := \left(1+ \frac{1}{2M}\right)^n4^nM^n
 $$
  holds true on the entire set $V = \overline{\osector{\varphi}{0}\cup B(a_0,0)}$. Now observe that the sector $\csector{\theta_0}{a_0}$ with
\[\theta_0 := \arctan\left(\frac{a_0\sin\varphi}{a_0(-1+\cos\varphi)}\right)\]
is entirely contained in $V$ and we obtain the statement for $S_R^{-1}(s,T)$.

Since $S_L^{-1}(s,T) = S_R^{-1}(s,T)$ for $s\in(-\infty,0]$, the estimate \eqref{ResEst} applies also to the left $S$-resolvent. Thus we can use analogous arguments to prove \eqref{StripEstL}.

\end{proof}

\begin{definition}\label{FracPow}
Let $I\in\S$ and let $\csector{\theta_0}{a_0}$ be the sector obtained from Lemma~\ref{StripCy}. Let $\theta\in (\theta_0, \pi)$ and choose a piecewise smooth path $\Gamma$ in $(\csector{\theta_0}{ a_0}\cap\C_I)\setminus(-\infty,0]$ that goes from $\infty e^{I\theta}$ to $\infty e^{-I\theta}$. For $\alpha > 0$, we define
\begin{equation}\label{tialpfa}
T^{-\alpha}:= \frac{1}{2\pi}\int_{\Gamma} s^{-\alpha}\, ds_I\, S_R^{-1}(s,T).
\end{equation}
\end{definition}
\begin{theorem}\label{theoindep}
For any $\alpha>0$, the operator $T^{-\alpha}$ is bounded and independent of the choice of $I\in\S$, of $\theta\in(\theta_0,\pi)$ and of  the concrete path $\Gamma$ in $\C_I$ and therefore well-defined.
\end{theorem}
\begin{proof}
The estimate \eqref{StripEstimate} assures that the integral \eqref{tialpfa} exists and defines a bounded right-linear operator. Since $s\mapsto s^{-\alpha}$ is right slice hyperholomorphic and $s\mapsto S_R^{-1}(s,T)$ is left slice hyperholomorphic, the independence of the choice of $\theta$ and the independence of the choice of the path $\Gamma$ in the complex plane $\C_I$ follow from Cauchy's integral theorem for operator-valued slice hyperholomorphic functions, Theorem~\ref{COP}.

In order to show that $T^{-\alpha}$ is independent of the choice of the imaginary unit $I\in\S$, consider an arbitrary imaginary unit $J\in\S$ with $J\neq I$. Let $\theta_0 < \theta_s < \theta_p <\pi$ and set $U_s := \H\setminus\overline{\osector{\theta_s}{ 0}\cup B(a_0/2,0)}$ and $U_p:= \H\setminus\overline{\osector{\theta_p}{ 0}\cup B(a_0/3,0)}$. (The indices $s$ and $p$ are chosen in order to indicate the variable of integration over the boundary of the respective set in the following calculation.) Then $U_p$ and $U_s$ are slice domains that contain $\sigma_S(T)$ and $\partial(U_s\cap\C_I)$ and $\partial(U_p\cap\C_J)$ are paths that are admissible in Definition~\ref{FracPow}.

Observe that $s\mapsto s^{-\alpha}$ is right slice hyperholomorphic on $\overline{U_p}$ and that, by our choices of $U_p$ and $U_s$, we have $s\in U_p$ for any $s\in\partial(U_s\cap\C_I)$. If we choose $r>0$ large enough, then $s\in U_p\cap B(r,0)$ and we obtain from Theorem~\ref{Cauchy} that
\begin{align*}
s^{\alpha} &= \lim_{r\to\infty}\frac{1}{2\pi}\int_{\partial(U_p\cap B(r,0)\cap\C_I)} p^{-\alpha}\, dp_J\, S_R^{-1}(p,s) \\
&= \frac{1}{2\pi}\int_{\partial(U_p\cap\C_J)}p^{-\alpha}\, dp_J\, S_R^{-1}(p,s),
\end{align*}
where the second equations holds since $p^{-\alpha}\to 0$ uniformly as $p\to\infty$ in $U_p$. For $T^{-\alpha}$, we thus obtain
\begin{align}
\notag T^{-\alpha} &= \frac{1}{2\pi}\int_{\partial(U_s\cap\C_I)}s^{-\alpha}\, ds_I\, S_R^{-1}(s,T) \\
&= \frac{1}{(2\pi)^2}\int_{\partial(U_s\cap\C_I)}\left(\int_{\partial(U_p\cap\C_J)}p^{-\alpha}\, dp_J\, S_R^{-1}(p,s)\right)\, ds_I\, S_R^{-1}(s,T). \label{BeforeFubini}
\end{align}
We now apply Fubini's theorem. The estimate that justifies this is very technical and is therefore moved to Appendix~\ref{Fubini1} at the end of the paper. By exchanging the order of integration, we get
\begin{align*}
\notag T^{-\alpha}&= \frac{1}{2\pi}\int_{\partial(U_p\cap\C_J)}p^{-\alpha}\, dp_J \left(\frac{1}{2\pi}\int_{\partial(U_s\cap\C_I)}\, S_R^{-1}(p,s)\, ds_I\, S_R^{-1}(s,T)\right)\\
&=  \frac{1}{2\pi}\int_{\partial(U_p\cap\C_J)}p^{-\alpha}\, dp_J\, S_R^{-1}(p,T),
\end{align*}
where the last equation follows as an application of the $S$-functional calculus and Theorem~\ref{SCalcInt} since $S_R^{-1}(p,\infty) = \lim_{s\to\infty} S_R^{-1}(p,s) = 0$. Hence, the operator $T^{-\alpha}$ is also independent of the choice of the imaginary unit $I\in\S$.

\nopagebreak
\end{proof}

If $\alpha\in\N$, then $s^{-\alpha}$ is right slice hyperholomorphic at infinity. The following Corollary then immediately follows as an application of the $S$-functional calculus and Theorem~\ref{SCalcInt}.
\begin{corollary}\label{NaturalPowers}
If $\alpha\in\N$, then the operator $T^{-\alpha}$ defined in \eqref{tialpfa} coincides with the $\alpha$-th inverse power of $T$.
\end{corollary}
If we follow the arguments of the proof of Theorem~5.27 in \cite[Chapter~II]{EngelNagel}, we obtain an integral representation of  $T^{-\alpha}$ that is almost identical to the one derived in  \cite{EngelNagel} for the complex case: the only difference is the different constant in front of the integral. This is due to the different choice of the branch of the logarithm that is used in \cite{EngelNagel} in order to define the fractional powers. As pointed out in Remark~\ref{LogRem} it is not possible to define different branches of the logarithm in a quaternionic slice hyperholomorphic  setting.

In Corollary~\ref{novesei} we however obtain an integral representation that is clearly different from any integral representation known from the classical complex setting.
\begin{theorem}\label{noveseiALLN}
 Let $n\in\N$. For $\alpha\in (0,n+1)$ with $\alpha\notin\N$, the operator $T^{\alpha}$ defined in \eqref{tialpfa} has the representation
\begin{equation}
\label{talphaALLN}
\begin{split}
T^{-\alpha}=(-1)^{n+1}\frac{\sin(\alpha\pi)}{\pi}\frac{n!}{(n-\alpha)\cdots(1-\alpha)}\int_{0}^{+\infty}t^{n-\alpha}S_R^{-(n+1)}(-t,T)\,dt.
\end{split}
\end{equation}

\end{theorem}
\begin{proof}
Let $a_0$ and $\theta_0$ be the constants obtained from Corollary~\ref{StripCy}. For $a\in(0,a_0)$ and $\theta \in (\theta_0,\pi)$, we can choose $U = \H\setminus\csector{\theta}{a}$ and integrate  over the boundary $\partial(U\cap\C_I)$ of $U$ in $\C_I$ for some $I\in\S$ in the integral representation of $T^{-\alpha}$. The boundary consists of the path
\[\gamma(t) = \begin{cases} a-te^{I\theta}, &t\in(-\infty, 0]\\ a+te^{-I\theta}, & t\in(0,\infty)\end{cases}, \]
and hence it is
\begin{align*}
T^{-\alpha} =& \frac{1}{2\pi}\int_{-\infty}^0 (a - te^{I\theta})^{-\alpha}(-I)(-e^{I\theta}) S_R^{-1}(a - te^{I\theta},T)\, dt \\
&+ \frac{1}{2\pi}\int_0^{+\infty} (a + te^{-I\theta})^{-\alpha}(-I)e^{-I\theta} S_R^{-1}(a + te^{-I\theta},T)\,dt\\
=& \frac{I}{2\pi}\int_0^{+\infty} (a +te^{I\theta})^{-\alpha}e^{I\theta} S_R^{-1}(a + te^{I\theta},T)\,dt \\
&- \frac{I}{2\pi}\int_0^{+\infty} (a + te^{-I\theta})^{-\alpha}e^{-I\theta} S_R^{-1}(a + te^{-I\theta},T)\,dt.
\end{align*}
Integrating $n$ times by parts yields
\begin{align*}T^{-\alpha} =& \frac{n!}{(n-\alpha)\cdots(1-\alpha)}\frac{I}{2\pi}\int_0^{+\infty} (a +te^{I\theta})^{n-\alpha}e^{I\theta} S_R^{-(n+1)}(a + te^{I\theta},T)\,dt \\
&- \frac{n!}{(n-\alpha)\cdots(1-\alpha)}\frac{I}{2\pi}\int_0^{+\infty} (a + te^{-I\theta})^{n-\alpha}e^{-I\theta} S_R^{-(n+1)}(a + te^{-I\theta},T)\,dt.
\end{align*}
Because of the estimate \eqref{StripEstimate}, we can apply Lebesgue's dominated convergence theorem with dominating function
\[f(t)=\begin{cases} C(1+t^{n-\alpha})&\text{if }t\leq 1\\ Ct^{-\alpha-1}&\text{if }t>1,\end{cases}\]
where $C>0$ is a sufficiently large constant. Taking the limit $a\to 0$, we obtain
\begin{equation}\begin{split}\label{AnyAngle}
T^{-\alpha} =
& \frac{n!}{(n-\alpha)\cdots(1-\alpha)}\frac{I}{2\pi}\int_0^{+\infty} t^{n-\alpha}e^{I\theta(n-\alpha)}e^{I\theta} S_R^{-(n+1)}(te^{I\theta},T)\,dt \\
&- \frac{n!}{(n-\alpha)\cdots(1-\alpha)}\frac{I}{2\pi}\int_0^{+\infty} t^{n-\alpha}e^{-I\theta(n-\alpha)}e^{-I\theta} S_R^{-(n+1)}(te^{-I\theta},T)\,dt
\end{split}
\end{equation}
and then, taking the limit $\theta\to\pi$, we get
\begin{equation*}
\begin{split}
T^{-\alpha} =& -\frac{n!}{(n-\alpha)\cdots(1-\alpha)}\frac{I}{2\pi}\int_0^{+\infty} t^{n-\alpha}e^{I\pi(n-\alpha)} S_R^{-(n+1)}(-t,T)\,dt \\
&+ \frac{n!}{(n-\alpha)\cdots(1-\alpha)}\frac{I}{2\pi}\int_0^{+\infty} t^{n-\alpha}e^{-I\pi(n-\alpha)}S_R^{-(n+1)}(-t,T)\,dt\\
=& (-1)^{n+1}\frac{\sin(\alpha\pi)}{\pi}\frac{n!}{(n-\alpha)\cdots(1-\alpha)}\int_{0}^{+\infty}t^{n-\alpha}S_R^{-(n+1)}(-t,T)\,dt,
\end{split}
\end{equation*}
where the last equation follows from the identity $-Ie^{I\pi(n-\alpha)}+Ie^{-I\pi(n-\alpha)}= \sin((n-\alpha))\pi=(-1)^{n+1}\sin(\alpha\pi)$.

\nopagebreak
\end{proof}

\begin{corollary}\label{Ialpha}
For the identity operator $\id$, it is $\id^{-\alpha}=\id$ for $\alpha \geq 0$.
\end{corollary}
\begin{proof}
If $\alpha\in\N$, this is follows immediately from Corollary~\ref{NaturalPowers}. For $\alpha\notin\N$, consider $n\in\N$ with $\alpha\in(0,n+1)$. Then, since $S_R^{-1}(s,T) = (s-1)^{-1}\id$, it is
\begin{align*}
\id^{-\alpha} &= (-1)^{n+1}\frac{\sin(\alpha\pi)}{\pi}\frac{n!}{(n-\alpha)\cdots(1-\alpha)}\int_{0}^{+\infty}\frac{t^{n-\alpha}}{(-t-1)^{n+1}}\,dt\,\id\\
&=\frac{\sin(\alpha\pi)}{\pi}\frac{n!}{(n-\alpha)\cdots(1-\alpha)}\int_{0}^{+\infty}\frac{t^{n-\alpha}}{(t+1)^{n+1}}\,dt\,\id.
\end{align*}
By \cite[3.194]{GTABLEOFINT}, we have
\begin{equation}
\label{ScalInt}
\int_{0}^{+\infty}\frac{t^{n-\alpha}}{(t+1)^{n+1}}\,dt =B(n-\alpha+1,\alpha) = \frac{(n-\alpha)\cdots(1-\alpha)}{n!}\frac{\pi}{\sin(\pi\alpha)},
\end{equation}
where $B(x,y)$ denotes the Beta function, and hence $\id^{-\alpha} = \id$.

\end{proof}

\begin{corollary}\label{novesei}
 Let $\alpha\in (0,1)$. Then
\begin{equation}
\label{talpha}
\begin{split}
T^{-\alpha}=-\frac{\sin(\alpha\pi)}{\pi}\int_0^{\infty} t^{-\alpha}S_R^{-1}(-t,T)\,dt.
\end{split}
\end{equation}
\end{corollary}

\begin{corollary}\label{CyBound}
For $\alpha\in(0,n+1)$, the operators $T^{-\alpha}$ are uniformly bounded by the constant $M_{n+1}$ obtained from Lemma~\ref{StripCy}.
\end{corollary}
\begin{proof}
From \eqref{talphaALLN}, Lemma~\ref{StripCy} and \eqref{ScalInt}, we obtain the estimate
\[\|T^{-\alpha}\|\leq \frac{\sin(\alpha\pi)}{\pi}\frac{n!}{(n-\alpha)\cdots(1-\alpha)}\int_{0}^{+\infty}t^{n-\alpha}\frac{M_{n+1}}{(1+t)^{n+1}}\,dt
= M_{n+1}.
\]
\end{proof}

\begin{corollary}
Assume that $\sigma_S(T)\subset\{s \in\hh: \Re(s)>0\}$ and that $\theta_0$ in Lemma~\ref{StripCy} can be chosen lower or equal to $\pi/2$. For $\alpha\in(0,1)$, then
\[T^{-\alpha} = \frac{1}{\pi}\int_{0}^\infty \tau^{-\alpha} \left(\cos\left(\frac{\alpha\pi}{2}\right)T +\sin\left(\frac{\alpha\pi}{2}\right)\tau\id\right)(T^2 + \tau^2)^{-1}\, d\tau.\]
\end{corollary}
\begin{proof}
By our assumptions, we can choose $n=0$ and $\theta = \pi/2$ in $\eqref{AnyAngle}$. Since $e^{I\frac{\pi}{2}} = I$ and $e^{-I\frac{\pi}{2}} = -I$, we then have
\[T^{-\alpha} = \frac{I}{2\pi}\int_0^{\infty} t^{-\alpha}e^{-I\frac{\alpha-1}{2}\pi} S_R^{-1}(It,T)\,dt - \frac{I}{2\pi}\int_0^{\infty}  t^{-\alpha}e^{I\frac{\alpha-1}{2}\pi} S_R^{-1}( -It,T)\,dt.\]
Observe that
$$
S_R^{-1}(\pm t I,T) = -(T\pm t I\id)(T^2 + t^2)^{-1}.
$$
 Thus,
\begin{align*}
T^{-\alpha} = \frac{I}{2\pi}\int_{0}^\infty t ^{-\alpha} \left(-e^{-I\frac{\alpha-1}{2}\pi} (T + t  I\id)+ e^{I\frac{\alpha-1}{2}\pi} (T - t  I\id)\right)(T^2 + t ^2)^{-1}\, dt .
\end{align*}
Some easy simplifications show
\[-e^{-I\frac{\alpha-1}{2}\pi} (T + t  I\id)+ e^{I\frac{\alpha-1}{2}\pi} (T - t  I\id)= -2I\left[\cos\left(\frac{\alpha\pi}{2}\right)T +2\sin\left(\frac{\alpha\pi}{2}\right)t \id\right],\]
and in turn
\[T^{-\alpha} = \frac{1}{\pi}\int_{0}^\infty t ^{-\alpha} \left(\cos\left(\frac{\alpha\pi}{2}\right)T +\sin\left(\frac{\alpha\pi}{2}\right)t \id\right)(T^2 + t ^2)^{-1}\, dt .\]

\end{proof}

Observe that $s\mapsto s^{-\alpha}$ is both right and left slice hyperholomorphic. Hence, we could also use the left $S$-resolvent operator to define fractional powers of $T$. Indeed, this yields exactly the same operator.

\begin{Pn}\label{LeFT}
Let $\alpha>0$ and let $\Gamma$ be an admissible path as in Definition~\ref{FracPow}. The operator $T^{-\alpha}$ satisfies
\begin{equation}\label{TAlphaSL}
 T^{-\alpha} = \frac{1}{2\pi} \int_{\Gamma} S_L^{-1}(s,T)\,ds_I\,s^{-\alpha}.
\end{equation}
\end{Pn}
\begin{proof}
Computations analogue to those in the proof of Theorem~\ref{noveseiALLN} show that, for $n\in\N$ and $\alpha\in(0,n+1)$ with $\alpha\notin\N$, one has
\begin{align*}
&\frac{1}{2\pi} \int_{\Gamma} S_L^{-1}(s,T)\,ds_I\,s^{-\alpha}\\
=&(-1)^{n+1}\frac{\sin(\alpha\pi)}{\pi}\frac{n!}{(n-\alpha)\cdots(1-\alpha)}\int_{0}^{+\infty}S_L^{-(n+1)}(-t,T)t^{n-\alpha}\,dt.\end{align*}
But for real $t$ one has $S_L^{-1}(-t,T) = (-t-T)^{-1} = S_R^{-1}(-t,T)$, and in turn this integral equals
\[(-1)^{n+1}\frac{\sin(\alpha\pi)}{\pi}\frac{n!}{(n-\alpha)\cdots(1-\alpha)}\int_{0}^{+\infty}t^{n-\alpha}S_R^{-(n+1)}(-t,T)\,dt = T^{-\alpha},\]
where the last equation follows from Theorem~\ref{noveseiALLN}.

If $\alpha\in\N$, then this follows immediately from the $S$-functional calculus and Theorem~\ref{SCalcInt} because $s^{-\alpha}$ is left and right slice hyperholomorphic at infinity.

\end{proof}

We recall the following lemma from \cite[Lemma~3.23]{acgs}.
\begin{lemma}\label{Lemma321} Let $B\in \mathcal{B}(V)$. Let $G$ be a bounded axially symmetric s-domain and assume that
$f\in \mathcal{N}(G)$.
Then, for $p\in G$, we have
$$
\frac{1}{2\pi}\int_{\partial(G\cap\mathbb{C}_I)}f(s)ds_I
(\overline{s}B-Bp)(p^2-2s_0p+|s|^2)^{-1}=Bf(p).
$$
\end{lemma}
\begin{theorem}\label{SemiGroup}
The family $\{T^{-\alpha}\}_{\alpha > 0 }$ has the semigroup property $T^{-\alpha}T^{-\beta}=T^{-\alpha-\beta}$.
\end{theorem}

\begin{proof}
Choose $\theta_p$ and $\theta_s$ such that $\max\{\theta_0,\pi/2\} < \theta_p <\theta_s < \pi$ and $a_p$ and $a_s$ such that $0 < a_s < a_p< a_0$, where $a_0$ and $\theta_0$ are the constants obtained from Lemma~\ref{StripCy} and $a_p$ is sufficiently small such that $\overline{B_{a_p}(0)}\subset\osector{\theta_0}{a_0}$. Then the sets
\[G_p = \H\setminus\left(\csector{\theta_p}{0}\cup \overline{B_{a_p}(0)}\right)\quad\text{ and }\quad G_s=\H\setminus\left(\csector{\theta_s}{0}\cup\overline{ B_{a_s}(0)}\right)\]
 satisfy $\sigma_S(T)\subset G_p$ and $\overline{G_p}\subset\ G_s$ and for $I\in \mathbb{S}$ their boundaries $\partial(G_p\cap\C_I)$ and $\partial(G_s\cap\C_I)$ are admissible paths as in Definition~\ref{FracPow}. The subscripts $p$ and $s$ refer again to the respective variables of integration in the following calculation.

The $S$-resolvent equation \eqref{resEQ} and Proposition~\ref{LeFT} imply
\begin{align*}
 T^{-\alpha} T^{-\beta}&=
 {{1}\over{(2\pi)^2 }} \int_{\partial (G_s\cap \mathbb{C}_I)} \  s^{-\alpha}\ ds_I \ S_R^{-1} (s,T)
  \int_{\partial (G_p\cap \mathbb{C}_I)} \ S_L^{-1} (p,T)\ dp_I \  p^{-\beta}
 \displaybreak[2]\\
&=\frac{1}{(2\pi)^2 }\int_{ \partial ( G_s \cap \mathbb{C}_I) } s^{-\alpha}\ ds_I \int_{ \partial ( G_p \cap \mathbb{C}_I) }
S_R^{-1}(s,T)p(p^2-2s_0p+|s|^2)^{-1}dp_I\  p^{-\beta}
 \\
&-\frac{1}{(2\pi)^2 }\int_{ \partial ( G_s \cap \mathbb{C}_I) } s^{-\alpha}\ ds_I \int_{ \partial ( G_p \cap \mathbb{C}_I) }
S_L^{-1}(p,T)p(p^2-2s_0p+|s|^2)^{-1}dp_I\  p^{-\beta}
\\
&
-\frac{1}{(2\pi)^2 }\int_{ \partial ( G_s \cap \mathbb{C}_I) } s^{-\alpha}\ ds_I
\int_{ \partial ( G_p \cap \mathbb{C}_I) }\overline{s}S_R^{-1}(s,T)(p^2-2s_0p+|s|^2)^{-1}
 dp_I\ p^{-\beta}
 \\
&
+\frac{1}{(2\pi)^2 }\int_{ \partial ( G_s \cap \mathbb{C}_I) } s^{-\alpha}\ ds_I
\int_{ \partial ( G_p \cap \mathbb{C}_I) }\overline{s}S_L^{-1}(p,T)(p^2-2s_0p+|s|^2)^{-1}
 dp_I\  p^{-\beta}.
 \end{align*}
But since the functions $p\mapsto p(p^2-2s_0p+|s|^2)^{-1}p^{-\beta}$ and $p\mapsto (p^2-2s_0p+|s|^2)^{-1}p^{-\beta}$ are holomorphic on an open set that contains $\overline{G_{p}\cap\C_I}$ and since they tend uniformly to zeros as $p\to\infty$ in $G_p$, Cauchy's integral theorem implies
$$
\frac{1}{(2\pi)^2 }\int_{ \partial ( G_s \cap \mathbb{C}_I) } s^{-\alpha}\ ds_I \int_{ \partial ( G_p \cap \mathbb{C}_I) }
S_R^{-1}(s,T)p(p^2-2s_0p+|s|^2)^{-1}dp_I\  p^{-\beta}=0
$$
and
$$
-\frac{1}{(2\pi)^2 }\int_{ \partial ( G_s \cap \mathbb{C}_I) } s^{-\alpha}\ ds_I
\int_{ \partial ( G_p \cap \mathbb{C}_I) }\overline{s}S_R^{-1}(s,T)(p^2-2s_0p+|s|^2)^{-1}
 dp_I\  p^{-\beta}=0.
$$
It follows that
\begin{equation}\label{anEquation}
\begin{split}
 &T^{-\alpha}\  T^{-\beta}\\
 =&-\frac{1}{(2\pi)^2 }\int_{ \partial ( G_s \cap \mathbb{C}_I) } s^{-\alpha}\ ds_I \int_{ \partial ( G_p \cap \mathbb{C}_I) }
S_L^{-1}(p,T)p(p^2-2s_0p+|s|^2)^{-1}dp_I\  p^{-\beta}
\\
&
+\frac{1}{(2\pi)^2 }\int_{ \partial ( G_s \cap \mathbb{C}_I) } s^{-\alpha}\ ds_I
\int_{ \partial ( G_p \cap \mathbb{C}_I) }\overline{s}S_L^{-1}(p,T)(p^2-2s_0p+|s|^2)^{-1}
 dp_I\  p^{-\beta}.
 \end{split}
\end{equation}
Quite technical estimates, which can be found in Appendix~\ref{Fubini2}, justify the application of Fubini's theorem in these integrals such that we can exchange the order of integration and obtain
\begin{multline*}
 T^{-\alpha}\  T^{-\beta}=
 \frac{1}{(2\pi)^2 }\int_{ \partial ( G_s \cap \mathbb{C}_I) } s^{-\alpha}\ ds_I \\
 \cdot \int_{ \partial ( G_p \cap \mathbb{C}_I) }
[\overline{s}S_L^{-1}(p,T)-S_L^{-1}(p,T)p](p^2-2s_0p+|s|^2)^{-1}dp_I\  p^{-\beta}.
 \end{multline*}

Using Lemma \ref{Lemma321} with $B = S_L^{-1}(p,T)$, we finally get
\begin{align*}
 T^{-\alpha}\  T^{-\beta}&=\frac{1}{2\pi } \int_{ \partial ( G_p \cap \mathbb{C}_I) }S_L^{-1}(p,T)dp_I \ p^{-\alpha}\  p^{-\beta}\\
 &=\frac{1}{2\pi } \int_{ \partial ( G_p \cap \mathbb{C}_I) }S_L^{-1}(p,T)dp_I \ p^{-\alpha-\beta}=T^{-\alpha-\beta}.
 \end{align*}

\end{proof}

\begin{lemma}
The semigroup $(T^{-\alpha})_{\alpha\geq 0}$ is strongly continuous.
\end{lemma}
\begin{proof}
We first consider $v\in\dom(T)$. For $\alpha\in(0,1)$, we have
\[S_R^{-1}(t,T)v - S_R^{-1}(t,\id)v = S_R^{-1}(t,T)S_R^{-1}(t,\id)(Tv - \id v)\quad\text{if $t\in\R$.}\]
 Hence, we deduce from  Corollary~\ref{novesei} that
\begin{align*}T^{-\alpha}v - \id^{-\alpha}v &= -\frac{\sin(\alpha\pi)}{\pi}\int_0^\infty t^{-\alpha}S_R^{-1}(-t,T)v\,dt +\frac{\sin(\alpha\pi)}{\pi}\int_0^\infty t^{-\alpha}S_R^{-1}(-t,\id)v\,dt\\
& =- \frac{\sin(\alpha\pi)}{\pi}\int_0^\infty t^{-\alpha}S_R^{-1}(t,T)S_R^{-1}(t,\id)(Tv - \id v)\,dt \overset{\alpha\to 0}{\longrightarrow} 0
\end{align*}
because $\sin(\alpha\pi)\to 0$ as $\alpha\to 0$ and the integral is uniformly bounded for ${\alpha\in[0,1/2]}$ due to~\eqref{ResEst}. Since $\id^{-\alpha} = \id$ by Corollary~\ref{Ialpha}, we get $T^{-\alpha}v \to v$ as $\alpha\to 0$ for any $v\in\dom(T)$.

For arbitrary $v\in V$ and $\varepsilon >0$, there exists $v_\varepsilon\in\dom(T)$ with $\|v-v_\varepsilon\|<\varepsilon$ because $\dom(T)$ is dense in~$V$. Corollary~\ref{CyBound} therefore implies
\[
\begin{split}
\lim_{\alpha\to 0}\|Tv - v\| &\leq  \lim_{\alpha\to 0}\|Tv - T^{-\alpha}v_\varepsilon\| + \|T^{-\alpha} v_\varepsilon - v_\varepsilon\| + \| v_\varepsilon - v \|
\\
&
\leq (M_1+1)\|v- v_\varepsilon\|
\\
&
\leq (M_1+1)\varepsilon.
\end{split}
\]
Since $\varepsilon>0$ was arbitrary, we deduce that $T^{-\alpha}v \to v$ as $\alpha\to 0$ even for arbitrary $v\in V$. This is equivalent to the strong continuity of the semigroup $(T^{-\alpha})_{\alpha\geq 0}$.

\end{proof}

\begin{Pn}
The operator $T^{-\alpha}$ is injective for any $\alpha > 0$.
\end{Pn}
\begin{proof}
For $\alpha>0$ choose $\beta>0$ with $n=\alpha + \beta\in\N$. Then $T^{-\beta}T^{-\alpha}=T^{-n}$ and in turn $T^nT^{-\beta}T^{-\alpha} = \id$, which implies the injectivity of $T^{-\alpha}$.

\end{proof}
The previous proposition allows us to define powers of $T$ also for $\alpha>0$.

\begin{definition}\label{TPos}
For $\alpha>0$ we define the operator $T^{\alpha}$ as the inverse of the operator $T^{-\alpha}$, which is defined on $\dom(T^{\alpha})=\ran(T^{-\alpha})$.
\end{definition}

\begin{corollary}
Let $\alpha,\beta\in\rr$. Then the operators $T^{\alpha}T^{\beta}$ and $T^{\alpha+\beta}$ agree on $\dom(T^\gamma)$ with $\gamma = \max\{\alpha,\beta,\alpha+\beta\}$.
\end{corollary}
\begin{proof}
If $\alpha,\beta\geq 0$ and $v\in \dom(T^{\alpha,\beta})$ then, since $T^{-(\alpha+\beta)} = T^{-\beta}T^{-\alpha}$ by Theorem~\ref{SemiGroup}, we have
\[T^{\alpha}T^{\beta} v = T^{\alpha}T^{\beta}(T^{-\beta}T^{-\alpha}T^{\alpha+\beta})v= (T^{\alpha}T^{\beta}T^{-\beta}T^{-\alpha})T^{\alpha+\beta}v = T^{\alpha+\beta}v.\]
The other cases follow in a similar way.

\end{proof}

With these definitions it is possible to establish a theory of interpolation spaces for strongly continuous quaternionic semigroups analogue to the one for complex operator semigroups. Since the proofs follow the lines of this classical case, we only state the main result and refer to
\cite[Chapter II]{EngelNagel} for an overview on the theory.

\begin{definition}
Let $(\U(t))_{t\geq 0}$ be a strongly continuous semigroup with growth bound $\omega_0<0$. For each $\alpha\in(0,1]$ we define the Favard space
\[F_\alpha:=\left\{v\in V: \sup_{t>0}\left\| \frac{1}{t^{\alpha}}(\U(t)v-v) \right\|<\infty \right\}\]
with the norm
\[ \|v\|_{F_{\alpha}} := \sup_{t>0}\left\| \frac{1}{t^{\alpha}}(\U(t)v-v) \right\|.\]
The subspace
\[X_\alpha:= \|v\|_{F_{\alpha}} := \sup_{t>0}\left\| \frac{1}{t^{\alpha}}(\U(t)v-v) \right\|\]  of $F_\alpha$ is called the abstract H\"{o}lder space of order $\alpha$.
\end{definition}

\begin{Pn}
Let $A$ be the generator of a strongly continuous semigroup $(\U(t))_{t\geq 0}$ with growth bound $w_0 <0$ and let $\alpha,\beta\in(0,1)$ such that $\alpha>\beta$. Then
\[X_{\alpha}\hookrightarrow\dom((-A)^{-\alpha})\hookrightarrow X_\beta,\]
where $\hookrightarrow$ denotes a continuous embedding.
\end{Pn}

We point out that in contrast to the classical case discussed in \cite{EngelNagel} the interpolation spaces must be defined using the powers of $-A$ instead of $A$. This is due to the following fact: in the complex setting one may choose a branch of the function $z\to z^{-\alpha}$ that is defined and holomorphic everywhere except for the positive real axis.  Since the spectrum of the operator is then contained in the domain of holomorphicity of $z^{-\alpha}$ by assumption, one can use this branch to define fractional powers of the operator and in turn the interpolation spaces. In the quaternionic setting this is however not possible because the logarithm and in turn the fractional powers $s\mapsto s^{-\alpha}$ are single-valued. They are defined and slice hyperholomorphic everywhere except for the negative real axis, which does not necessarily lie in $\rho_S(A)$ but in $\rho_S(-A)$ since $\Re(s) \leq w_0 < 0$ for all $s\in\sigma_S(T)$.

\section{Kato's formula and the generation of analytic semigroups}\label{KatoSect}
Kato showed in \cite{Kato} that certain fractional powers of generators of analytic semigroups are again generators of analytic semigroups. Analogue results can be shown for quaternionic linear operators, but therefore we need a modified definition of fractional powers of an operator.
\begin{definition}\label{DefiType}
A densely defined closed operator $T$ is of type $(M,\omega)$ with $M>0$ and $\omega\in(0,\pi)$ if
\begin{enumerate}[(i)]
\item the open sector $\osector{\omega}{0}$ is contained in the $S$-resolvent set of $T$ and $\|s S_R^{-1}(s,T)\|$ is uniformly bounded on any smaller sector $\osector{\theta}{0}$ with $\theta\in(\omega,\pi)$ and
\item $M$ is the uniform bound of $\|t S_R^{-1}(-t,T)\|$ on the negative real axis, that is
\begin{equation}\label{ResEst2}
\|tS_R^{-1}(t,T)\| \leq M, \qquad\text{for }t\in (-\infty,0).
\end{equation}
\end{enumerate}
\end{definition}
Note that this definition is different from the notion of sectorial operators used in \cite{EngelNagel}. Moreover, note that if \eqref{ResEst2} holds true then, as in Lemma~\ref{StripCy}, the power series expansion of the right $S$-resolvent implies the existence of a sector $\osector{\theta_0}{0}$ and a constant $M_1>0$ such that
$$
\|S_R^{-1}(s,T)\| \leq M_1/|s|\ \ \ \rm  for\ \ \ s\in\osector{\theta_0}{0}.
$$
 Hence, \eqref{ResEst2} is sufficient for $T$ to be an operator of type $(M,\omega)$ for some $\omega\in(0,\pi)$. In particular any operator that satisfies \eqref{ResEst} is of type $(M,\omega)$ with $\omega\leq\theta_0$, cf. Lemma~\ref{StripCy}.
\begin{Pn}\label{Kato}
Let $T$ be of type $(M,\omega)$ with $M>0$ and $\omega\in(0,\pi)$. Let $0<\alpha <1$ and let $\pi > \phi_0 > \max(\alpha\pi, \omega)$. The parameter integral
\begin{equation}\label{KatoEQ}
F_{\alpha}(p,T) = \frac{\sin(\alpha\pi)}{\pi}\int_0^{+\infty}t^{\alpha}(p^2-2pt^{\alpha}\cos(\alpha\pi)+t^{2\alpha})^{-1} S_R^{-1}(-t,T)\,dt.\end{equation}
defines a $\boundOP(V)$-valued function on $\osector{\phi_0}{0}$ in $p$ that is left slice hyperholomorphic.
\end{Pn}
\begin{proof}
For any compact subset $K$ of $\osector{\phi_0}{0}$, we have $\min_{p\in K}\arg(p)>\alpha\pi$ and thus there exists some $\delta_K > 0$ such that
\begin{equation}\label{DeltaK}
\left| p^2 - 2pt^{\alpha}\cos(\alpha\pi) + t^{2\alpha} \right| = \left| p-t^{\alpha}e^{I_p\alpha\pi}\right|\left| p-t^{\alpha}e^{-I_p\alpha\pi}\right| \geq \delta_K
\end{equation}
for $p\in K$ and $t\geq 0$. For the same reason, we can find a constant $C_K>0$ such that
\[
 \sup_{t\in[0,+\infty)\atop p\in K} \left| p^2 - 2pt^{\alpha}\cos(\alpha\pi) + t^{2\alpha} \right| ^{-1} t^{2\alpha}=  \sup_{t\in[0,+\infty)\atop p\in K}\frac{1}{\left| \frac{p}{t^{\alpha}}-e^{I_p\alpha\pi}\right|} \frac{1}{\left| \frac{p}{t^{\alpha}}-e^{-I_p\alpha\pi}
\right|} < C_K
\]
and hence
\begin{equation}\label{CK}
\left| p^2 - 2pt^{\alpha}\cos(\alpha\pi) + t^{2\alpha} \right|^{-1} \leq  C_K t^{-2\alpha} \quad t\in[1,\infty),\, p\in K.
\end{equation}

Now consider  $p\in\osector{\phi_0}{0}$ and let $K$ be a compact neighborhood of $p$. The integral in \eqref{KatoEQ} converges absolutely and hence defines a bounded operator: because of \eqref{ResEst2} and the above estimates, we have for $s\in K$, and thus in particular for $p$ itself, that
\begin{align*}
\left\| F_{\alpha}(s,T) \right\|\leq& \frac{\sin(\alpha\pi)}{\pi}\int_0^{+\infty}t^{\alpha}\left|s^2-2st^{\alpha}\cos(\alpha\pi)+t^{2\alpha}\right|^{-1} \frac{M}{t}\,dt \\
\leq& \frac{M\sin(\alpha\pi)}{\delta_K\pi}\int_0^{1}t^{\alpha-1}\,dt + \frac{M\sin(\alpha\pi)C_K}{\pi}\int_1^{+\infty}t^{-\alpha-1} \,dt < +\infty.
\end{align*}
Using \eqref{DeltaK} and \eqref{CK}, one can derive analogous estimates for the partial derivatives of the integrand $p\mapsto t^{\alpha}(p^2-2st^{\alpha}\cos(\alpha\pi)+t^{2\alpha})^{-1} S_R^{-1}(-t,T)$ with respect to $p_0$ and $p_1$.

Since these estimates are uniform on the neighborhood $K$ of $p$, we can exchange differentiation and integration in order to compute the partial derivatives $\frac{\partial}{\partial p_0}F_{\alpha}(p,T)$ and $\frac{\partial}{\partial p_1}F_{\alpha}(p,T)$ of $F_{\alpha}(\cdot,T)$ at $p$. The integrand is however left slice hyperholomorphic and therefore also $F_{\alpha}(p,T)$ is left slice hyperholomorphic.

\end{proof}

\begin{lemma}\label{SomeLemma213}
Let $T$ be of type $(M,\omega)$ with $M>0$, let $0<\alpha <1$ and $\omega\in(0,\pi)$ and assume that $0\in\rho_S(T)$. Moreover, let $\phi_0$ and $F_{\alpha}(p,T)$ be defined as in Proposition~\ref{Kato}. If $\Gamma$ is a piecewise smooth path that goes from $\infty e^{I\theta}$ to $\infty e^{-I\theta}$ in $(\osector{\phi_0}{0}\cap\C_I)\setminus(-\infty,0]$ for some $I\in\S$ and some $\theta\in(\phi_0,\pi]$, then
\begin{equation}\label{ShortInt7}
F_{\alpha}(p,T) = \frac{1}{2\pi}\int_{\Gamma} S_R^{-1}(p,s^{\alpha})\,ds_{I}\,S_R^{-1}(s,T).
\end{equation}
\end{lemma}
\begin{proof}
First of all observe that the function $s\mapsto S_R^{-1}(p,s^{\alpha})$ is the composition of the intrinsic function $s\mapsto s^{\alpha}$ defined on $\H\setminus(-\infty,0]$ and the right slice hyperholomorphic function $s\mapsto S_R^{-1}(p,s)$ defined on $\H\setminus[p]$. This composition is in particular well defined on all of $\H\setminus(-\infty,0]$, because $s^{\alpha}$ maps $\H\setminus(-\infty,0]$ to the set $\{s\in\H: \arg(s) < \alpha\pi\}$, which is contained in the domain of definition of $S_R^{-1}(p,s^{\alpha})$ because $\arg(p) > \phi_0 >\alpha \pi$ by assumption. By Corollary~\ref{lkjsWW}, the function  $s\mapsto S_R^{-1}(p,s^{\alpha})$ is therefore right slice hyperholomorphic on $\H\setminus(-\infty,0]$.

 An estimate similar to the one in the proof of Proposition~\ref{Kato} moreover assures that the integral in \eqref{ShortInt7} converges absolutely. It thus follows from Theorem~\ref{COP} that the value of the integral in \eqref{ShortInt7} is the same for any choice of $\Gamma$ in $(\osector{\phi_0}{0}\cap\C_I)\setminus(-\infty,0]$ and any choice of~$\theta$. Let us denote the value of this integral by $\mathfrak{I}_{\alpha}(p,T)$.

Since $0\in\rho_S(T)$, the open ball $B(\varepsilon,0)$ is contained in $\rho_S(T)$ if $\varepsilon>0$ is small enough. For ${\theta\in(\phi_0,\pi)}$, we set $U(\varepsilon,\theta)=\H\setminus(\csector{\theta}{0}\cup \overline{B(\varepsilon,0)})$. Then
\begin{equation*}%\label{ShortInt7a}
\mathfrak{I}_{\alpha}(p,T) = \frac{1}{2\pi}\int_{\partial(U(\varepsilon,\theta)\cap\C_ {I})} S_R^{-1}(p,s^{\alpha})\,ds_{I}\,S_R^{-1}(s,T).
\end{equation*}
We assumed that $0\in\rho_S(T)$, and hence the  right $S$-resolvent is bounded near $0$, which allows us to  take the limit $\varepsilon\to 0$. We obtain
 \begin{align*}\label{FSecInt}
 \mathfrak{J}_{\alpha}(p,T) =& \frac{1}{2\pi}\int_{-\partial(\osector{\theta}{0}\cap\C_ {I})} S_R^{-1}(p,s^{\alpha})\,ds_{I}\,S_R^{-1}(s,T)\\
 =&-\frac{1}{2\pi} \int_{0}^{+\infty} S_{R}^{-1}\left(p,t^\alpha e^{I\alpha\theta}\right)e^{I\theta}(-I)S_R^{-1}\left(te^{I\theta},T\right)\,dt \\
 &+ \frac{1}{2\pi}\int_{0}^{+\infty}S_R^{-1}\left(p,t^{\alpha}e^{-I\alpha\theta}\right)e^{-I\theta}(-I)S_R^{-1}\left(te^{I\theta},T\right)\,dt \displaybreak[1]\\
 =&-\frac{1}{2\pi} \int_{0}^{+\infty} \left(p^2 - 2t^{\alpha}\cos(\alpha\theta) + t^{2\alpha}\right)^{-1}\left(p-t^{\alpha}e^{-I\alpha\theta}\right) e^{I\theta}(-I)S_R^{-1}\left(te^{I\theta},T\right)\,dt \\
 &+ \frac{1}{2\pi}\int_{0}^{+\infty} \left(p^2 - 2t^{\alpha}\cos(\alpha\theta) + t^{2\alpha}\right)^{-1}\left(p-t^{\alpha}e^{I\alpha\theta}\right)  e^{-I\theta}(-I)S_R^{-1}\left(te^{-I\theta},T\right)\,dt.
 \end{align*}
Again an estimate analogue to the one in the proof of Proposition~\ref{Kato} allows us to take the limit as $\theta$ tends to $\pi$ and we obtain
\begin{align*}
\mathfrak{J}_{\alpha}(p,T)  =&-\frac{1}{2\pi} \int_{0}^{+\infty} \left(p^2 - 2t^{\alpha}\cos(\alpha\pi) + t^{2\alpha}\right)^{-1}\left(p-t^{\alpha}e^{-I\alpha\pi}\right) e^{I\pi}(-I)S_R^{-1}\left(te^{I\pi},T\right)\,dt \\
 &+ \frac{1}{2\pi}\int_{0}^{+\infty} \left(p^2 - 2t^{\alpha}\cos(\alpha\pi) + t^{2\alpha}\right)^{-1}\left(p-t^{\alpha}e^{I\alpha\pi}\right)  e^{-I\pi}(-I)S_R^{-1}\left(te^{-I\pi},T\right)\,dt\\
 =&  \frac{\sin(\alpha\pi)}{\pi}\int_0^{+\infty}t^{\alpha}(p^2-2pt^{\alpha}\cos(\alpha\pi)+t^{2\alpha})^{-1} S_R^{-1}(-t,T)\,dt = F_{\alpha}(p,T).
\end{align*}

\end{proof}

\begin{lemma}
Let $T$ be of type $(M,\omega)$ with $M>0$  and $\omega\in(0,\pi)$. Let $0<\alpha <1$ and let $\phi_0$ and  $F_{\alpha}(p,T)$ be defined as in Proposition~\ref{Kato}. We have
\begin{equation}\label{RealREQ}
F_{\alpha}(\mu, T) - F_{\alpha}(\lambda,T) = (\lambda-\mu)F_{\alpha}(\mu, T)F_{\alpha}(\lambda,T) \quad\text{for }\lambda,\mu\in(-\infty,0].
\end{equation}
\end{lemma}
\begin{proof}
Assume first that $0\in\rho_S(T)$. Any real $\lambda$ commutes with $S_R^{-1}(-t,T)$ and thus we have
\[F_{\alpha}(\lambda,T) = \frac{\sin(\alpha\pi)}{\pi}\int_{0}^{+\infty} S_L^{-1}(-t,T)(\lambda^2-2\lambda t^{\alpha}\cos(\alpha\pi) + t^{2\alpha})^{-1}t^{\alpha}\,dt\]
because $S_R^{-1}(-t,T) = (-t\id-T)^{-1} = S_L^{-1}(-t,T)$ as $t$ is also real. Computations analogue to those in the proof of Lemma~\ref{SomeLemma213} show that $F_\alpha(\lambda,T)$ can thus be represented as
\begin{equation}\label{PPUL}
 F_{\alpha}(\lambda,T) = \frac{1}{2\pi}\int_{\Gamma} S_L^{-1}(s,T)\,ds_I\,S_L^{-1}(\lambda,s^{\alpha}),
\end{equation}
where $\Gamma$ is any path as in Lemma~\ref{SomeLemma213}.

Now let $\varepsilon>0$ such that $\overline{B(\varepsilon,0)}\subset\rho_S(T)$, choose $I\in\S$ and set
\[U_s := \H \setminus  \overline{\osector{\theta_s}{0} \cup B(\varepsilon_s,0)} \quad\text{and}\quad U_p := \H \setminus  \overline{\osector{\theta_p}{0} \cup B(\varepsilon_p,0)}\]
 with $0<\varepsilon_s<\varepsilon_p<\varepsilon$ and $\phi_0 < \theta_p < \theta_s <\pi$. Then $\overline{U_p}\subset U_s$ and $\Gamma_s = \partial( U_s\cap\C_I)$ and $\Gamma_p = \partial(U_p\cap\C_I)$ are paths as in Lemma~\ref{SomeLemma213}. Moreover, since $T$ is of type $(M,\omega)$  with $0\in\rho_S(T)$, we can finde a constant $C$ such that $\| S_R^{-1}(s,T)\| < C/(1+|s|)$ for $s\in (-\infty,0]$. By Lemma~\ref{StripCy} we may choose $\varepsilon_p$, $\varepsilon_s$, $\theta_p$ and $\theta_s$ such that
\begin{equation}\label{MMJMJ}
\|S_R^{-1}(s,T) \|\leq\frac{M_1}{1+|s|}, s\in\Gamma_s\quad\text{and}\quad\| S_L^{-1}(p,T)\|\leq \frac{M_1}{1+|p|}, p\in\Gamma_p
\end{equation}
for some constant $M_1>0$. Lemma~\ref{SomeLemma213} and \eqref{PPUL} then imply
\[
 F_{\alpha}(\mu,T)F_{\alpha}(\lambda,T) = \frac{ 1}{(2\pi)^2} \int_{\Gamma_s}\int_{\Gamma_p}S_R^{-1}(\mu,s^{\alpha})\,ds_I\,S_R^{-1}(s,T)S_L^{-1}(p,T)\,dp_I\,S_L^{-1}(\lambda,p^{\alpha}).
\]
Applying the $S$-resolvent equation \eqref{resEQ} yields
\begin{align*}
& F_{\alpha}(\mu,T)F_{\alpha}(\lambda,T) \\=& \frac{1}{(2\pi)^2} \int_{\Gamma_s}\int_{\Gamma_p}S_R^{-1}(\mu,s^{\alpha})\,ds_I\,S_R^{-1}(s,T)p(p^2-2s_0p+|s|^2)^{-1}\,dp_I\,S_L^{-1}(\lambda,p^{\alpha})\\
 &- \frac{ 1}{(2\pi)^2} \int_{\Gamma_s}\int_{\Gamma_p}S_R^{-1}(\mu,s^{\alpha})\,ds_I\,S_L^{-1}(p,T)p(p^2-2s_0p+|s|^2)^{-1}\,dp_I\,S_L^{-1}(\lambda,p^{\alpha})\\
 &- \frac{ 1}{(2\pi)^2} \int_{\Gamma_s}\int_{\Gamma_p}S_R^{-1}(\mu,s^{\alpha})\,ds_I\,\overline{s}S_R^{-1}(s,T)(p^2-2s_0p+|s|^2)^{-1}\,dp_I\,S_L^{-1}(\lambda,p^{\alpha})\\
&+ \frac{ 1}{(2\pi)^2} \int_{\Gamma_s}\int_{\Gamma_p}S_R^{-1}(\mu,s^{\alpha})\,ds_I\,\overline{s}S_L^{-1}(p,T)(p^2-2s_0p+|s|^2)^{-1}\,dp_I\,S_L^{-1}(\lambda,p^{\alpha}).
\end{align*}
The functions $p\mapsto (p^2-2s_0p+|s|^2)^{-1}S_L^{-1}(\lambda,p^{\alpha})$ and $p\mapsto p(p^2-2s_0p+|s|^2)^{-1}S_L^{-1}(\lambda,p^{\alpha})$ are holomorphic on $\overline{U_p\cap\C_I}$ and tend uniformly to zero as $p$ tends to infinity in $U_p$. We therefore deduce from Cauchy's integral theorem that the first and the third of the above integrals equal zero. The estimate \eqref{MMJMJ} allows us to apply Fubini's theorem in order to exchange the order of integration such that we are left with
\begin{multline}\label{FUFUFJ}
 F_{\alpha}(\mu,T)F_{\alpha}(\lambda,T)= \frac{ 1}{2\pi} \int_{\Gamma_p}\bigg[ \frac{ 1}{2\pi}\int_{\Gamma_s}S_R^{-1}(\mu,s^{\alpha})\,ds_I\\
 \cdot\left(\overline{s}S_L^{-1}(p,T)-S_L^{-1}(p,T)p\right)(p^2-2s_0p+|s|^2)^{-1}\bigg]\,dp_I\,S_L^{-1}(\lambda,p^{\alpha}).
\end{multline}
We want to apply Lemma~\ref{Lemma321} and thus define the set $U_{s,r} := U_{s}\cap B(r,0)$ for $r>0$, which is a bounded axially symmetric slice domain. Its boundary $\partial(U_{s,r}\cap\C_I)$ in $\C_I$ consists of $\Gamma_{s,r}:= \Gamma_s\cap B(r,0)$ and the set $C_{r} := \{ re^{I\varphi}: -\theta_s\leq \varphi \leq \theta_s\}$. If $p\in\Gamma_p$, then $p\in\ U_{s,r}$ for sufficiently large $r$ because  $\overline{U_p}\subset U_s$. Since the function $s\mapsto S_R^{-1}(\mu, s^{\alpha}) = (\mu-s^{\alpha})^{-1}$ is intrinsic because $\mu$ is real, we can therefore apply Lemma~\ref{Lemma321} and obtain for any such $r$
\begin{align*}
&S_L^{-1}(p,T)S_R^{-1}(\mu,p^{\alpha})\\
 =& \frac{1}{2\pi}\int_{\partial (U_{s,r}\cap\C_I)} S_R^{-1}(\mu,s^{\alpha})\,ds_I\left(\overline{s}S_L^{-1}(p,T)-S_L^{-1}(p,T)p\right)(p^2-2s_0p+|s|^2)^{-1}\displaybreak[2]\\
  =& \frac{1}{2\pi}\int_{\Gamma_{s,r}} S_R^{-1}(\mu,s^{\alpha})\,ds_I\left(\overline{s}S_L^{-1}(p,T)-S_L^{-1}(p,T)p\right)(p^2-2s_0p+|s|^2)^{-1}\\
   &+ \frac{1}{2\pi}\int_{C_r} S_R^{-1}(\mu,s^{\alpha})\,ds_I\left(\overline{s}S_L^{-1}(p,T)-S_L^{-1}(p,T)p\right)(p^2-2s_0p+|s|^2)^{-1}.
\end{align*}
As $r$ tends to infinity the integral over $C_r$ vanishes and hence
\begin{align*}
&S_L^{-1}(p,T)S_R^{-1}(\mu,p^{\alpha})\\
   =& \lim_{r\to+\infty}\frac{1}{2\pi}\int_{\Gamma_{s,r}} S_R^{-1}(\mu,s^{\alpha})\,ds_I\left(\overline{s}S_L^{-1}(p,T)-S_L^{-1}(p,T)p\right)(p^2-2s_0p+|s|^2)^{-1}\\
   =&\frac{1}{2\pi}\int_{\Gamma_{s}} S_R^{-1}(\mu,s^{\alpha})\,ds_I\left(\overline{s}S_L^{-1}(p,T)-S_L^{-1}(p,T)p\right)(p^2-2s_0p+|s|^2)^{-1}.
\end{align*}
Applying this identity in \eqref{FUFUFJ}, we obtain
\[  F_{\alpha}(\mu,T)F_{\alpha}(\lambda,T)= \frac{ 1}{2\pi} \int_{\Gamma_p}S_L^{-1}(p,T)\,dp_I\,S_R^{-1}(\mu,p^{\alpha})S_L^{-1}(\lambda,p^{\alpha})\]
because $S_R^{-1}(\mu,p^{\alpha})$ and $dp_I$ commute as $\mu$ is real.
Since also  $\lambda$ is real, we have
\begin{gather*}
S_R^{-1}(\mu,p^{\alpha})S_L^{-1}(\lambda,p^{\alpha}) = \frac1{\mu-p^{\alpha}}\frac{1}{\lambda-p^{\alpha}} \\
= \frac{1}{\lambda-\mu}\left(\frac{1}{\mu-p^{\alpha}}-\frac{1}{\lambda-p^{\alpha}}\right) = (\lambda-\mu)^{-1}\left(S_L^{-1}(\mu,p^{\alpha})-S_L^{-1}(\lambda,p^{\alpha})\right)
\end{gather*}
and thus, recalling \eqref{PPUL}, we obtain
\begin{align*}
  &F_{\alpha}(\mu,T)F_{\alpha}(\lambda,T)\\
  =& (\lambda-\mu)^{-1}\left(\frac{ 1}{2\pi} \int_{\Gamma_p}S_L^{-1}(p,T)\,dp_I\,S_L^{-1}(\mu,p^{\alpha}) - \frac{ 1}{2\pi} \int_{\Gamma_p}S_L^{-1}(p,T)\,dp_I\,S_L^{-1}(\lambda,p^{\alpha})\right)\\
    =&(\lambda-\mu)^{-1}\left(F_{\alpha}(\mu,T) - \F_{\alpha}(\lambda,T)\right).
  \end{align*}

  If $0\notin \rho_S(T)$, then we consider the operator $T+\varepsilon\id$ for small $\varepsilon > 0$. This operator satisfies $0\in\rho_S(T+\varepsilon\id) = \rho_S(T)+\varepsilon$ and hence $\eqref{RealREQ}$ applies.
 Moreover, for real $t$, we have
\[
 S_R^{-1}(-t,T+\varepsilon\id) = S_R^{-1}(-(t+\varepsilon),T).
\]
The estimate
\[
  \|S_R^{-1}(-t,T+\varepsilon\id)\| \leq \frac{M}{t+\varepsilon} \leq \frac{M}{t}
 \]
  therefore allows us to apply Lebesgue's dominated convergence theorem to see that
   \begin{gather*}
   F_\alpha(p,T+\varepsilon\id) = \frac{\sin(\alpha\pi)}{\pi}\int_0^{+\infty}t^{\alpha}(p^2-2pt^{\alpha}\cos(\alpha\pi)+t^{2\alpha})^{-1} S_R^{-1}(-t,T+\varepsilon\id)\,dt\\
   \overset{\varepsilon\to 0}{\longrightarrow} \frac{\sin(\alpha\pi)}{\pi}\int_0^{+\infty}t^{\alpha}(p^2-2pt^{\alpha}\cos(\alpha\pi)+t^{2\alpha})^{-1} S_R^{-1}(-t,T)\,dt = F_{\alpha}(p,T).
    \end{gather*}
Consequently, we have
\begin{multline*}
F_{\alpha}(\mu, T) - F_{\alpha}(\lambda,T) = \lim_{\varepsilon\to 0} F_{\alpha}(\mu, T+\varepsilon\id) - F_{\alpha}(\lambda,T+\varepsilon\id)\\
 = \lim_{\varepsilon\to 0} (\lambda-\mu)F_{\alpha}(\mu, T+\varepsilon\id)F_{\alpha}(\lambda,T+\varepsilon\id)
= (\lambda-\mu)F_{\alpha}(\mu, T)F_{\alpha}(\lambda,T)
\end{multline*}
for $\lambda,\mu\in(-\infty,0]$ also in this case.

  \end{proof}

\begin{theorem}\label{BalphaExists}Let $T$ be of type $(M,\omega)$, let $\alpha\in(0,1)$ and let $\phi_0 > \max(\alpha\pi, \omega)$. There exists a densely defined closed operator $B_{\alpha}$ such that
\[S_R^{-1}(p,B_\alpha) = F_{\alpha}(p,T)\qquad\text{for } \ \ \ p\in\osector{\phi_0}{0},\]
where $F_{\alpha}(p,T)$ is the operator-valued function defined by the integral \eqref{ShortInt7}. Moreover, $B_{\alpha}$ is of type $(M,\alpha\omega)$.
\end{theorem}
 \begin{proof}
From identity \eqref{RealREQ} it follows immediately that $F_\alpha(\mu,T)$ and $F_\alpha(\lambda,T)$ commute and have the same kernel. Rewriting this equation in the form
\begin{equation}\label{RealREQ2}
F_{\alpha}(\mu, T) =  F_{\alpha}(\lambda,T)\left(\id +(\lambda-\mu)F_{\alpha}(\mu, T)\right)
\end{equation}
shows that $\ran(F_\alpha(\mu,T)) \subset \ran(F_\alpha(\lambda,T))$ and exchanging the roles of $\mu$ and $\lambda$ yields \linebreak[4]$\ran(F_\alpha(\mu,T)) = \ran(F_\alpha(\lambda,T))$. Hence, $\ran(F_\alpha(\mu,T))$ does not depend on $\mu$ and so we denote it by $\ran(F_\alpha(\bcdot,T))$.

We show now that
\begin{equation}\label{LimID}\lim_{\rr\ni\mu\to-\infty} \mu F_\alpha(\mu,T) v = v\quad\text{for all }v\in V,\end{equation}
 where $\lim_{\rr\ni\mu\to-\infty} \mu F_\alpha(\mu,T) v$ denotes the limit as $\mu$ tends to $-\infty$ in $\rr$. From \eqref{LimID}, we easily deduce that $\ran(F_\alpha(\bcdot,T))$ is dense in $V$ because
\[ V = \overline{\bigcup_{\mu\in(-\infty,0]}\ran(F_\alpha(\mu,T))} = \overline{\ran(F_\alpha(\bcdot,T))}.\]
We consider first $v\in\dom(T)$. Since
\begin{equation}\label{FraInt}
\int_0^{+\infty} \frac{t^{\alpha-1}}{\mu^2 - 2\mu t^{\alpha}\cos(\alpha\pi) + t^{2\alpha}}\, dt = -\frac{\pi}{\mu\sin(\alpha\pi)}\quad  \text{for }\mu\leq 0,
\end{equation}
it is
\begin{align*}
\mu F_{\alpha}(\mu,T)v - v& =- \frac{\sin(\alpha\pi)}{\pi}\int_0^{+\infty} \frac{-\mu t^{\alpha-1}}{\mu^2-2\mu t^{\alpha}\cos(\alpha\pi) + t^{2\alpha}}\left(tS_R^{-1}(-t,T)v + v\right)\,dt.
\end{align*}
For $-\mu \geq 1$ and $t\in(0,\infty)$, we can estimate
\[\frac{-\mu t^{\alpha-1}}{\mu^2 - 2\mu t^{\alpha}\cos(\alpha\pi) + t^{2\alpha}} = \frac{-\mu t^{\alpha-1}}{\mu^2\sin(\alpha\pi)^2 + (\mu\cos(\alpha\pi) - t^{\alpha})^2} \leq \frac{-\mu t^{\alpha-1}}{\mu^2\sin(\alpha\pi)^2} \leq \frac{t^{\alpha-1}}{\sin(\alpha\pi)^2}\]
and due to \eqref{ResEst2} we have $\|tS_R^{-1}(-t,T)v + v\| \leq (M+1)\|v\|$. On the other hand, since $v\in\dom(T)$, it is
\begin{equation}\label{Res2ID}
tS_R^{-1}(-t,T)v +v = -S_R^{-1}(-t,T)Tv
\end{equation}
and in turn, again due to \eqref{ResEst2}, we can also estimate $ \|tS_R^{-1}(-t,T)v + v\| \leq \|Tv\|/t$ such that we can apply Lebesgue's dominated convergence theorem with dominating function
\[f(t) = \begin{dcases}\frac{K}{\sin(\alpha\pi)^2}t^{\alpha-1},&\text{for }t\in(0,1)\\\frac{K}{\sin(\alpha\pi)^2}t^{-\alpha-1},&\text{for }t\in[1,+\infty)\end{dcases},\]
with $K>0$ large enough, in order to exchange the integral with the limit for $\mu \to -\infty$ in $\rr$. In view of \eqref{Res2ID}, we obtain
\begin{multline*}\lim_{\rr\ni\mu\to-\infty}\mu F_{\alpha}(\mu, T)v - v  \\
=-\frac{\sin(\alpha\pi)}{\pi}\int_0^{+\infty} \lim_{\rr\ni\mu\to-\infty} \frac{\mu t^{\alpha-1}}{\mu^2-2\mu t^{\alpha}\cos(\alpha\pi) + t^{2\alpha}}S_R^{-1}(-t,T)Tv \,dt = 0. \end{multline*}
For arbitrary $v\in V$ and $\varepsilon>0$ consider a vector $v_{\varepsilon}\in\dom(T)$ with $\|v-v_\varepsilon\| < \varepsilon$. Because of \eqref{ResEst2} and \eqref{FraInt}, we have the uniform estimate
\begin{align}\label{KAR}
\left\|\mu F_\alpha(\mu,T)\right\| &\leq \frac{-\mu\sin(\alpha\pi)}{\pi}\int_{0}^{+\infty}\frac{ t^{\alpha}}{\mu^2-2\mu t^\alpha\cos(\alpha\pi) + t^{2\alpha}} \frac{M}{t} dt= M.
\end{align}
Therefore
\begin{align*}
 \lim_{\rr\ni\mu\to-\infty}\|\mu F_{\alpha}(\mu,T)v - v\| &\leq \lim_{\rr\ni\mu\to-\infty} \|\mu F_{\alpha}(\mu,T)\|\|v - v_{\varepsilon}\| + \|F_{\alpha}(\mu,T)v_{\varepsilon} - v_{\varepsilon}\|+\|v_{\varepsilon}-v\|\\
 &\leq(M+1)\varepsilon .
\end{align*}
Since $\varepsilon>0$ was arbitrary, we deduce that \eqref{LimID} also holds true for arbitrary $v\in V$.

Overall, we obtain that  $\ran(F_{\alpha}(\bcdot, T))$ is dense in $V$. The identity \eqref{LimID} moreover also implies $\ker(F_{\alpha}(\bcdot,T)) = \{0\}$ because  $ v = \lim_{\rr\ni\mu\to-\infty} F_{\alpha}(\mu,T)v = 0$ for $v\in\ker(F_{\alpha}(\bcdot,T))$.

We consider now an arbitrary point $\mu_0\in(-\infty,0)$. By the above arguments, the mapping $F_{\alpha}(\mu_0,T): V\to\ran(F_{\alpha}(\bcdot,T))$ is invertible. Hence, we can define the operator $B_\alpha := {\mu_0\id- F_{\alpha}(\mu_0,T)^{-1}}$ that maps $\dom(B_\alpha) = \ran(F_\alpha(\mu_0,T))$  to $V$. Apparently, $B_{\alpha}$ has dense domain and $S_{R}^{-1}(\mu_0,B_\alpha) = F_\alpha(\mu_0,B_\alpha)$. For $\mu\in(-\infty,0]$, we can apply \eqref{RealREQ2} and \eqref{RealREQ} to obtain
\begin{align*}
(\mu\id-B_\alpha)F_\alpha(\mu,T) &= ((\mu-\mu_0)\id + (\mu_0\id-B_\alpha))F_\alpha(\mu_0,T)(\id + (\mu_0-\mu) F_\alpha(\mu,T))\\
&= \id + (\mu-\mu_0)\big(F_\alpha(\mu_0,T) + (\mu_0-\mu)F_\alpha(\mu_0,T)F_\alpha(\mu,T)-F_\alpha(\mu,T)\big)=\id.
\end{align*}
 A similar calculation shows that ${F_\alpha(\mu,T)(\mu\id-B_\alpha)v = v}$ for all $v\in\dom(B_\alpha)$. We conclude that   $S_R^{-1}(\mu,B_\alpha) = F_\alpha(\mu,T)$ for any $\mu\in(-\infty,0)$. Since $p\mapsto F_\alpha(p,T)$ and $p\mapsto S_R^{-1}(p,B_\alpha)$ are left slice hyperholomorphic and agree on $(-\infty,0)$,   Theorem~\ref{IDPOP} implies $S_R^{-1}(p,B_\alpha) = F_\alpha(p,T)$  for any $p\in\osector{0}{\phi_0}$.

Finally, in order to show that $B_{\alpha}$ is of type $(M,\alpha\omega)$, we choose an arbitrary imaginary unit $I\in\S$ and consider the restriction of $S_{R}^{-1}(\bcdot,B_{\alpha})$ to the plane $\C_I$. This restriction is  a holomorphic function with values in the left-vector space $\boundOP(V)$ over $\C_I$. We show now that this restriction has a holomorphic continuation to the sector $\osector{\alpha\omega}{0} \cap\C_I$. Since this sector is symmetric with respect to the real axis, we can apply Corollary~\ref{extLemOP} and obtain a left slice hyperholomorphic continuation of $S_{R}^{-1}(p,B_{\alpha})$ to the sector $\osector{\alpha\omega}{0}$. By Theorem~\ref{SExt}, this implies in particular that $\osector{\alpha\omega}{0}\subset\rho_S(T)$.

 The above considerations showed that we can represent $S_{R}^{-1}(p,B_{\alpha})$ for $p\in\osector{\omega}{0}\cap\C_I$ using Kato's formula \eqref{KatoEQ}. Rewriting this formula as a path integral over the path $\gamma_{0}(t) = te^{I\pi}, t\in[0,+\infty)$, we obtain
\[ S_{R}^{-1}(p,B_{\alpha})= -\frac{\sin(\alpha\pi)}{\pi}\int_{\gamma_0} \frac{z^\alpha e^{-I\pi\alpha}}{(z^{\alpha}e^{-I2\alpha\pi }-p)(z^{\alpha}-p)} S_{R}^{-1}(z,T) \ dz,\]
where $z$ denotes a complex variable in $\C_I$ and  $z\mapsto z^{\alpha}$ is a branch of a complex $\alpha$-th power of $z$ that is holomorphic on $\C_I\setminus [0,\infty)$. To be more precise, let us choose  $\left(re^{I\theta}\right)^{\alpha} = r^\alpha e^{I\theta\alpha}$ with $\theta\in (0,2\pi)$. (This is however not the restriction of the quaternionic function $s\mapsto s^{\alpha}$ defined in \eqref{fracDef} to the plane $\C_I$, cf. Remark~\ref{LogRem}!)

Observe that for fixed $p$ the integrand is holomorphic on $D_0:=\osector{\omega}{0}\cap\C_I$. Hence, by applying Cauchy's Integral Theorem, we can exchange the path of integration $\gamma_0$ by a suitable path $\gamma_{\kappa}(t) = te^{I(\pi-\kappa)}, t\in[0,\infty)$ and obtain
\begin{equation}\label{KappaRes} S_{R}^{-1}(p,B_{\alpha})= -\frac{\sin(\alpha\pi)}{\pi}\int_{\gamma_\kappa} \frac{z^\alpha e^{-I\pi\alpha}}{(z^{\alpha}e^{-I2\alpha\pi }-p)(z^{\alpha}-p)} S_{R}^{-1}(z,T) \ dz.\end{equation}
On the other hand, for any $\kappa\in (-\omega, \omega)$, such integral defines a holomorphic function on the sector $D_\kappa := \{p\in\C_I: \alpha(\pi-\kappa) < \arg p < 2\pi - \alpha(\pi+\kappa)\}$, where the convergence of the integral is guaranteed because the operator $T$ is of type $(M,\omega)$. The above argument showed that this function coincides with $S_R^{-1}(p,B_{\alpha})$ on the common domain $D_0\cap D_{\kappa}$, and hence $p\mapsto S_{R}^{-1}(p,B_{\alpha})$ has a holomorphic continuation $F_I$ to
\[D = \bigcup_{\kappa\in(-\omega,\omega)} D_{\kappa} = \{p\in\C_I : \alpha(\pi-\kappa)< \arg_{\C_I}(p) < 2\pi - \alpha(\pi-\kappa)\}.\]
This set is symmetric with respect to the real axis and, as mentioned above,  we deduce from Corollary~\ref{extLemOP} that there exists a left slice hyperholomorphic continuation $F$ of $F_I$ to the axially symmetric hull $[D] = \osector{\alpha\omega}{ 0}$ of $D$. Consequently, $ \osector{\alpha\omega}{ 0}\subset \rho_S(B_{\alpha})$ and $F$ coincides with $S_{R}^{-1}(\bcdot, B_{\alpha})$ on $\osector{\alpha\omega}{0}$.

In order to show that $\|p S_{R}(p,B_{\alpha})\|$ is bounded on every sector $\osector{\theta}{0}$ with $\theta\in(\omega\alpha,0)$, we consider first a set
\[D_{\kappa,\delta} := \{p\in\C_I: \delta + \alpha(\pi-\kappa)  < \arg p < 2\pi - \alpha(\pi+\kappa) - \delta\}\]
with $\kappa\in(-\omega,\omega)$ and small $\delta>0$. For $p\in D_{\kappa,\delta}$ with $\phi = \arg_{\C_I}(p)\in (0,2\pi)$, we may represent $pS_{R}^{-1}(p,B_{\alpha})$ by means of \eqref{KappaRes} and estimate
\begin{gather*}
\|pS_{R}^{-1}(p,B_{\alpha})\| \leq  \frac{|p|\sin(\alpha\pi)}{\pi}\int_{0}^{+\infty} \frac{r^\alpha}{|(r^{\alpha}e^{-I(\pi+\kappa)\alpha}-p)(r^{\alpha}e^{I(\pi-\kappa)\alpha}-p)|} \|S_{R}^{-1}(re^{I(\pi-\kappa)},T)\| \ dr\\
= \frac{|p|\sin(\alpha\pi)}{\pi}\int_{0}^{+\infty} \frac{r^\alpha}{|(r^{\alpha}-|p|e^{I(\phi+(\pi+\kappa)\alpha)})(r^{\alpha}-|p|e^{I(\phi-(\pi-\kappa)\alpha)})|} \|S_{R}^{-1}(re^{I(\pi-\kappa)},T)\| \ dr.
\end{gather*}
The operator $T$ is of type $(\omega,M)$ and hence exists a constant $M_{\kappa}>0$ such that $\|S_{R}^{-1}(re^{I(\pi-\kappa)},T)\|\leq M/r$. Substituting $\tau = r^{\alpha}/|p|$ yields
\begin{align*}
\|pS_{R}^{-1}(p,B_{\alpha})\| \leq  \frac{\sin(\alpha\pi)}{\pi}\int_{0}^{+\infty} \frac{M_\kappa}{|(\tau-e^{I(\phi+(\pi+\kappa)\alpha)})(\tau-e^{I(\phi-(\pi-\kappa)\alpha)})|} \ d\tau.
\end{align*}
This integral is uniformly bounded for
\[\phi \in {(\delta + \alpha(\pi-\kappa) , 2\pi - \alpha(\pi+\kappa) - \delta)}\]
such that there exists a constant that depends only on $\kappa$ and $\delta$ such that
\[ \|pS(p,B_\alpha)\|\leq C(\kappa,\delta)\quad \text{for }p\in D_{\kappa,\delta}.\]
Now consider a sector $\osector{\theta}{0}$ with $\theta\in(\omega\alpha,\pi)$. Then there exist $(\kappa_i, \delta_i), i=1,\ldots,n$ such that $\osector{\theta}{0}\cap\C_I\subset \bigcup_{i=1}^{n}D_{\kappa_i,\delta_i}$ and hence
\[ \|pS_{R}^{-1}(p,B_{\alpha})\| \leq C := \max_{1\leq i\leq n} C(\kappa_i,\delta_i)\quad \text{for } p\in \osector{\theta}{0}\cap\C_I.\]
For arbitrary $p = p_0 + I_p p_1\in \osector{\theta}{0}\cap\C_I$, set $p_I = p_0 + I p_1$. Then the Representation Formula, Theorem~\ref{RepFoOP}, implies
\[ \|pS_R^{-1}(p,T)\|  \leq \frac12 \|(1-I_pI)p_IS_R^{-1}(p_I,B_{\alpha}) \|+ \frac12 \|(1+I_pI)\overline{p_I}S_R^{-1}(\overline{p_I},B_{\alpha})\| \leq 2 C.\]
Finally, the estimate $\|tS_R^{-1}(-t,B_{\alpha})\| \leq M/t$ follows immediately from \eqref{KAR}.

\end{proof}

\begin{definition}\label{TalphaKato}
Let $T\in\closOP(V)$ be of type $(\omega, M)$. For $\alpha\in(0,1)$ we define $T^{\alpha} := B_{\alpha}$.
\end{definition}

\begin{corollary}
Definition~\ref{TalphaKato} is consistent with Definition~\ref{TPos}.
\end{corollary}
\begin{proof}
Let $T\in\closOP(V)$, let $\alpha\in(0,1)$ and let $T^{\alpha}$ be the operator obtained from Definition~\ref{TalphaKato}. If $\|S_{R}^{-1}(s,T)\|\leq K/(1+|s|)$ for $s\in(-\infty,0]$, then we can apply Lebesgue's dominated convergence theorem in order to pass to the limit as $p$ tends to $0$ in Kato's formula \eqref{KatoEQ} for the right $S$-resolvent of $T^{\alpha}$. We obtain
\[ (T^{\alpha})^{-1} =- S_R^{-1}(0,T^{\alpha}) = -\frac{\sin(\alpha\pi)}{\pi}\int_{0}^{+\infty}t^{-\alpha}S_R^{-1}(-t,T),\,dt = T^{-\alpha}\]
where the last equality follows from Corollary~\ref{novesei}.

\end{proof}

As an immediate consequence of \cite[Theorem 5.6]{GR} and Theorem~\ref{BalphaExists}, we obtain the following Corollary.
\begin{corollary}
Let $T\in\closOP(V)$ be of type $(\omega,M)$. If $\alpha\in(0,1)$ with $\alpha\omega<\pi/2$, then $-T^{\alpha}$ is the infinitesimal generator of a strongly continuous semigroup that is analytic in time.
\end{corollary}
\appendix
\section{Estimate for applying  Fubini's Theorem in \eqref{BeforeFubini}}\label{Fubini1}
We want to show that we can apply Fubini's theorem to exchange the order of integration in \eqref{BeforeFubini}, i.e. in
\begin{align}\label{APXBeforeFubini}
 T^{-\alpha} = \frac{1}{(2\pi)^2}\int_{\partial(U_s\cap\C_I)}\left(\int_{\partial(U_p\cap\C_J)}p^{-\alpha}\, dp_J\, S_R^{-1}(p,s)\right)\, ds_I\, S_R^{-1}(s,T).
\end{align}
In order to show that the integrand is absolutely integrable, we consider the parameterizations $\Gamma_s$ and $\Gamma_p$ of $\partial(U_s\cap\C_I)$ and $\partial(U_p\cap\C_J)$ that are given by
\begin{equation*}
\Gamma_s(r) = \begin{cases} \Gamma_s^+(r):= re^{-\theta_s I},& r\in[a_0/2,+\infty)\\
\Gamma_s^0(r):=\frac{a_0}{2}e^{-\frac{2\theta_s}{a_0} Ir} & r\in(-a_0/2,a_0/2)\\
\Gamma_s^-(r):=re^{\theta_s I},& r\in(-\infty, -a_0/2]\\
 \end{cases}
\end{equation*}
and
\begin{equation*}
\Gamma_p(t) = \begin{cases} \Gamma_p^{+}(t):= te^{-\theta_p J},& t\in[a_0/3,+\infty)\\
\Gamma_p^0(t):=\frac{a_0}{3}e^{-\frac{3\theta_p}{a_0}J t} & t\in(-a_0/3,a_0/3)\\
\Gamma_p^{-}(t):=te^{\theta_p J},& t\in(-\infty, -a_0/3]
 \end{cases}.
\end{equation*}
Then
\begin{gather}
\notag
\frac{1}{(2\pi)^2}\int_{\Gamma_s}\int_{\Gamma_p}\left\|p^{-\alpha}\, dp_J\, S_R^{-1}(p,s)\, ds_I\, S_R^{-1}(s,T)\right\|\\
\label{APXGammaSplit}= \sum_{\tau,\nu\in\{-,0,+\}} \int_{\Gamma_s^\tau}\int_{\Gamma_p^\nu}\left\|p^{-\alpha}\, dp_J\, S_R^{-1}(p,s)\, ds_I\, S_R^{-1}(s,T)\right\|
\end{gather}
and it is sufficient to estimate each of the terms in the sum separately. Applying Theorem~\ref{RepFo} allows us to estimate
\begin{equation}
\label{APXRSEst}
|S_R^{-1}(p,s)| \leq \frac12 |1-IJ|\frac{1}{|p_I -s|}+ \frac12|1+IJ|\frac{1}{|\overline{p_I} -s|} \leq  \frac{2}{|p_I-s_I|},
\end{equation}
where $p_I = p_0 + I p_1 $ for $p = p_0  I_pp_1$. Hence, for $\tau,\nu\in\{+,-\}$, by applying~\eqref{StripEstimate}, we have
\begin{align*}
&\int_{\Gamma_s^\tau}\int_{\Gamma_p^\nu}\left\|p^{-\alpha}\, dp_J\, S_R^{-1}(p,s)\, ds_I\, S_R^{-1}(s,T)\right\|\\
=&\int_{\frac{a_0}{2}}^{+\infty}\int_{\frac{a_0}{3}}^{+\infty} t^{-\alpha} \frac{2}{\left|te^{\theta_pI}-re^{\theta_sI}\right|}\frac{M_1}{1+r}\,dt\,dr=\int_{\frac{a_0}{2}}^{+\infty}\int_{\frac{a_0}{3}}^{+\infty}  \frac{t^{-\alpha}}{\left|t-re^{(\theta_s-\theta_p)I}\right|}\frac{2M_1}{1+r}\,dt\,dr\\
=&\int_{\frac{a_0}{2}}^{+\infty}\int_{\frac{a_0}{3}}^{+\infty}  \frac{\left(\frac{t}{r}\right)^{-\alpha}}{\left|\frac{t}{r}-e^{(\theta_s-\theta_p)I}\right|}\,dt\,\frac{1}{r^{1+\alpha}}\frac{2M_1}{1+r}\,dr = \int_{\frac{a_0}{2}}^{+\infty}\int_{\frac{a_0}{3r}}^{+\infty}  \frac{\mu^{-\alpha}}{\left|\mu-e^{(\theta_s-\theta_p)I}\right|}\,d\mu\,\frac{1}{r^{\alpha}}\frac{2M_1}{1+r}\,dr.
\end{align*}
The modulus of $\mu-e^{(\theta_s-\theta_p)I}$ can be estimated from below by the absolute value of its real part or by the absolute value of its imaginary part, and therefore,
\begin{align*}
&\int_{\Gamma_s^\tau}\int_{\Gamma_p^\nu}\left\|p^{-\alpha}\, dp_J\, S_R^{-1}(p,s)\, ds_I\, S_R^{-1}(s,T)\right\|  \\
\leq& \int_{\frac{a_0}{2}}^{+\infty}\int_{2}^{+\infty}  \frac{\mu^{-\alpha}}{\mu-\cos(\theta_s-\theta_p)}\,d\mu\,\frac{1}{r^{\alpha}}\frac{2M_1}{1+r}\,dr  +
 \int_{\frac{a_0}{2}}^{+\infty}\int_{\frac{a_0}{3r}}^{2}  \frac{\mu^{-\alpha}}{\sin(\theta_p-\theta_s)}\,d\mu\,\frac{1}{r^{\alpha}}\frac{2M_1}{1+r}\,dr\\
=& \underbrace{\int_{2}^{+\infty} \frac{\mu^{-\alpha}}{\mu-\cos(\theta_s-\theta_p)}\,d\mu}_{=:C_1 < +\infty}
\underbrace{\int_{\frac{a_0}{2}}^{+\infty} \frac{1}{r^{\alpha}}\frac{2M_1}{1+r}\,dr}_{=:C_2<+\infty} \\
&\qquad\qquad\qquad\quad\quad
 +\int_{\frac{a_0}{2}}^{+\infty} \frac{1}{\sin(\theta_p-\theta_s)}\left(\frac{2^{1-\alpha}}{1-\alpha} -
\frac{a_0^{1-\alpha}}{(1-\alpha)3^{1-\alpha}}\frac{1}{r^{1-\alpha}}\right)
\frac{1}{r^{\alpha}}\frac{2M_1}{1+r}\,dr\\
&= C_1C_2  +
\frac{2M_1}{\sin(\theta_p-\theta_s)}\frac{2^{1-\alpha}}{1-\alpha}
\int_{\frac{a_0}{2}}^{+\infty}
\frac{r^{-\alpha}}{1+r}\,dr
\\
&\qquad\qquad\qquad\qquad+
\frac{2M_1}{\sin(\theta_p-\theta_s)}\frac{a_0^{1-\alpha}}{(1-\alpha)3^{1-\alpha}}
\int_{\frac{a_0}{2}}^{+\infty}\frac{1}{r(1+r)}\, d\mu\,dr
\end{align*}
where each of these integrals is finite.

 For $\tau = 0$ and $\nu = +$, we can again use \eqref{APXRSEst} to estimate
\begin{align*}
&\int_{\Gamma_s^0}\int_{\Gamma_p^+}\left\|p^{-\alpha}\, dp_J\, S_R^{-1}(p,s)\, ds_I\, S_R^{-1}(s,T)\right\|\\
\leq&\int_{-\frac{a_0}{2}}^{\frac{a_0}{2}}\int_{\frac{a_0}{3}}^{+\infty} t^{-\alpha} \left|S_R^{-1}\left(te^{J\theta_p}, \frac{a_0}{2}e^{-\frac{2\theta_s}{a_0}rI}\right)\right|\frac{\theta_s M_1}{1+\frac{a_0}{2}}\,dt\,dr\\
\leq&\int_{-\frac{a_0}{2}}^{\frac{a_0}{2}}\int_{\frac{a_0}{3}}^{+\infty} t^{-\alpha} \frac{2}{\left|te^{I\theta_p} - \frac{a_0}{2}e^{\frac{2\theta_s}{a_0}rI}\right|}\frac{\theta_s M_1}{1+\frac{a_0}{2}}\,dt\,dr \\
=&\frac{2\theta_s M_1}{1+\frac{a_0}{2}}
\int_{-\frac{a_0}{2}}^{\frac{a_0}{2}}\int_{\frac{a_0}{3}}^{+\infty} \frac{t^{-\alpha}}{\left|t - \frac{a_0}{2}e^{\left(\frac{2\theta_s}{a_0}r-\theta_p\right)I}\right|}\,dt\,dr.
\end{align*}
Since $0 < \theta_s < \theta_p < \pi$, the distance $\delta$ of the set
\[\left\{\frac{a_0}{2}e^{\left(\frac{2\theta_s}{a_0}r-\theta_p\right)I}: -a_0/2 < r < a_0/2\right\}\]
 to the positive real axis is greater than zero, and hence,
\begin{align*}
&\int_{\Gamma_s^0}\int_{\Gamma_p^+}\left\|p^{-\alpha}\, dp_J\, S_R^{-1}(p,s)\, ds_I\, S_R^{-1}(s,T)\right\|\\
\leq&\frac{2\theta_s M_1}{1+\frac{a_0}{2}}
\int_{-\frac{a_0}{2}}^{\frac{a_0}{2}}\int_{\frac{a_0}{3}}^{a_0}  \frac{t^{-\alpha}}{\delta}\,dt\,dr+\frac{2\theta_s M_1}{1+\frac{a_0}{2}}
\int_{-\frac{a_0}{2}}^{\frac{a_0}{2}}\int_{a_0}^{+\infty} t^{-\alpha} \frac{1}{t - \frac{a_0}{2}}\,dt\,dr\\
=&\frac{2\theta_s a_0 M_1}{\delta\left(1+\frac{a_0}{2}\right)}\int_{\frac{a_0}{3}}^{a_0} t^{-\alpha}\,dt
+\frac{2\theta_s a_0M_1}{1+\frac{a_0}{2}}\int_{a_0}^{+\infty}  \frac{t^{-\alpha}}{t - \frac{a_0}{2}}\,dt,
\end{align*}
where again these integrals are finite. A similar computation can be done for the case $\tau  =0$ and $\nu = -$.

For $\tau = +$ and $\nu = 0$, we apply once more \eqref{APXRSEst} and obtain
\begin{align*}
&\int_{\Gamma_s^+}\int_{\Gamma_p^0}\left\|p^{-\alpha}\, dp_J\, S_R^{-1}(p,s)\, ds_I\, S_R^{-1}(s,T)\right\|\\
\leq&\int_{\frac{a_0}{2}}^{+\infty}\int_{-\frac{a_0}{3}}^{\frac{a_0}{3}}\left(\frac{a_0}{3}\right)^{-\alpha}\theta_p \left|S_R^{-1}\left(\frac{a_0}{3}e^{\frac{-3\theta_p}{a_0}Jt}, re^{-\theta_sI}\right)\right|\frac{M_1}{1+r}\,dt\,dr\\
\leq& 2\left(\frac{a_0}{3}\right)^{-\alpha}\theta_pM_1\int_{\frac{a_0}{2}}^{+\infty}\int_{-\frac{a_0}{3}}^{\frac{a_0}{3}} \frac{1}{\left|\frac{a_0}{3}e^{\frac{3\theta_p}{a_0}It}- re^{\theta_sI}\right|}\frac{1}{1+r}\,dt\,dr\\
=& 2\left(\frac{a_0}{3}\right)^{-\alpha}\theta_pM_1\int_{\frac{a_0}{2}}^{+\infty}\int_{-\frac{a_0}{3}}^{\frac{a_0}{3}} \frac{1}{\left|r- \frac{a_0}{3}e^{\left(\frac{3\theta_p}{a_0}t-\theta_s\right)I}\right|}\frac{1}{1+r}\,dt\,dr.
\end{align*}
Estimating the modulus of the denominator from below with the modulus of its real part, we obtain
\begin{align*}
&\int_{\Gamma_s^+}\int_{\Gamma_p^0}\left\|p^{-\alpha}\, dp_J\, S_R^{-1}(p,s)\, ds_I\, S_R^{-1}(s,T)\right\|\\
\leq& 2\left(\frac{a_0}{3}\right)^{-\alpha}\theta_pM_1\int_{\frac{a_0}{2}}^{+\infty}\int_{-\frac{a_0}{3}}^{\frac{a_0}{3}} \frac{1}{\left|r- \frac{a_0}{3}\cos\left(\frac{3\theta_p}{a_0}t-\theta_s\right)\right|}\frac{1}{1+r}\,dt\,dr\displaybreak[2]\\
\leq& 2\left(\frac{a_0}{3}\right)^{-\alpha}\theta_pM_1\int_{\frac{a_0}{2}}^{+\infty}\int_{-\frac{a_0}{3}}^{\frac{a_0}{3}} \frac{1}{r- \frac{a_0}{3}}\frac{1}{1+r}\,dt\,dr\\
 =& \frac{4a_0}{3}\left(\frac{a_0}{3}\right)^{-\alpha}\theta_pM_1\int_{\frac{a_0}{2}}^{+\infty} \frac{1}{r- \frac{a_0}{3}}\frac{1}{1+r}\,dr,
\end{align*}
and this last integral is finite. The estimate of the case $\tau =-$ and $\nu = 0$ can be done in a similar way.

Finally, the summand for $\tau = 0$ and $\nu=0$ consists of the integral of a continuous function over a bounded domain and is therefore finite.

Putting these pieces together, we obtain that the integrand in \eqref{APXBeforeFubini} is absolutely integrable, which allows us to apply Fubini's theorem in order to exchange the order of integration.

\section{Estimate for applying  Fubini's theorem in \eqref{anEquation}}\label{Fubini2}
We want to show that we can apply Fubini's theorem to
\begin{equation}\label{anEquationA}
\begin{split}
 &T^{-\alpha}\  T^{-\beta}\\
 =&-\frac{1}{(2\pi)^2 }\int_{ \partial ( G_s \cap \mathbb{C}_I) } s^{-\alpha}\ ds_I \int_{ \partial ( G_p \cap \mathbb{C}_I) }
S_L^{-1}(p,T)p(p^2-2s_0p+|s|^2)^{-1}dp_I\  p^{-\beta}
\\
&
+\frac{1}{(2\pi)^2 }\int_{ \partial ( G_s \cap \mathbb{C}_I) } s^{-\alpha}\ ds_I
\int_{ \partial ( G_p \cap \mathbb{C}_I) }\overline{s}S_L^{-1}(p,T)(p^2-2s_0p+|s|^2)^{-1}
 dp_I\  p^{-\beta}.
 \end{split}
\end{equation}
For $\tau\in\{s,p\}$ we therefore decompose $\partial (G_{\tau}\cap\C_I) = \Gamma_{\tau,-} \cup\, \Gamma{\tau,\circ} \cup\, \Gamma_{\tau,+}$ with
\begin{align*}
\Gamma_{\tau}^- &= \left\{-re^{I\theta_{\tau}}, r\in(-\infty,-a_{\tau}]\right\}\\
\Gamma_{\tau}^0 &= \left\{ a_\tau e^{-I\theta}, \theta\in(-\theta_\tau,\theta_\tau)\right\}\\
\Gamma_{\tau}^+ &= \left\{re^{-I\theta_{\tau}}, r\in[a_{\tau},+\infty)\right\}
\end{align*}
such that
\begin{equation}\label{JJOF}
\begin{split}& T^{-\alpha}\  T^{-\beta}\\
=&\sum_{u,v\in\{+,o,-\}}-\frac{1}{(2\pi)^2 }\int_{ \Gamma_{s}^{u} } s^{-\alpha}\ ds_I \int_{ \Gamma_{p}^{v} }
S_L^{-1}(p,T)p(p^2-2s_0p+|s|^2)^{-1}dp_I\  p^{-\beta}
\\
&
+\sum_{u,v\in\{+,o,-\}}\frac{1}{(2\pi)^2 }\int_{\Gamma_{s}^{u} } s^{-\alpha}\ ds_I
\int_{ \Gamma_{p}^{v}}\overline{s}S_L^{-1}(p,T)(p^2-2s_0p+|s|^2)^{-1}
 dp_I\  p^{-\beta}.
 \end{split}
\end{equation}
Since $p$ and $s$ commute, we have $(p^2 - 2s_0 p + |s|^2)^{-1} = (p-s)^{-1}(p-\overline{s})^{-1}$ and thus for $u=+$ and $v=+$
\begin{align*}
&\int_{ \Gamma_{s}^+ } s^{-\alpha}\ ds_I \int_{ \Gamma_{p}^+ }S_L^{-1}(p,T)p(p^2-2s_0p+|s|^2)^{-1}dp_I\  p^{-\beta}\\
=& \int_{a_s}^{+\infty} r^{-\alpha}e^{I\alpha\theta_s}e^{-I\theta_s}(-I) \int_{a_p}^{+\infty}S_L^{-1}(te^{-I\theta_p},T)\cdot\\
&\qquad\quad\cdot te^{-I\theta_p}\left(te^{-I\theta_p}-re^{-I\theta_s}\right)^{-1}\left(te^{-I\theta_p}-re^{I\theta_s}\right)^{-1} e^{-I\theta_p}(-I)t^{-\beta}e^{I\beta \theta_p}\,dt\,dr.
\end{align*}
Using the estimate $\| S_L^{-1}(s,T) \| \leq M_1/(1+ |s|)$ obtained from Lemma~\ref{StripCy} and setting $C = \sup_{t\in[0,+\infty)}M_1t/(1+t) < + \infty$, we find that the integral of the norm of the integrand is lower or equal to
\begin{align}
\notag&\int_{a_s}^{+\infty}\int_{a_p}^{+\infty} r^{-\alpha} \frac{M_1t}{1+t}\frac{1}{\left| t - re^{I(\theta_p-\theta_s)}\right|}\frac{1}{\left| t - r e^{I(\theta_p+\theta_s)}\right|} t^{-\beta}\,dt\,dr\\
\notag\leq & C \int_{a_s}^{+\infty}\frac{r^{-(\alpha+\beta)}}{r^2}\int_{a_p}^{+\infty}\frac{1}{\left|\frac{t}{r} - e^{I(\theta_p-\theta_s)}\right|}\frac{1}{\left| \frac{t}{r} -  e^{I(\theta_p+\theta_s)}\right|}\left(\frac{t}{r}\right)^{-\beta}\,dt\,dr\displaybreak[3]\\
\label{POTZ}\leq & C \int_{a_s}^{+\infty}\frac{r^{-(\alpha+\beta)}}{r}\int_{0}^{+\infty}\frac{1}{\left|\xi - e^{I(\theta_p-\theta_s)}\right|}\frac{1}{\left| \xi -  e^{I(\theta_p+\theta_s)}\right|}\xi^{-\beta}\,d\xi\,dr\\
\notag\leq & \frac{C}{(\alpha+\beta)a_s^{\alpha+\beta}} \int_{0}^{+\infty}\frac{1}{\left|\xi - e^{I(\theta_p-\theta_s)}\right|}\frac{1}{\left| \xi -  e^{I(\theta_p+\theta_s)}\right|}\xi^{-\beta}\,d\xi.
\end{align}
We set $\mu_0 = \max \cos(\theta_p\pm\theta_s) $  and $\mu_1 = \min\left|\sin(\theta_p \mp \theta_s)\right|$ and observe that $\mu_1 >0$ since we chose $\pi/2  < \theta_p <\theta_s < \pi$. Estimating $\left|\xi - e^{I(\theta_p\pm \theta_s)}\right|$ from below by the modulus of its real and imaginary part, we finally obtain that the above integral is bounded by
\begin{align*}
 \frac{C}{(\alpha+\beta)a_s^{\alpha+\beta}\mu_1} \int_{0}^{2}\xi^{-\beta}\,d\xi +
  \frac{C}{(\alpha+\beta)a_s^{\alpha+\beta}} \int_{2}^{+\infty}\frac{\xi^{-\beta}}{(\xi - \mu_0)^2}\,d\xi < \infty
\end{align*}
because $0<\beta<1$.

The second integral in \eqref{JJOF} with $u = +$ and $v=+$ is
\begin{align*}
&\int_{\Gamma_{s}^+}s^{-\alpha}\,ds_I  \int_{\Gamma_{p}^{+}}\overline{s}S_L^{-1}(p,T)(p^2-2s_0p + |s|^2)^{-1}\,dp_I\,p^{-\beta}\\
=& \int_{a_s}^{+\infty} r^{-\alpha}e^{I\alpha\theta_s}e^{-I\theta_s}(-I)\int_{a_p}^{+\infty}re^{I\theta_s}S_L^{-1}(te^{I\theta_p},T)\cdot\\
&\qquad\quad\cdot \left(te^{-I\theta_p} - re^{-I\theta_s}\right)^{-1} \left(te^{-I\theta_p} - re^{I\theta_s}\right)^{-1}e^{-I\theta_p}(-I)t^{-\beta}e^{I\beta\theta_p}\,dt\,dr.
\end{align*}
Using again the estimate $\| S_L^{-1}(s,T) \| \leq M_1/(1+ |s|)$ obtained from Lemma~\ref{StripCy} and setting $C = \sup_{t\in[0,+\infty)}M_1t/(1+t) < + \infty$, we can estimate the integral of the norm of the integrand by
\begin{align*}
& \int_{a_s}^{+\infty}\int_{a_p}^{+\infty}r^{1-\alpha}\frac{M_1}{1+t}\frac{1}{\left|t e^{-I(\theta_p-\theta_s)} - r \right|}\frac{1}{\left| te^{-I(\theta_p+\theta_s)} - r \right|}t^{-\beta}\,dt\,dr\\
\leq & \int_{a_p}^{+\infty}\frac{M_1t}{1+t}\frac{t^{-(\alpha+\beta)}}{t^2} \int_{a_s}^{+\infty}\left(\frac{r}{t}\right)^{1-\alpha}\frac{1}{\left|\frac{r}{t} -  e^{-I(\theta_p-\theta_s)} \right|}\frac{1}{\left|\frac{r}{t} - e^{-I(\theta_p+\theta_s)} \right|}\,dr\,dt\\
\leq & C \int_{a_p}^{+\infty}\frac{t^{-(\alpha+\beta)}}{t} \int_{0}^{+\infty}\xi^{1-\alpha}\frac{1}{\left|\xi -  e^{-I(\theta_p-\theta_s)} \right|}\frac{1}{\left|\xi - e^{-I(\theta_p+\theta_s)} \right|}\,d\xi\,dt.
\end{align*}
An integral of this form appeared in \eqref{POTZ} and we have already seen that it is finite. Similar estimates also hold for all terms in \eqref{JJOF} with $u,v\in\{+,-\}$.

For $u=+$ and $v = 0$, we have
\begin{align*}
&\int_{\Gamma_{s}^{+}}s^{-\alpha}\,ds_I\int_{\Gamma_{p}^{0}}S_L^{-1}(p,T)p(p^2-2s_0p+|s|^2)^{-1}\,dp_I\,p^{-\beta}\\
=& \int_{a_s}^{+\infty}r^{-\alpha}e^{I\alpha\theta_s} e^{-I\theta_s}(-I)\int_{-\theta_p}^{\theta_p}S_L^{-1}(a_pe^{-I\theta},T)a_pe^{-I\theta} \cdot\\
&\qquad\qquad\quad\cdot\left(a_pe^{-I\theta}-re^{-I\theta_s}\right)^{-1}\left(a_pe^{-I\theta}-re^{I\theta_s}\right)^{-1} a_pe^{-I\theta}(-I)^2a_p^{-\beta}e^{I\beta\theta}\,d\theta\,dr
\end{align*}
and, again using the estimate $\| S_L^{-1}(s,T) \| \leq M_1/(1+ |s|)$ obtained from Lemma~\ref{StripCy}, we find that the integral of the absolut value of the integrand is lower or equal to
\begin{align*}
&\int_{a_s}^{+\infty} r^{-\alpha}\int_{-\theta_p}^{\theta_p} \frac{M_1a_p^{2-\beta}}{1+a_{p}}\frac{1}{\left|r - a_pe^{-I(\theta-\theta_s)}\right|}\frac{1}{\left|r - a_pe^{-I(\theta+\theta_s)}\right|}\,d\theta\,dr.
\end{align*}
Since $\pi/2 < \theta_p<\theta_s<\pi$, the distance $\delta$ between the set
\begin{equation}\label{SomeSet}
\left\{a_pe^{-I(\theta+\theta_s)}, \theta\in[-\theta_p, \theta_p]\right\} \cup \left\{a_pe^{-I(\theta-\theta_s)}, \theta\in[-\theta_p, \theta_p]\right\}
\end{equation}
and the positive real axis is greater than zero and hence the above integral can be estimated by
\begin{align*}
\int_{a_s}^{a_s + 2a_p} r^{-\alpha}\int_{-\theta_p}^{\theta_p} \frac{M_1a_p^{2-\beta}}{1+a_{p}}\frac{1}{\delta^2}\,dr\,d\theta + \int_{a_s+2a_p}^{+\infty} \frac{r^{-\alpha}}{(r - a_p)^2}\int_{-\theta_p}^{\theta_p} \frac{M_1a_p^{2-\beta}}{1+a_{p}}\,d\theta\,dr <+\infty.
\end{align*}

For the second integral in \eqref{JJOF} with $u=+$ and $v=0$, we have
\begin{align*}
&\int_{\Gamma_s^+}s^{-\alpha}\,ds_I\int_{\Gamma_p^{0}}\overline{s}S_L^{-1}(p,T)(p^2-2s_0p+|s|^2)^{-1}\,dp_I\,p^{-\beta}\\
=& \int_{a_s}^{+\infty}r^{-\alpha}e^{I\alpha\theta_s} e^{-I\theta_s}(-I)\int_{-\theta_p}^{\theta_p}re^{I\theta_s}S_L^{-1}(a_pe^{-I\theta},T) \cdot\\
&\qquad\qquad\quad\cdot\left(a_pe^{-I\theta}-re^{-I\theta_s}\right)^{-1}\left(a_pe^{-I\theta}-re^{I\theta_s}\right)^{-1} a_pe^{-I\theta}(-I)^2a_p^{-\beta}e^{I\beta\theta}\,d\theta\,dr
\end{align*}
Using the estimate $\| S_L^{-1}(s,T) \| \leq M_1/(1+ |s|)$ obtained from Lemma~\ref{StripCy}, we can estimate the integral of the absolute value of the integrand by
\begin{align*}
&\int_{a_s}^{+\infty}r^{1-\alpha}\int_{-\theta_p}^{\theta_p}\frac{a_p^{1-\beta}M_1}{1+a_p} \frac{1}{\left| a_pe^{-I(\theta-\theta_s)}-r\right|}\frac{1}{\left|a_pe^{-I(\theta+\theta_s)}-r\right|}\,d\theta\,dr\\
\leq& \int_{a_s}^{a_s+2a_p}r^{1-\alpha}\int_{-\theta_p}^{\theta_p}\frac{a_p^{1-\beta}M_1}{1+a_p}\frac{1}{\delta^2}\,d\theta\,dr +
\int_{a_s+2a_p}^{+\infty}\frac{r^{1-\alpha}}{(r-a_p)^2}\int_{-\theta_p}^{\theta_p}\frac{a_p^{1-\beta}M_1}{1+a_p} \,d\theta\,dr < +\infty,
\end{align*}
where $\delta >0$ is again the distance between the set in \eqref{SomeSet} and the positive real axis. Similar estimates hold true if $u = -$ and $v = 0$.

If $u = 0$ and $v=+$, then
\begin{align*}
&\int_{ \Gamma_{s}^{0} } s^{-\alpha}\ ds_I \int_{ \Gamma_{p}^{+} }
S_L^{-1}(p,T)p(p^2-2s_0p+|s|^2)^{-1}dp_I\  p^{-\beta}\\
=&\int_{-\theta_s}^{\theta_s}a_{s}^{-\alpha}e^{I\alpha\theta}  a_{s}e^{-I\theta}(-I)^2\int_{a_p}^{+\infty}S_L^{-1}(te^{-I\theta_p},T)te^{-I\theta_s}\cdot\\
&\qquad\qquad \cdot\left(te^{-I\theta_p}-a_{s}e^{-I\theta}\right)^{-1}  \left(te^{-I\theta_p}-a_{s}e^{I\theta}\right)^{-1}e^{-I\theta_s}(-I)t^{-\beta}e^{I\beta\theta_p}\,dt\,d\theta.
\end{align*}
Once more the estimate $\| S_L^{-1}(s,T) \| \leq M_1/(1+ |s|)$ obtained from Lemma~\ref{StripCy} allows us to estimate the integral of the absolute value of the integrand by
\begin{align*}
&\int_{-\theta_s}^{\theta_s}a_{s}^{1-\alpha}\int_{a_p}^{+\infty}\frac{M_1t}{1+t}\frac{1}{\left|t-a_{s}e^{I(\theta_p-\theta)}\right|}  \frac{1}{\left|t-a_{s}e^{I(\theta_p+\theta)}\right|}t^{-\beta}\,dt\,d\theta\\
\leq & C\int_{-\theta_s}^{\theta_s}a_{s}^{1-\alpha}\int_{a_p}^{+\infty}  \frac{t^{-\beta}}{\left(t-a_{s}\right)^2}\,dt\,d\theta < +\infty,
\end{align*}
where again $C = \sup_ {t\in [0,+\infty)}M_1t/(1+t)<+\infty$ and the second inequality follows because $a_s < a_p$. For the second integral in \eqref{JJOF} with $u=0$ and $v=+$, we similarly have
\begin{align*}
&\int_{ \Gamma_{s}^{0} } s^{-\alpha}\ ds_I \int_{ \Gamma_{p}^{+} }\overline{s}S_L^{-1}(p,T)(p^2-2s_0p+|s|^2)^{-1}dp_I\  p^{-\beta}\\
=&\int_{-\theta_s}^{\theta_s}a_{s}^{-\alpha}e^{I\alpha\theta}  a_{s}e^{-I\theta}(-I)^2\int_{a_p}^{+\infty}a_{s}e^{I\theta_s}S_L^{-1}(te^{-I\theta_p},T)\cdot\\
&\qquad\qquad\quad \cdot\left(te^{-I\theta_p}-a_{s}e^{-I\theta}\right)^{-1}  \left(te^{-I\theta_p}-a_{s}e^{I\theta}\right)^{-1}e^{-I\theta_s}(-I)t^{-\beta}e^{I\beta\theta_p}\,dt\,d\theta.
\end{align*}
As above, we can estimate the integral of the absolute value of the integrand by
\begin{align*}
&\int_{-\theta_s}^{\theta_s}a_s^{2-\alpha}\int_{a_p}^{+\infty} \frac{M_1}{1+t}\frac{1}{\left|t-a_{s}e^{I(\theta_p-\theta)}\right|}  \frac{1}{\left|t-a_{s}e^{I(\theta_p+\theta)}\right|}t^{-\beta}\,dt\,d\theta\\
\leq& C\int_{-\theta_s}^{\theta_s}a_s^{2-\alpha}\int_{a_p}^{+\infty} \frac{t^{-(1+\beta)}}{\left(t-a_{s}\right)^2} \,dt\,d\theta < +\infty,
\end{align*}
where the last inequality follows again because $a_s < a_p$. Similar estimates hold for the case $u=0$ and $v=-$.

Finally, the integrals in \eqref{JJOF} with $u=0$ and $v=0$ are absolutely convergent since, in this case, we integrate a continuous and hence bounded function over a bounded domain.

Putting these pieces together, we obtain that we can actually apply Fubini's theorem in \eqref{anEquation} resp. \eqref{anEquationA} to exchange the order of integration.

\section*{Acknowledgments}
The authors are grateful for the careful reading and the detailed report of the anonymous referee that led to the development of the result in Section~\ref{HolSect} and increased the significance of the present paper.

%\bibliographystyle{plain}
%\bibliography{/users/faculty/math/dany/Travaux_courants/bib/all}
%\bibliography{all}

\vskip 1cm

\begin{thebibliography}{10}


\bibitem{adler}
 S. Adler, {\em Quaternionic Quantum Mechanics and Quaternionic Quantum Fields}, Oxford University Press,
1995.


\bibitem{MR2002b:47144}
D. Alpay,
{\em The {S}chur algorithm, reproducing kernel spaces and system
  theory},
 American Mathematical Society, Providence, RI, 2001,
Translated from the 1998 French original by Stephen S. Wilson.


 \bibitem{FUCGEN}
D.Alpay, F. Colombo, J. Gantner, D. P. Kimsey,
{\em Functions of the infinitesimal generator of a strongly continuous quaternionic group}, to appear in: Anal. Appl. (Singap.), \doi{10.1142/S021953051650007x}.

 \bibitem{acgs}
D.~{Alpay}, F.~{Colombo}, J. Gantner, I.~{Sabadini},
{\em A new resolvent equation for the $S$-functional calculus}, J. Geom. Anal, {\bf 25(3)} (2015), 1939--1968.

\bibitem{OPValPaper}
D.~{Alpay}, F.~{Colombo}, I.~{Lewkowicz}, I.~{Sabadini},
{\em Realizations of slice hyperholomorphic generalized contractive and positive functions},
Milan J. Math., {\bf 83 (1)} (2015), 91--144.

\bibitem{ack}
D.~{Alpay}, F.~{Colombo},  D. P. Kimsey,
{\em The spectral theorem for for quaternionic unbounded normal operators based on the $S$-spectrum}, J.~Math.~Phys. {\bf 57} (2016), 023503

 \bibitem{acks2}
D.~{Alpay}, F.~{Colombo},  D. P. Kimsey, I.~{Sabadini}.
{\em The spectral theorem for unitary operators based on the $S$-spectrum}, to appear in: Milan J. Math. , \doi{10.1007/s00032-015-0249-7}



 \bibitem{acs3}
D. {Alpay}, F. {Colombo}, I. {Sabadini},
{\em Krein-Langer factorization and related topics in the slice
  hyperholomorphic setting},  J. Geom. Anal., {\bf 24(2)} (2014),  843--872.


 \bibitem{perturbation}
D.~{Alpay}, F.~{Colombo}, I.~{Sabadini},
 { \em Perturbation of the generator of a quaternionic evolution operator},
 Anal. Appl. (Singap.), {\bf 13(4)} (2015), 347--370.


 \bibitem{acs2}
D. {Alpay}, F. {Colombo},  I. {Sabadini},
 {\em Pontryagin-De Branges-Rovnyak spaces of slice hyperholomorphic functions},
J. Anal. Math., {\bf 121} (2013), 87-125.

\bibitem{acs1}
D. {Alpay}, F. {Colombo},  I. {Sabadini},
 {\em Schur functions and their realizations in the slice hyperholomorphic
  setting},
{Integral Equations Operator Theory}, {\bf 72(2)}  (2012), 253--289.

\bibitem{ACSBOOK}
D. Alpay, F. Colombo, I. Sabadini,
{\em Slice Hyperholomorphic Schur Analysis}, Preprint  (2015),
Quaderni del Dipartimento di Matematica, Politecnico di Milano,
Codice	QDD209.



\bibitem{adrs}
D. Alpay, A. Dijksma, J. Rovnyak,  H. de Snoo,
 {\em {Schur} functions, operator colligations, and reproducing kernel
  {P}ontryagin spaces}, volume~96 of {\em Operator Theory: {A}dvances and
  {A}pplications}.
 Birkh{\" a}user Verlag, Basel, 1997.


\bibitem{Balakrishnan}
A. V. Balakrishnan,
{\em Fractional powers of closed operators and the semigroups generated by them},
Pacific J. Math., {\bf 10} (1960), 419--437.

\bibitem{BvN}
G. Birkhoff,  J. von Neumann, {\em The logic of quantum mechanics},  Ann. of Math. (2),
 {\bf 37(4)} (1936), 823--843.



\bibitem{CGTAYLOR}
F. Colombo, J. Gantner,
{\em On power series expansions of the S-resolvent operator and the Taylor formula}, Preprint (2015),
 \href{http://arxiv.org/abs/1501.07055}{\tt arXiv:1501.07055 [math.FA]}

\bibitem{JGA} F. Colombo, I. Sabadini,
{\em On some properties of the quaternionic functional calculus},
J. Geom. Anal., {\bf 19(3)}  (2009), 601--627.

\bibitem{CLOSED} F. Colombo, I. Sabadini,
{\em  On the  formulations of the quaternionic functional calculus},
 J. Geom. Phys., {\bf 60(10)} (2010), 1490--1508.


\bibitem{MR2803786}
F. Colombo, I. Sabadini,
{\em  The quaternionic evolution operator},
 Adv. Math., {\bf 227(5)} (2011), 1772--1805.

\bibitem{SigSSigR}
F. Colombo, I. Sabadini, {\em On some notions of spectra for quaternionic operators and for $n$-tuples of operators}, C. R. Math. Acad. Sci. Paris, Ser. I, {\bf 350 (7-8)} (2012), 399--402

\bibitem{SpecProp}
F. Colombo, I. Sabadini, {\em Some remarks on the $S$-spectrum}, Complex Var. Elliptic Equ., {\bf 58(1)} (2013), 1--6


\bibitem{MR2752913}
F. Colombo, I. Sabadini, D.~C. Struppa,
 {\em Noncommutative functional calculus. Theory and applications of slice hyperholomorphic functions},
 volume 289 of {\em Progress
  in Mathematics}.
 Birkh\"auser/Springer Basel AG, Basel, 2011.



\bibitem{ds}
N. Dunford, J. Schwartz, {\it Linear Operators, part I: General
Theory}, J. Wiley and Sons, New York, 1988.



\bibitem{ds2}
N. Dunford, J. Schwartz, {\it Linear Operators, part II: Spectral theory}, J. Wiley and Sons, New York, 1988.

\bibitem{12}
 G. Emch, {\em M\'ecanique quantique quaternionienne et relativit\'e restreinte},  I, Helv. Phys.
Acta,  {\bf 36} (1963), 739--769.

\bibitem{EngelNagel}
K. J. Engel, R. Nagel,
{\em One-parameter semigroups for linear evolution equations}.
 Graduate Texts in Mathematics, 194. Springer-Verlag, New York, 2000.

\bibitem{fp}
D. R. Farenick, B. A. F. Pidkowich, {\em The spectral theorem in quaternions},
Linear Algebra Appl.,
{\bf 371} (2003), 75--102.

\bibitem{14}
 D. Finkelstein, J. M. Jauch, S. Schiminovich ,  D. Speiser, {\em Foundations of
quaternion quantum mechanics}, J. Mathematical Phys., {\bf 3} (1962), 207--220.

\bibitem{MasterThesis}
J. Gantner, {\em Slice hyperholomorphic functions and the quaternionic functional calculus}, Master Thesis, 2014, Vienna University of Technology, \url{http://www.asc.tuwien.ac.at/~funkana/downloads_general/dipl_gantner.pdf}

\bibitem{DA}
J. Gantner, {\em A direct approach to the $S$-functional calculus for closed operators},  \href{http://arxiv.org/abs/1602.03910}{\tt arXiv:1602.03910 [math.SP]},  to appear in: J. Operator Theory

\bibitem{GSSb}
 G. Gentili, C. Stoppato, D. C.  Struppa, {\em Regular functions of a quaternionic variable}.
  Springer Monographs in Mathematics. Springer, Heidelberg, 2013.

\bibitem{GMP} R. Ghiloni, V. Moretti, A. Perotti,
{\em Continuous slice functional calculus in quaternionic Hilbert spaces},
Rev. Math. Phys.,  {\bf 25(4)} (2013), 1350006, 83 pp.



\bibitem{spectcomp} R. Ghiloni, V. Moretti, A. Perotti,
{\em Spectral properties of compact normal quaternionic operators},
 in: S. Bernstein, U. K\"{a}hler, I. Sabadini, F. Sommen (eds.): Hypercomplex Analysis: New Perspectives and Applications, Trends in Mathematics, Birkh\"{a}user, Basel,  (2014), 133--143.

\bibitem{GP}
R. Ghiloni, A. Perotti, {\em  Slice regular functions on real alternative algebras},
 Adv. Math., {\bf 226(2)} (2011), 1662--1691.


\bibitem{GR} R. Ghiloni, V. Recupero,
{\em Semigroups over real alternative *-algebras: generation theorems and spherical sectorial operators},
Trans. Amer. Math. Soc., {\bf 368(4)} (2016), 2645--2678.

\bibitem{GTABLEOFINT}
 I. S. Gradshteyn, I. M. Ryzhik, {\em  Table of integrals, series, and products}.
  Translated from the Russian. Seventh edition. Elsevier/Academic Press, Amsterdam, 2007.

\bibitem{LogBook}
K. G\"{u}rlebeck, K. Habetha, W. Spr\"{o}\ss ig, {\em Holomorphic functions in the plane and $n$-dimensional space}. Translated from the 2006 German original. Birkh\"{a}user Verlag, Basel, 2008.


\bibitem{Guzman3}
A. Guzman, {\em Fractional-power semigroups of growth $n$},  J. Funct. Anal., {\bf 30(2)} (1978), 223--237.

\bibitem{Guzman1}
A. Guzman, {\em Growth properties of semigroups generated by fractional powers of certain linear operators},
J. Funct. Anal., {\bf 23(4)} (1976), 331--352.

\bibitem{Guzman2}
A. Guzman, {\em Further study of growth of fractional-power semigroups}, J. Funct. Anal., {\bf 29(2)} (1978),  133--141.

\bibitem{Hille}
E. Hille, R. S. Phillips,  {\em Functional analysis and semi-groups},
 rev. ed., volume 31 of American Mathematical Society Colloquium Publications,
 American Mathematical Society, Providence, R. I., 1957

\bibitem{21}
L. P. Horwitz, L. C. Biedenharn, {\em Quaternion quantum mechanics: Second quantization
and gauge fields}, Annals of Physics, {\bf 157(2)} (1984), 432--488

\bibitem{Kantorovitz}
S. Kantorovitz,
{\em Topics in operator semigroups},
Progress in Mathematics, 281. Birkh\"{a}user Boston, Inc., Boston, MA, (2010).

\bibitem{Kato}
 T. Kato, {\em Note on fractional powers of linear operators}, Proc. Japan Acad., {\bf 36} (1960), 94--96.

\bibitem{Komatsu1}
H. Komatsu, {\em Fractional powers of operators}, Pacific J. Math., {\bf 19} (1966), 285--346.

\bibitem{Komatsu2}
H. Komatsu,
{\em  Fractional powers of operators. II. Interpolation spaces}, Pacific J. Math., {\bf 21} (1967), 89--111.

\bibitem{Komatsu3}
H. Komatsu,
{\em  Fractional powers of operators. III. Negative powers}, J. Math. Soc. Japan, {\bf 21} (1969), 205--220.

\bibitem{Lunardi}
A. Lunardi, {\em Analytic semigroups and optimal regularity in parabolic problems},
 volume 16 of Progress in Nonlinear Differential Equations and their Applications, Birkh\"{a}user, Basel, (1995).

\bibitem{Pazy}
A. Pazy, {\em  Semigroups of linear operators and applications to partial differential equations},
volume 44 of Applied Mathematical Sciences,
Springer, New York, 1983.

\bibitem{rudin}
W. Rudin,  {\em Real and complex Analysis}, Third Edition,  McGraw-Hill Book Co., New York, 1987.

\bibitem{sc} {C. S. Sharma,  T. J. Coulson},
{\em Spectral theory for unitary operators on a quaternionic
              {H}ilbert space}, J. Math. Phys., {\bf 28(9)} (1987), 1941--1946.

\bibitem{Souc} {G. A. G. Soukhomlinoff},
{\em \"{U}ber Fortsetzung von linearen {F}unktionalen in linearen komplexen {R}\"aumen und linearen {Q}uaternionr\"aumen}, Rec. Math. [Mat. Sbornik] N.S., {3(45)-2}(1938), 353--358

\bibitem{Viswanath}
K. Viswanath, {\em Normal operators on quaternionic Hilbert spaces}, Trans. Amer. Math. Soc.,
{\bf 162} (1971), 337--350.

\bibitem{Watanabe}
J. Watanabe,
{\em On some properties of fractional powers of linear operators},
Proc. Japan Acad., {\bf 37} (1961), 273--275.

\bibitem{Yosida}
K, Yosida, {\em Fractional powers of infinitesimal generators and the analyticity of the semi-groups generated by them},
 Proc. Japan Acad., {\bf 36} (1960), 86--89.




\end{thebibliography}

\end{document}